\documentclass{m2an}
\allowdisplaybreaks
\usepackage{graphicx}
\usepackage[caption=false]{subfig}
\usepackage{amsmath}
\usepackage{amssymb,amsfonts}
\usepackage{amsthm}
\usepackage{bm}
\usepackage{mathrsfs}
\usepackage{amssymb}
\usepackage{multirow}
\usepackage{color}

\newcommand{\norm}[1]{\left\lVert#1\right\rVert}
\newcommand{\innOm}[2]{(#1,#2)_\Omega}
\newcommand{\innGa}[2]{\langle #1,#2\rangle_\Gamma}
\newcommand{\duaOm}[2]{[#1,#2]_\Omega}
\newcommand{\duaGa}[2]{[#1,#2]_\Gamma}

\usepackage[margin=1in]{geometry}

\usepackage{algorithm}
\usepackage{algorithmic}

\numberwithin{equation}{section}
\usepackage{multirow}
\usepackage{makecell}
\usepackage{booktabs}
\theoremstyle{definition}
\newtheorem{theorem}{Theorem}[section]

\newtheorem{lemma}{Lemma}[section]
\newtheorem{corollary}{Corollary}[section]
\newtheorem{definition}{Definition}[section]
\newtheorem{remark}{Remark}[section]
\newtheorem{example}{Example}[section]

\usepackage{cite}
\usepackage{hyperref}
\usepackage[nameinlink]{cleveref}

\usepackage{fancyhdr}
\begin{document}

\title{Analysis of a hybridizable discontinuous Galerkin scheme for the tangential control of the Stokes system}

\author{Wei Gong}\address{NCMIS \& LSEC, Institute of Computational Mathematics, Academy of Mathematics and Systems Science, Chinese Academy of Sciences, Beijing 100190, China. \email{wgong@lsec.cc.ac.cn}. W.~Gong was supported by the National Natural Science Foundation of China under grants 11671391 and 91530204, and the National Key Basic Research Program (2018YFB0704304).}

\author{Weiwei Hu}\address{Department of Mathematics, Univeristy of Georgia, Athens, GA. \email{Weiwei.Hu@uga.edu}. W. Hu was partially supported by the NSF grant DMS-1813570.}

\author{Mariano Mateos}\address{Dpto. de Matem\'aticas. Universidad de Oviedo,
Campus de Gij\'on, Spain. \email{mmateos@uniovi.es}. M.~Mateos was supported by the Spanish Ministerio de Econom\'{\i}ay Competitividad under projects MTM2014-57531-P and MTM2017-83185-P.}

\author{John~R.~Singler}\address{Department of Mathematics
 	and Statistics, Missouri University of Science and Technology,
 	Rolla, MO, USA. \email{singlerj@mst.edu}.  J.~Singler was supported in part by National Science Foundation grant DMS-1217122.}

 \author{Yangwen Zhang}\address{Department of Mathematics Science, University of Delaware, Newark, DE, USA. \email{ywzhangf@udel.edu}. Yangwen Zhang was supported by the US National Science Foundation (NSF) under
 	grants DMS-1619904 and DMS-1217122.}

%
%
\begin{abstract}
	We consider an unconstrained tangential Dirichlet boundary control problem for the Stokes equations with an $ L^2 $ penalty on the boundary control.  The contribution of this paper is twofold.  First, we obtain well-posedness and regularity results for the tangential Dirichlet control problem on a convex polygonal domain.  The analysis contains new features not found in similar Dirichlet control problems for the Poisson equation; an interesting result is that the optimal control has higher local regularity on the individual edges of the domain compared to the global regularity on the entire boundary.  Second, we propose and analyze a hybridizable discontinuous Galerkin (HDG) method to approximate the solution.  For convex polygonal domains, our theoretical convergence rate for the control is optimal with respect to the global regularity on the entire boundary.  We present numerical experiments to demonstrate the performance of the HDG method.
 \end{abstract}
%
\maketitle

	\section{Introduction}
Let $\Omega\subset \mathbb{R}^2$ be a convex domain with a polygonal boundary $\Gamma = \partial \Omega$.  For a given target state $ \bm y_d $, we consider the following unconstrained Dirichlet boundary control problem for the Stokes equations:
\begin{align}\label{Ori_problem1}
\min\limits_{\bm u\in \bm U} J(\bm u),  \quad  J(\bm u) := \frac{1}{2}\|\bm y_{\bm u}-\bm y_{d}\|^2_{\bm L^{2}(\Omega)}+\frac{\gamma}{2}\| \bm u\|^2_{\bm U},
\end{align}
where $\bm y_{\bm u}$ is the solution, in a sense to be defined later, of
\begin{equation}\label{Ori_problem2}
\begin{split}
-\Delta\bm y+\nabla p &=\bm f \quad  \text{in}\ \Omega,\\
\nabla\cdot\bm y&=0 \quad\  \text{in}\ \Omega,\\
\bm y&=\bm u \quad\  \text{on}\  \Gamma,\\
\int_{\Omega} p&=0,
\end{split}
\end{equation}
$\bm U \subset \bm L^2(\Gamma)$ is the control space and $\gamma>0$ is a fixed constant.

Control of fluid flows modeled by the Stokes or Navier-Stokes equations is an important and active area of interest.  After the pioneering works by Glowinski and Lions \cite{MR1352473} and Gunzburger \cite{MR1946726,MR1632420,MR1613873,MR1759904,MR1720145,MR1135991,MR1145711}, many important developments have been made both theoretically and computationally in the past decades. For an extensive body of literature devoted to this subject we refer to, e.g., \cite{MR1613897,MR1871460,MR2326293,MR2974748,MR3712009,CasasChrysafinos15,BallarinManzoniRozzaSalsa14,BochevGunzburger06,FouresCaulfieldSchmid14,YEfendiev_TYHou_book_1,YanKeyes15,Pearson15,Ravindran17,YangCai17,MR2591040,MR2338434} and the references therein.  Despite the large amount of existing work on numerical methods for fluid flow control problems,  we are not aware of any contributions to the analysis and approximation of the tangential Stokes Dirichlet boundary control problem.  Work on this problem is an essential step towards the analysis and approximation of similar Dirichlet boundary control problems for the Navier-Stokes equations and other fluid flow models.


In this work, we focus on the case where the control acts tangentially along the boundary through a Dirichlet boundary condition.  This scenario has broad applications to optimal mixing and heat transfer problems.  Omari and Guer in \cite{el2010alternate} conducted a numerical study of  the effect of wall rotation on the enhancement
of heat transport in the whole fluid domain. Gouillart et al.\ in \cite{GDDRT, GKDDRT,GTD, TGD} studied in detail this crucial effect
of moving wall on the mixing efficiency for the homogenization of concentration in
a 2D closed flow environment.  These problems naturally lead to the study of tangential  boundary control and optimization of fluid flows. Recently, Hu and Wu in \cite{hu2017boundary, hu2017an, hu2017enhancement, HuWu2018boundaryNS} provided  rigorous  mathematical  approaches for optimal mixing and heat transfer via an active control of Stokes and Navier-Stokes flows through  Navier slip boundary conditions.  Other tangential boundary control problems for fluid flows have been considered by  Barbu, Lasiecka and Triggiani \cite{MR2218543,MR2215059,MR3348933,MR3295779} and Osses \cite{MR1871454}.  However, the authors are not aware of any existing work on approximation and numerical analysis for these problems.

%

Discontinuous Galerkin (DG) methods are widely used for fluid flow problems since they can capture shocks and large gradients in solutions. However, most existing DG methods are commonly considered to have a major drawback: the memory requirement and computational cost of DG methods are typically much larger than the standard finite element method.

Hybridizable discontinuous Galerkin (HDG) methods were proposed by Cockburn et al.\ in \cite{MR2485455} as an improvement of traditional DG methods. The HDG methods are based on a mixed formulation and utilize a {\it numerical flux} and a {\it numerical trace} to approximate the flux and the trace of the solution.  The approximate flux and solution variables  can be  eliminated element-by-element. This process leads to a global equation for the approximate boundary traces only. As a result,  HDG methods have significantly less globally coupled unknowns, memory requirement, and computational cost compared to other DG methods.  Furthermore, HDG methods have been successfully applied to flow problems \cite{MR3626531,MR3556409,MR2772094,MR3194122,MR2753354,MR3432358}, distributed optimal control problems \cite{MR3508834,HuShenSinglerZhangZheng_HDG_Dirichlet_control2,HuShenSinglerZhangZheng_HDG_Dirichlet_control3}, and Dirichlet boundary control problems \cite{HuShenSinglerZhangZheng_HDG_Dirichlet_control1,HuMateosSinglerZhangZhang1,MR3831243}.

For the Stokes tangential Dirichlet boundary control problem considered here, the Dirichlet boundary data $ \bm u \in \bm L^2(\Gamma) $ takes the form $ \bm u = u \bm \tau $, where $ u $ is the control and $ \bm \tau $ is the unit tangential vector to the boundary.  Formally, the optimal control $u\in L^2(\Gamma)$ and the optimal state $\bm y \in  \bm L^2(\Omega)$ minimizing the cost functional satisfy the optimality system
\begin{subequations}\label{optimality_system}
	\begin{align}
	-\Delta \bm y + \nabla p &= \bm f \qquad \qquad \textup{in} \ \Omega,\label{optimality_system1}\\
	\nabla\cdot\bm y &= 0\qquad \qquad \ \textup{in} \ \Omega,\label{optimality_system2}\\
	\bm y &= u\bm \tau \qquad \quad\  \ \textup{on} \ \Gamma,\label{optimality_system3}\\
	-\Delta \bm z -\nabla q &= \bm y - \bm y_d \quad \quad \textup{in} \ \Omega,\label{optimality_system4}\\
	\nabla\cdot\bm z &= 0\qquad \qquad \   \textup{in} \ \Omega,\label{optimality_system5}\\
	\bm z &= 0  \qquad \quad\ \ \   \ \textup{on} \ \Gamma,\label{optimality_system6}\\
	\partial_{\bm n}  \bm z &= \gamma  u\bm\tau \qquad\quad \textup{on} \  \Gamma.\label{optimality_system7}
	\end{align}
\end{subequations}


We use an HDG method to approximate the solution of a mixed formulation of this optimality system.  To do this, we first analyze the control problem in \Cref{Analysis_D_S}.  We give precise meaning to the state equation \eqref{optimality_system1} for Dirichlet boundary data in $ \bm L^2(\Gamma)$, and prove well-posedness and regularity results for the optimality system \eqref{optimality_system}.  The theoretical results for this problem share some similarities to results for Dirichlet boundary control of the Poisson equation on a 2D convex polygonal domain \cite{MR3432846}; however, there are new components to the analysis due to the mixed formulation and the regularity results for Stokes equations on polygonal domains \cite{MR977489}.  An interesting feature of our theoretical results is that the optimal control has higher local regularity (on each boundary edge) than global regularity (on the entire boundary $ \Gamma $).  This higher local regularity for the optimal control is not present for Dirichlet boundary control of the Poisson equation; furthermore, as we discuss below, this phenomenon may have an effect on the convergence rates of the approximate solution.

For the HDG method, we use polynomials of degree $k + 1$ to approximate the velocity $\bm y$ and dual velocity $\bm z$, and polynomials of degree $k \ge 0$ for the fluxes $\mathbb L = \nabla \bm y$ and $\mathbb G = \nabla \bm z$, pressure $p$ and dual pressure $q$. Moreover, we also use polynomials of degree $k$ to approximate the numerical trace of the velocity and dual velocity on the edges of the spatial mesh, which are the only globally coupled unknowns.  We describe the HDG method in \Cref{sec:HDG_formulation} and its implementation can be found in the arXiv preprint of this paper \cite{GongHuMateosSinglerZhang1}.

In \Cref{sec:analysis}, we prove convergence results for the HDG method.  Under certain assumptions on the largest angle of the convex polygonal domain and the smoothness of the desired state $ \bm y_d $, we prove the control converges at a superlinear rate.    Similar superlinear convergence results for Dirichlet boundary control of the Poisson equation have been obtained in \cite{ApelMateosPfeffererRosch17,MR2272157,MR2806572,HuShenSinglerZhangZheng_HDG_Dirichlet_control1,HuMateosSinglerZhangZhang1,MR3831243}.  To give a specific example of our results, for a rectangular domain, $\bm y_d \in \bm H^2(\Omega)$, and $ k = 1 $, we obtain the following a priori error bounds for the velocity $\bm y$, adjoint velocity $\bm z$, their fluxes $\mathbb L$ and $\mathbb G$, pressure $p$ and dual pressure $q$ and  the optimal control $u$:
\begin{align*}
&\|\bm y - \bm  y_h\|_{0,\Omega} =  O(h^{3/2-\varepsilon}), &  &\|\mathbb L - \mathbb L_h\|_{0,\Omega} =  O(h^{1-\varepsilon}),\\
&\|\bm z - \bm  z_h\|_{0,\Omega} =  O(h^{3/2-\varepsilon}), &  &\|\mathbb G - \mathbb G_h\|_{0,\Omega} =  O(h^{3/2-\varepsilon}),\\
&\|p - p_h\|_{0,\Omega} =O(h^{1-\varepsilon}), &  &\|q - q_h\|_{0,\Omega} =O(h^{3/2-\varepsilon}),
\end{align*}
and
\begin{align*}
\|u-  u_h\|_{0,\Gamma} =  O(h^{3/2-\varepsilon}),
\end{align*}
for any $\varepsilon>0$. The rate of convergence for the control $u$ is optimal in the sense of the maximal global regularity of the control $u\in H^{3/2-\varepsilon}(\Gamma)$.  However, the numerical results presented in \Cref{sec:numerics} show higher convergence rates than the rates predicted by our numerical analysis; we discuss this phenomenon in more detail in \Cref{sec:numerics}.  The numerical convergence rates observed here are different than typical numerical results for Dirichlet boundary control of the Poisson equation.

We emphasize that the HDG method in this work is usually considered to be a \emph{superconvergent} method.  Specifically, if polynomials of degree $k\ge 1$ are used for the numerical trace and the solution of the PDEs is smooth enough, then $O(h^{k+2})$ error estimates can be obtained for the state variable; see, e.g., \cite{MR3440284,MR3556409,HuShenSinglerZhangZheng_HDG_Dirichlet_control3}. Hence, from the viewpoint of globally coupled degrees of freedom, this method achieves superconvergence for the scalar variable. For Dirichlet boundary control problems, to obtain the \emph{superlinear} convergence rate, one usually needs a superconvergence mesh or higher order elements for the standard finite element method, see, e.g., \cite{ApelMateosPfeffererRosch17,MR2558321}.  However, the HDG method considered here achieves the superlinear convergence rate without any special considerations.

\section{ Analysis of the Tangential Dirichlet Control Problem}
\label{Analysis_D_S}
To begin, we set notation and prove some fundamental results concerning the optimality system for the control problem.  In this section, we assume $ \Omega $ is a convex polygonal domain and the forcing $ \bm f $ in the Stokes equations \eqref{Ori_problem2} is equal to zero; if the forcing is nonzero, then it can be eliminated using the technique in \cite[p.\ 3623]{MR3432846}.

Throughout the paper, we use the standard notation  $H^{m}(\Omega)$ to denote the Sobolev space with norm $\|\cdot\|_{m,\Omega}$ and seminorm $|\cdot|_{m,\Omega}$. Set $\mathbb H^m(\Omega) = [H^{m}(\Omega)]^{2\times 2}$, $\bm H^{m}(\Omega) = [H^{m}(\Omega)]^2$ and $\bm H_0^1(\Omega) =\{\bm v\in \bm H^1(\Omega) : \bm v = 0 \ \textup{on} \ \Gamma \}$. We denote the $L^2$-inner products on $ \mathbb L^2(\Omega) $,   $\bm L^2(\Omega)$, $ L^2(\Omega)$ and $ \bm L^2(\Gamma)$ by
\begin{align*}
(\mathbb L,\mathbb G)_{\Omega} &= \sum_{i,j=1}^2 \int_{\Omega}  L_{ij} G_{ij}, \qquad (\bm y,\bm z)_{\Omega} = \sum_{j=1}^2 \int_{\Omega}  y_j z_j,\\
(p,q)_{\Omega} &=  \int_{\Omega} pq,\qquad \qquad\qquad\langle \bm y,\bm z\rangle_{\Gamma} = \sum_{j=1}^2\int_{\Gamma}  y_j  z_j.
\end{align*}
Define the space $\mathbb H(\text{div};\Omega)$ as
\begin{align*}
\mathbb H(\text{div},\Omega) = \{\mathbb K \in \mathbb L^2(\Omega) : \nabla\cdot \mathbb K \in \bm L^2(\Omega)\}.
\end{align*}
Also, we define $ L_0^2(\Omega)$ as
\begin{align*}
L_0^2(\Omega) &= \left\{p \in L^2(\Omega) : (p,1)_{\Omega} = 0\right\}.
\end{align*}
Let $\innGa{\cdot}{\cdot}$ denote the inner product in $L^2(\Gamma)$ and let $\duaGa{\cdot}{\cdot}$ denote the duality product between $H^{-s}(\Gamma)$ and $H^{s}(\Gamma)$ for $0\leq s <3/2$,
where $H^{s}(\Gamma)$ denotes the space of traces of $H^{s+1/2}(\Omega)$ for $0<s<3/2$. (For $1/2\leq s<3/2$ it is the subspace of $\Pi_{i=1}^m H^{s}(\Gamma_i)$ satisfying certain compatibility conditions on the corners; see \cite[Theorem 1.5.2.8]{MR775683}. For $s=3/2$, this definition would lead to ambiguities.)
Following \cite[Section 2.1]{MR2371113} we introduce  the spaces
\begin{align*}
\bm V^s(\Omega) &= \{\bm y\in \bm H^s(\Omega) : \nabla \cdot \bm y = 0, \  \duaGa{\bm y\cdot\bm n}{1}=0\},\mbox{ for }s\geq 0,\\
\bm V^s_0(\Omega) &= \{\bm y\in \bm H^s(\Omega) : \nabla \cdot  \bm y = 0, \  \bm y=0\mbox{ on }\Gamma\},\mbox{ for }s>1/2,\\
\bm V^s(\Gamma) &=  \{\bm u\in \bm H^s(\Gamma) : \innGa{\bm u\cdot\bm n}{1} = 0\},\mbox{ for }0\leq s<3/2.
\end{align*}
For $-3/2<s<0$, $\bm V^{s}(\Gamma)$ is the dual space of $\bm V^{-s}(\Gamma)$. For $s<-1/2$, $\bm V^{s}(\Omega)$ is the dual space of $\bm V^{-s}_0(\Omega)$ and for $-1/2\leq s<0$, $\bm V^{s}(\Omega)$ is the dual space of $\bm V^{-s}(\Omega)$.

Consider a target state $\bm y_d\in \bm H$, where $\bm H\hookrightarrow \bm V^0(\Omega)$ is a function space
that will be specified later, and a Tykhonov regularization parameter $\gamma>0$. Consider also a space $\bm U\hookrightarrow \bm V^0(\Gamma)$.
We are interested in the optimal control problem
\begin{equation}\label{P}\tag{P}
\min_{\bm u\in \bm U} J(\bm u)=\frac{1}{2}\|\bm y_{\bm u}-\bm y_d\|^2_{\bm H} + \frac{\gamma}{2}\|\bm u\|_{\bm U}^2,
\end{equation}
where $\bm y_{\bm u}\in \bm  V^0(\Omega)$ is the unique solution in the transposition sense of the Stokes system (see \Cref{D2.1} below)
\begin{equation}\label{StateEquation}
\begin{split}
-\Delta\bm y+\nabla p &=0 \quad  \text{in}\ \Omega,\\
\nabla\cdot\bm y&=0 \quad  \text{in}\ \Omega,\\
\bm y&=\bm u \quad  \text{on}\  \Gamma,\\
(p,1)_{\Omega}&=0.
\end{split}
\end{equation}



Different choices of the spaces $\bm H$ and $\bm U$ appear in the related literature for Dirichlet control of Stokes and Navier-Stokes equations. In the early reference \cite{MR1135991}, $\bm H=\bm L^4(\Omega)$ and $\bm U = \bm V^1(\Gamma)$. The natural space for the controls to obtain a variational solution of
the state equation \eqref{StateEquation} is $\bm  V^{1/2}(\Gamma)$. This is the choice in \cite{MR2154110}. In that work, nevertheless, the Tykhonov regularization is done in the norm of $\bm L^2(\Gamma)$. To prove existence of solution, the tracking is done in the space $\bm H = \bm V^1(\Omega)$. In the reference \cite{MR2591040},
the authors work in a smooth domain with $\bm H= \bm V^0(\Omega)$ and $\bm U=\bm V^0(\Gamma)$. This choice involves a harder analysis, but leads to an optimality system easier to handle. In polygonal domains, this approach leads to optimal controls that are \emph{discontinuous} at the corners.

We assume throughout this work that the tracking term for the state is measured in the $\bm L^2(\Omega)$ norm.  We investigate the case $\bm U = \{u\bm\tau : u\in L^2(\Gamma)\}$, which corresponds to tangential boundary control; see \cite{MR2218543, MR2215059}. We first precisely define the concept of solution for Dirichlet data in $\bm V^0(\Gamma)$, prove precise regularity results, and use them to introduce a mixed formulation of the problem adequate for HDG methods.

\subsection{Regularity results}\label{S2}

The definition of very weak solution for data in $\bm V^0(\Gamma)$ was introduced in \cite[Appendix A]{MR884813} and is valid in convex polygonal domains; see also \cite{MR1739399} and \cite[Definition 2.1]{MR2591040} for a similar definition for the Navier-Stokes equations and smooth domains. It is worthwhile to mention that the first numerical analysis work  for a  Dirichlet problem with irregular
boundary datum by  using a very weak formulation was given in \cite{MR2084239}. Also in smooth domains, very
weak solutions can be defined for data in $\bm V^{-1/2}(\Gamma)$; see \cite[Appendix A]{MR2371113}. We will  prove that the optimal regularity $\bm V^{s+1/2}(\Omega)$ expected for the solution can be achieved. In \cite{MR1663460}, only suboptimal regularity $\bm V^{s+1/2-\varepsilon}(\Omega)$ for all $\varepsilon>0$ is proved. We obtain a result comparable to the one given in \cite[Appendix A]{MR2371113} for smooth domains.

Let $\omega$ denote the greatest interior angle of $\Gamma$.  Following \cite[Theorem 5.5]{MR977489}, we know there exists a number $\xi = \xi(\omega) \in(0.5,4]$ that gives the maximal $\bm H^s(\Omega)$ regularity for the problem \eqref{StokesRegularity}.
This means for very smooth $\bm f$ and $h$ satisfying the compatibility condition $\innOm{h}{1}=0$ we can only expect that the variational solution $(\bm z_{\bm f,h},q_{\bm f,h})\in \bm H^1_0(\Omega) \times L^2_0(\Omega)$ of the \emph{compressible} Stokes problem
%
\begin{equation}\label{StokesRegularity}
\begin{split}
-\Delta\bm z+\nabla q &=\bm f \quad  \text{in}\ \Omega,\\
\nabla\cdot\bm z&=h \quad  \text{in}\ \Omega,\\
\bm z&=0 \quad  \text{on}\  \Gamma,\\
(q,1)_{\Omega}&=0,
\end{split}
\end{equation}
satisfies $\bm z\in \bm H^{3/2+s}(\Omega)$ and $q\in H^{1/2+s}(\Omega)$ for $s<\xi-1/2$. This singular exponent $\xi$ is the smallest real part of all of the roots $ \lambda $ of the equation
\begin{align}\label{singular_ex}
\frac{ \sin^2(\lambda\omega)-\lambda^2\sin^2\omega }{ \lambda^2 (\lambda - 1) }=0,
\end{align}
and satisfies that $\omega\mapsto\xi$ is strictly decreasing, $\xi >\pi/\omega$ if $\omega<\pi$, and $0.5<\xi<\pi/\omega$ if $\omega > \pi$.

Let us denote
\[s^*  =\min\{\xi-1/2,1/2\}.\]
If $\bm f\in \bm L^2(\Omega)$ and $h\in H^1(\Omega)$ such that $\innOm{h}{1}=0$, \cite [Theorem 5.5(a)]{MR977489} states that for all \[0<s<s^*=\min\{\xi-1/2,1/2\}\] the solution of \eqref{StokesRegularity} satisfies $\bm z_{\bm f,h}\in \bm H^{3/2+s}(\Omega)$ and $q_{\bm f,h}\in H^{1/2+s}(\Omega)$.
Moreover, we have that
\begin{equation}\label{Continuity}
\|\bm z_{\bm f,h}\|_{\bm H^{3/2+s}(\Omega)}+\|q_{\bm f,h}\|_{H^{1/2+s}(\Omega)/\mathbb{R}} \leq C \big(\|\bm f\|_{\bm H^{s-1/2}(\Omega)} + \|h\|_{H^{s+1/2}(\Omega)/\mathbb R}\big).
\end{equation}
Notice that although the pressure is uniquely determined as a function with the condition $\innOm{q}{1}=0$, the norm must be taken modulo constant functions.
\begin{remark}
	Another remarkable fact is that this result holds for $s<1/2$. This means, in particular, that in convex domains one cannot expect in general to have $\bm H^2(\Omega)$ regularity of $\bm z$.
	To obtain this $\bm H^2(\Omega)$ regularity, an additional condition must be made on the divergence of $\bm z$. If, e.g., $h\in H^1_0(\Omega)$, $\innOm{h}{1}=0$, then it follows from \cite[Theorem 5.5(c)]{MR977489} or the early reference \cite{MR0404849} that the result also holds for $s=s^*$. In particular,  under these assumptions, it is obtained that  $\bm z\in\bm H^2(\Omega)$ in convex domains.
	This fact was used both in \cite{MR884813} and in \cite{MR1663460} to define very weak solutions in convex polygonal domains using $h\in H^1_0(\Omega)$ as a test function. Although the approach works to define the transposition solution, it leads only to suboptimal regularity results for the solution of the Dirichlet problem.
\end{remark}

For later reference, we state the regularity result for the case $h\equiv 0$.
\begin{theorem}\label{Dauge55b}{\cite[Theorem 5.5(b)]{MR977489}}
	Suppose $\bm f \in \bm H^{t-1}(\Omega)$ for some $1\leq t<\xi$. Then, the unique solution of the incompressible Stokes problem
	\begin{equation}\label{StokesRegularityIncompressible}
	\begin{split}
	-\Delta\bm z+\nabla q &=\bm f \quad  \text{in}\ \Omega,\\
	\nabla\cdot\bm z&=0 \quad  \text{in}\ \Omega,\\
	\bm z&=0 \quad  \text{on}\  \Gamma,\\
	(q,1)_{\Omega}&=0.
	\end{split}
	\end{equation}
	satisfies $\bm z\in\bm H^{t+1}(\Omega)$, $q\in H^t(\Omega)$ and
	\[
	\|\bm z\|_{\bm H^{1+t}(\Omega)}+\|q\|_{H^{t}(\Omega)/\mathbb{R}} \leq C \|\bm f\|_{\bm H^{t-1}(\Omega)}.
	\]
\end{theorem}

Although we will pose our control problem for data in $\bm u\in\bm V^0(\Gamma)$, the precise regularity results for the state equation will follow by interpolation; therefore we need a definition of very weak solution for data in $\bm u\in\bm V^{-s}(\Gamma)$ for $0<s<s^*$. The elements of this space do not always satisfy a condition analogous to $\innGa{\bm u\cdot \bm n}{1}=0$, i.e., we may have $\duaGa{\bm u}{\bm n}\neq 0$, and it is necessary to take this into account to define a solution in the transposition sense. Following \cite[Eq.\ (2.2)]{MR2371113}, we define for $(\bm z,q)\in \bm H^{3/2+s}(\Omega)\times H^{1/2+s}(\Omega)$, $s>0$, the constant
\begin{equation}
\label{Ray2.2}c(\bm z,q) = \frac{1}{|\Gamma|}\innGa{q-\partial_{\bm n} \bm z\cdot \bm n}{1}.
\end{equation}
This constant satisfies the relation
\[\|\partial_{\bm n} \bm z-q \bm n\|_{L^2(\Gamma)/\mathbb R} = \|\partial_{\bm n} \bm z-q \bm n+c(\bm z,q)\bm n\|_{L^2(\Gamma)}.\]
Using this fact, usual trace theory and \eqref{Continuity}, we have that for $0\leq s<1/2$
\begin{equation}\label{NormalTraceContinuity}
\|\partial_{\bm n} \bm z_{\bm f, h}-q_{\bm f, h}\cdot \bm n+c(\bm z_{\bm  f, h},q_{\bm f, h})\bm n\|_{H^s(\Gamma)}\leq C \big(\|\bm f\|_{\bm H^{s-1/2}(\Omega)} + \|h\|_{H^{s+1/2}(\Omega)/\mathbb R}\big).
\end{equation}
The following definition makes sense:

\begin{definition}\label{D2.1}Consider $0\leq s<s^*$ and $\bm u\in \bm V^{-s}(\Gamma)$.  We say $\bm y_{\bm u}\in \bm V^0(\Omega)$, $p_{\bm u} \in \left(H^{1}(\Omega)/\mathbb{R}\right)'$ is a solution in the transposition sense of \eqref{StateEquation} if $ (y_{\bm u},p_{\bm u}) $ satisfy
	\begin{equation}\label{veryWeakForm}
	\innOm{\bm y}{\bm f} - \duaOm{p}{h} = \duaGa{\bm u}{-\partial_{\bm n}\bm z_{\bm f,h}+q_{\bm f,h}\bm n+c(\bm z_{\bm f,h},q_{\bm f,h})\bm n},
	\end{equation}
	for all $\bm f\in \bm L^2(\Omega)$  and $h\in H^1(\Omega)/\mathbb{R}$ such that $\innOm{h}{1}=0$,
	where $(\bm z_{\bm f,h},q_{\bm f,h})\in \bm H^1_0(\Omega) \times L^2_0(\Omega)$ is the unique solution of \eqref{StokesRegularity} and $c(\bm z_{\bm f,h},q_{\bm f,h})$ is the constant given in \eqref{Ray2.2}.
\end{definition}
Notice that if $\bm u\in \bm V^0(\Gamma)$, equation \eqref{veryWeakForm} can be written as
\begin{equation}\label{veryWeakForm0}
\innOm{\bm y}{\bm f} - \duaOm{p}{h} = \innGa{\bm u}{-\partial_{\bm n}\bm z_{\bm f,h}+q_{\bm f,h}\bm n}.
\end{equation}

The definition follows integrating by parts twice the equation and once the null divergence condition. It can be written as two separate equations, one tested with $\bm f$ and the other one with $h$, as in \cite{MR2591040} or \cite{MR2371113}, or as single equation, cf. \cite{MR884813} or \cite{MR1739399}.

Next, we state a regularity result analogous to \cite[Corollary A.1]{MR2371113}. In that reference, a smooth domain is taken into consideration and the limit cases $s=-1/2$ and $s=3/2$ can be achieved; however, this is not possible for polygonal domains so the cited result cannot be directly applied.
\begin{theorem}\label{T2.1}
	Suppose $\bm u\in \bm V^s(\Gamma)$ for $ -s^*< s < \min\{1/2+\xi,3/2\}$. Then the solution of \eqref{StateEquation} satisfies
	\[\bm y_{\bm u}\in \bm V^{s+1/2}(\Omega)\mbox{ and }
	p_{\bm u}\in\left\{\begin{array}{cl}
	H^{s-1/2}(\Omega)/\mathbb R&\mbox{ if }s\geq 1/2,\\
	\left(H^{1/2-s}(\Omega)/\mathbb R\right)'&\mbox{ if }s\leq 1/2.
	\end{array}\right.
	\]
	Moreover, the control-to-state mapping $\bm u\mapsto \bm y_{\bm u}$ is continuous from $\bm V^s(\Gamma)$ to $\bm V^{s+1/2}(\Omega)$.
\end{theorem}
\begin{proof}
	The proof follows by interpolation. The technique of proof is the same as in \cite[Appendix A]{MR2371113} or \cite[Section 2]{MR3432846}, so we will just give a sketch of the proof and check some of the details that are different from those references.
	
	We first do the regular case. Suppose $1/2\leq s<\min\{1/2+\xi,3/2\}$.
	From the definition of $\bm V^s(\Gamma)$ we know that there exists $\bm Y\in \bm H^{s+1/2}(\Omega)$ such that the boundary trace of $\bm Y$ equals $\bm u$. So we have that $\bm F = -\Delta \bm Y\in \bm H^{s-3/2}(\Omega)$ and $H = \nabla \cdot \bm Y\in H^{s-1/2}(\Omega)$. By linearity, we have that $\bm y_{\bm u}-\bm Y = \bm z_{\bm F, H}$ and $p_{\bm u} = p_{\bm F, H}$, where $(\bm z_{\bm F, H},p_{\bm F, H})$ is the variational solution of \eqref{StokesRegularity} for data $(\bm F,H)$. From \cite[Theorem 5.5(a)]{MR977489}, and using that $s-1/2< \xi$, we have then that $\bm y_{\bm u}-\bm Y\in \bm H^{s+1/2}(\Omega)$ and $p_{\bm u}\in H^{s-1/2}(\Omega)/\mathbb R$, and the result follows in a straightforward way.
	
	Consider now $-s^*<s<0$. Uniqueness follows testing \eqref{veryWeakForm} for the data $\bm u=\bm 0$ and the pairs $(\bm f,0)$ and $(\bm 0, h)$ for $\bm f\in\bm L^2(\Omega)$ and $h\in H^1(\Omega)$ such that $\innOm{h}{1}=0$; compare to \cite[Theorem A.1(i)]{MR2371113} or \cite[Theorem 2.5]{MR3432846}.
	
	Existence follows by density arguments. Take $\bm u\in \bm V^{1/2}(\Gamma)$, which is dense in $\bm V^{s}(\Gamma)$.
	Notice that $-1/2<-s-1/2<0$ and $1/2<1/2-s<1$ and hence $\bm L^2(\Omega)$ is dense in $\bm H^{-s-1/2}(\Omega)$ and $H^1(\Omega)$ is dense in $ H^{1/2-s}(\Omega)$.  Therefore, we can consider
	\[\mathcal{F}=\{\bm f\in \bm L^2(\Omega):\ \|\bm f\|_{\bm H^{-s-1/2}(\Omega)} = 1 \}\]
	and
	\[\mathcal{H} = \{h\in H^1(\Omega)/\mathbb{R}:\ \|h\|_{H^{1/2-s}(\Omega)/\mathbb R} = 1\}\]
	to test the norms in $\bm H^{s+1/2}(\Gamma)$ and $\left(H^{1/2-s}(\Omega)/{\mathbb{R}}\right)'$ respectively of the variational solution $(\bm y_{\bm u},p_{\bm u})$ of \eqref{veryWeakForm}. We obtain, using estimate \eqref{NormalTraceContinuity},
	\begin{align*}
	\|\bm y_{\bm u}\|_{\bm H^{s+1/2}(\Gamma)} &= \sup_{\bm f\in\mathcal{F}}[\bm f,\bm y_{\bm u}]_{\bm H^{-s-1/2}(\Omega),\bm H^{s+1/2}(\Omega)} =  \sup_{\bm f\in\mathcal{F}}\innOm{\bm f}{\bm y_{\bm u}}\\
	&= \sup_{\bm f\in\mathcal{F}}  [\bm u, -\partial_{\bm n}\bm z_{\bm f, 0} + q_{\bm f,0}\bm n+
	c(\bm z_{\bm f,0},q_{\bm f,0})\bm n  ]_{\bm H^{s}(\Gamma),\bm H^{-s}(\Gamma)}\\
	&\leq \sup_{\bm f\in\mathcal{F}}  \|\bm u\|_{\bm H^{s}(\Gamma)}\| -\partial_{\bm n}\bm z_{\bm f, 0} + q_{\bm f,0}\bm n+
	c(\bm z_{\bm f,0},q_{\bm f,0})\bm n    \|_{\bm H^{-s}(\Gamma)}\\
	&\leq C \sup_{\bm f\in\mathcal{F}}  \|\bm u\|_{\bm H^{s}(\Gamma)} \|\bm f \| _{\bm H^{-s-1/2}(\Omega)} = C  \|\bm u\|_{\bm H^{s}(\Gamma)},
	\end{align*}
	and
	\begin{align*}
	\hspace{2em}&\hspace{-2em}\|p_{\bm u}\|_{\left(H^{1/2-s}(\Omega)/{\mathbb{R}}\right)'}\\
	&= \sup_{h\in\mathcal{H}}[p_{\bm u},h]_{\left(H^{1/2-s}(\Omega)/{\mathbb{R}}\right)',H^{1/2-s}(\Omega)/{\mathbb{R}}} \\
	&=  \sup_{h\in\mathcal{H}}[p_{\bm u},h]_{\left(H^{1}(\Omega)/{\mathbb{R}}\right)',H^{1}(\Omega)/{\mathbb{R}}} \\
	&= \sup_{h\in\mathcal{H}}[\bm u, -\partial_{\bm n}\bm z_{\bm 0, h} + q_{0,h}\bm n+
	c(\bm z_{\bm 0,h},q_{\bm 0,h})\bm n
	]_{\bm H^{s}(\Gamma)/{\mathbb{R}},\bm H^{-s}(\Gamma)/{\mathbb{R}}}\\
	&\leq \sup_{h\in\mathcal{H}}  \|\bm u\|_{\bm H^{s}(\Gamma)}\| -\partial_{\bm n}\bm z_{\bm 0, h} + q_{\bm 0, h}\bm n+
	c(\bm z_{\bm 0,h},q_{\bm 0,h})\bm n
	\|_{\bm H^{-s}(\Gamma)}\\
	&\leq C \sup_{h\in\mathcal{H}}  \|\bm u\|_{\bm H^{s}(\Gamma)} \|h \| _{\bm H^{1/2-s}(\Omega)/\mathbb R} = C  \|\bm u\|_{\bm H^{s}(\Gamma)}.
	\end{align*}
	The above proved estimates allow us to take a sequence $\bm u_n$ in $\bm V^{1/2}(\Gamma)$ converging to $\bm u$ in $\bm V^{s}(\Gamma)$ and obtain $\bm y_{\bm u}\in  \bm V^{s+1/2}(\Omega)$ and $p_{\bm u}\in (H^{1/2-s}(\Omega)/\mathbb{R})' $ as the limits of the sequences $\bm y_{\bm u_n}$ and $p_{\bm u_n}$; cf. \cite[Theorem A.1(ii)]{MR2371113} or \cite[Theorem 2.5]{MR3432846}.
	
	Finally, the case $0\leq s < 1/2$ follows by interpolation.
\end{proof}
\begin{remark}
	If $\bm u\in \bm V^{1/2}(\Gamma)$, then the very weak solution and the variational solution are the same.
\end{remark}

Next, we have to give some meaning to the  mixed form. The main problem is that for data in $\bm u\in \bm V^s(\Gamma)$, $s<1/2$, the gradient of the state is not a function in $\mathbb L^2(\Omega)$.

We start with the regular compressible Stokes problem. Consider $\bm f\in \bm L^2(\Omega)$ and $h\in H^1(\Omega)$ such that $\innOm{h}{1}=0$ and denote $\bm z = \bm z_{\bm f,h}$ and $q=q_{\bm f,h}$ the (variational) solution of \eqref{StokesRegularity}. If we denote $\mathbb{G}_{\bm f,h}=\nabla \bm z_{\bm f,h}\in \mathbb  L^2(\Omega)$, we have that the triplet $(\mathbb{G}_{\bm f,h},\bm z_{\bm f,h}, q_{\bm f,h})\in \mathbb  L^2(\Omega)\times \bm H^1_0(\Omega)\times L^2_0(\Omega)$ is the unique solution of the weak formulation
\begin{align}
\innOm{\mathbb{G}}{\mathbb{T}} +\innOm{\bm z}{\nabla\cdot\mathbb{T}}&= 0,\label{E2.4mixa}\\
\innOm{\mathbb{G}}{\nabla\bm v} - \innOm{q}{\nabla\cdot \bm v}&=\innOm{\bm f}{\bm v},\label{E2.4mixb}\\
-\innOm{\bm z}{\nabla w} &= \innOm{h}{w},\label{E2.4mixc}\\
\innOm{q}{1}&=0,\label{E2.4mixd}
\end{align}
for all $(\mathbb{T},\bm v,w )\in \mathbb H(\textup{div},\Omega)\times \bm H_0^1(\Omega)\times H^1(\Omega)$.
Moreover, it is clear that the regularity results stated above for \eqref{StokesRegularity} apply and $\mathbb G_{\bm f,h}\in \mathbb  H^{s-1/2}(\Omega)$ for $0<s<s^*$. Notice also we can define analogously to \eqref{Ray2.2}
\begin{equation}
\label{Ray2.2b}c(\mathbb G,q) = \frac{1}{|\Gamma|}\innGa{q-(\mathbb G \bm n)\cdot \bm n}{1}.
\end{equation}

Next we give a mixed formulation of problem \eqref{veryWeakForm} for Dirichlet data $\bm u\in\bm V^s(\Gamma)$ for $-s^*<s$.
\begin{definition}
	For $-s^*<s$ and $\bm u\in\bm V^s(\Gamma)$, we say $\bm y_{\bm u}\in \bm V^0(\Omega)$, $\mathbb L_{\bm u} = \nabla \bm y_{\bm u} \in (\mathbb H^{1}(\Omega))'$, $p_{\bm u}\in (H^{1}(\Omega)/\mathbb R)'$ is a solution in the transposition sense of
	\begin{align*}
	-\Delta\bm y+\nabla p &=\bm f \quad  \text{in}\ \Omega,\\
	\nabla\cdot\bm y&=0 \quad\  \text{in}\ \Omega,\\
	\bm y&=\bm u \quad\  \text{on}\  \Gamma,
	\end{align*}
	if $ (\bm y_{\bm u}, \mathbb L_{\bm u}, p_{\bm u}) $ satisfy
	\begin{subequations}\
		\begin{align}
		\duaOm{\mathbb{L}}{\mathbb{T}} &= -\innOm{\bm y}{\nabla\cdot\mathbb{T}}+\duaGa{\bm u}{\mathbb{T}\bm n} ,\label{DE2.3b}\\
		\innOm{\bm y}{\bm f} - \duaOm{p}{h} &=\duaGa{\bm u}{-\mathbb{G}_{\bm f,h} \bm n+q_{\bm f,h}\bm n+c(\mathbb{G}_{\bm f,h},q_{\bm f,h}) \bm n}, \label{DE2.3a}
		\end{align}
	\end{subequations}
	for every $\bm f\in \bm L^2(\Omega)$, $h\in H^1(\Omega)$ such that $\innOm{h}{1}=0$ and $\mathbb T\in  \mathbb H^1(\Omega)$,
	where $(\mathbb{G}_{\bm f,h},\bm z_{\bm f,h}, q_{\bm f,h})\in \mathbb L^2(\Omega)\times \bm H^1_0(\Omega)\times L^2_0(\Omega)$ is the solution of \eqref{E2.4mixa}--\eqref{E2.4mixd} for data $(\bm f,h)$.
\end{definition}

The above definition simply incorporates an adequate definition for the gradient to the transposition solution defined in \Cref{D2.1}. Nevertheless, this formulation is still not appropriate to use together with \eqref{E2.4mixa}--\eqref{E2.4mixd} in the context of hybridizable discontinuous Galerkin methods. Taking advantage of the regularity results stated in \Cref{T2.1}, we have $\mathbb{L}_{\bm u} \in \mathbb H^{s-1/2}(\Omega)$ if $1/2\leq s< \min\{1/2+\xi,3/2\}$ and $\mathbb{L}_{\bm u} \in (\mathbb  H^{1/2-s}(\Omega))'$ if $-s^*<s< 1/2$.

So we have that if $-s^*<s< 1/2$ and $\bm u\in \bm V^s(\Gamma)$, then there exists a unique solution $(\mathbb{L}_{\bm u},\bm y_{\bm u},p_{\bm u})\in(\mathbb H^{1/2-s}(\Omega))'\times \bm V^{1/2+s}(\Omega)\times  \left(H^{1/2-s}(\Omega)/\mathbb R\right)'$ of the problem
\begin{subequations}
	\begin{align}
	\duaOm{\mathbb{L}}{\mathbb{T}}+\innOm{\bm y}{\nabla\cdot\mathbb{T}} &= \duaGa{\bm u}{\mathbb{T}\bm n},\label{DE2.4c}\\
	\duaOm{\mathbb{L}}{\mathbb{G}_{\bm{f},h}} - \duaOm{p}{h} &= 0 ,\label{DE2.4a} \\
	\duaOm{\nabla q_{\bm{f},h}}{\bm y}&= \duaGa{\bm u}{q_{\bm{f},h} \bm n},\label{DE2.4b}
	\end{align}
\end{subequations}
for every $\bm f\in \bm L^2(\Omega)$, $h\in H^1(\Omega)$ such that $\innOm{h}{1}=0$ and $\mathbb T\in \mathbb H^1(\Omega)$,
where
$(\mathbb{G}_{\bm{f},h},\bm z_{\bm{f},h}, q_{\bm{f},h})\in \displaystyle\bigcap_{t<s^*}\mathbb  H^{1/2+t}(\Omega) \times \bm H^{3/2+t}(\Omega)\times H^{1/2+t}(\Omega)/\mathbb R\hookrightarrow  \mathbb H^{1/2-s}(\Omega)\times \bm H^{3/2-s}(\Omega)\times H^{1/2-s}(\Omega)/\mathbb R$
is the solution of \eqref{E2.4mixa}--\eqref{E2.4mixd} for data $(\bm f,h)$.

Taking all this into account we can summarize our results in the following theorem.
\begin{theorem}\label{T2.4}
	For every $\bm u\in \bm V^0(\Gamma)$, there exists a unique solution $$(\mathbb{L}_{\bm u},\bm y_{\bm u},p_{\bm u}) \in  (\mathbb H^{1/2}(\Omega))'\times\bm V^{1/2}(\Omega)\times  \left(H^{1/2}(\Omega)/\mathbb R\right)'$$ of
	\begin{subequations}
		\begin{align}
		\duaOm{\mathbb{L}}{\mathbb{T}}+\innOm{\bm y}{\nabla\cdot\mathbb{T}} &= \innGa{\bm u}{\mathbb{T}\bm n},\label{SE3}\\
		\duaOm{\mathbb{L}}{\nabla \bm v} - \duaOm{p}{\nabla \cdot\bm v} &= 0,\label{SE1} \\
		\duaOm{\nabla w}{\bm y}&= \innGa{\bm u}{w \bm n},\label{SE2}
		\end{align}
		for all $(\mathbb T, \bm v, w) \in \mathbb H^1(\Omega) \times\bigcap_{t<s^*} \bm H^{3/2+t}(\Omega)\cap \bm H^1_0(\Omega)\times \bigcap_{t<s^*} H^{1/2+t}(\Omega)$.
	\end{subequations}
	Moreover, if  $\bm u\in \bm V^s(\Gamma)$, $-1/2<s<s^*$, then  $(\mathbb{L}_{\bm u},\bm y_{\bm u},p_{\bm u})\in (\mathbb H^{1/2-s}(\Omega))'\times  \bm V^{1/2+s}(\Omega)\times  \left(H^{1/2-s}(\Omega)/\mathbb R\right)'$. Finally, the control-to-state mapping $\bm u \mapsto $ \,\!\! $ (\mathbb{L}_{\bm u}, \bm y_{\bm u},p_{\bm u})$ is continuous from $$\bm V^s(\Gamma) \quad \mbox{to} \quad (\mathbb H^{1/2-s}(\Omega))'\times \bm V^{1/2+s}(\Omega)\times \left(H^{1/2-s}(\Omega)/\mathbb R\right)'$$ for $-s^*<s<\min\{1/2+\xi,3/2\}$.
	
\end{theorem}

\subsection{Well posedness and regularity of the tangential control problem}
It is clear that $\bm U\hookrightarrow \bm L^2(\Gamma)$ and there is no ambiguity in denoting by $ u$ the elements of $\bm U$. Hence the control-to-state mapping $u\mapsto \bm y_u$ is continuous from $\bm U$ to $\bm V^{1/2}(\Omega)$, and there exists a unique solution of the  control problem
\[(P_{\bm \tau})\qquad\min J(u) = \frac{1}{2}\|\bm y_u-\bm y_d\|^2_{\bm L^2(\Omega)}+\frac{\gamma}{2}\|u\|^2_{L^2(\Gamma)},\]
where $\bm y_{u}$ is the solution of the state equation
\begin{subequations}
	\begin{align}
	\duaOm{\mathbb{L}}{\mathbb{T}}+\innOm{\bm y}{\nabla\cdot\mathbb{T}} &= \innGa{u\bm \tau}{\mathbb{T}\bm n},\label{SE3t}\\
	\duaOm{\mathbb{L}}{\nabla \bm v} - \duaOm{p}{\nabla\cdot\bm v} &=0,\label{TCSE1} \\
	\duaOm{\nabla w}{\bm y}&=0,\label{SE1t}
	\end{align}
	for all $(\mathbb T, \bm v, w) \in \mathbb H^1(\Omega) \times \bigcap_{t<s^*} \bm H^{3/2+t}(\Omega)\cap \bm H^1_0(\Omega)\times \bigcap_{t<s^*} H^{1/2+t}(\Omega)$.
\end{subequations}
Notice that \eqref{SE3t}, \eqref{TCSE1}, \eqref{SE1t} is the weak formulation \eqref{SE3}, \eqref{SE1}, \eqref{SE2} obtained in \Cref{T2.4} for the Stokes problem \eqref{Ori_problem2} with Dirichlet datum $u\bm \tau$, where we have used that $u\bm \tau \cdot  w \bm n = 0$ for any pair of functions $u,w$ in $L^2(\Gamma)$.
%


\begin{theorem}\label{T2.8}Suppose $\bm y_d \in \bm H^{\min\{2,\xi\} }(\Omega)$. Let $u\in L^2(\Gamma)$ be the solution of problem $(P_{\bm\tau})$. Then
	\begin{align}\label{opt_control_global_reg}
	u\in H^s(\Gamma)
\end{align}
	for all $1/2< s<\min\{3/2,\xi-1/2\}$ and there exists
	\begin{align*}
	\bm y &\in \bm V^{s+1/2}(\Omega),  &  \mathbb{L} &\in \mathbb  H^{s-1/2}(\Omega),  &  p &\in H^{s-1/2}(\Omega)\cap L_0^2(\Omega),\\
	\bm z &\in \bm V_0^{r+1}(\Omega),  &  \mathbb{G} &\in \mathbb H^r(\Omega),  &  q &\in H^{r}(\Omega)\cap L_0^2(\Omega),
	\end{align*}
	for all $1 <r<\min\{3,\xi\}$ such that
	\begin{subequations}
		\begin{align}
		\innOm{\mathbb{L}}{\nabla \bm v} - \innOm{p}{\nabla\cdot\bm v} &= 0 ,\label{TCOSE11} \\
		\innOm{\nabla w}{\bm y}&= 0,\label{TCOSE12}\\
		\innOm{\mathbb{L}}{\mathbb{T}}+\innOm{\bm y}{\nabla\cdot\mathbb{T}} &= \innGa{u\bm\tau}{\mathbb{T}\bm n},\label{TCOSE13}\\
		\innOm{\mathbb{G}}{\nabla\bm v} +\innOm{q}{\nabla\cdot \bm v}&= \innOm{\bm y-\bm y_d}{\bm v},\label{TCOASE11}\\
		-\innOm{\bm z}{\nabla w} &=0,\label{TCOASE12}\\
		\innOm{\mathbb{G}}{\mathbb{T}} +\innOm{\bm z}{\nabla\cdot\mathbb{T}}&=0,\label{TCOASE4}\\
		\innGa{\gamma  u\bm\tau-\mathbb{G}\bm n}{\mu\bm \tau} &= 0,\label{1TCOOC}
		\end{align}
	\end{subequations}
	for all $(\mathbb T, \bm v, w, \mu )\in \mathbb H(\textup{div}, \Omega)\times \bm H_0^1(\Omega)\times H^1(\Omega)\times L^2(\Gamma)$. Moreover,
	\begin{equation}\label{eqn:opt_control_local_reg}
	u \in \prod_{i=1}^{m} H^{r-1/2}(\Gamma_i)\mbox{ for all }r<\min\{3,\xi\}.
	\end{equation}
\end{theorem}
\begin{proof}
	The optimality conditions follow in a standard way by computing the derivative of the functional with the help of the chain rule, the integration by parts formula and \Cref{D2.1}. The regularity follows from a bootstrapping argument.
	
	From \Cref{T2.1} we have that $\bm y\in \bm V^{1/2}(\Omega)$. Using this and the regularity of the data $\bm y_d$, we deduce from \Cref{Dauge55b} that
	$\bm z\in \bm V^{t+1}(\Omega)$ and $q\in H^{t}(\Omega)\cap L_0^2(\Omega)$ for all $t\leq 3/2$ such that $t<\xi$. From the trace theory, it is clear that
	\[\mathbb{G} \bm n = \partial_{\bm n}\bm z\in \Pi_{i=1}^m \bm H^{t-1/2}(\Gamma_i) \mbox{ for all } t\leq 3/2\mbox{ such that }t<\xi.\]
	Since $\xi>1$, we notice that the gradient of the dual pressure $q$ is a function in $H^{t-1}(\Omega)$ with $t-1>0$. So we have that each component $z^i$, $i=1,2$ of $\bm z$, satisfies $\-\Delta z^i\in H^{t-1}(\Omega)$ and $z^i=0$ on $\Gamma$. Therefore, we have that $\partial_{\bm n}z^i(x_j)=0$, $i=1,2$, for every corner $x_j$ (cf.\ \cite[Appendix A]{MR2567245}, \cite[Section 4]{MR2837508}), and hence
	we also have (cf.\ \cite[Lemma A.2]{MR2567245}) that
	\[\mathbb{G} \bm n=\partial_{\bm n}\bm z\in  \bm H^{t-1/2}(\Gamma)\mbox{ for all }t\leq 3/2\mbox{ such that }t<\xi.\]
	Next, using that the pressure does not appear in the optimality condition \eqref{1TCOOC}, we can write
	\[\gamma u\bm \tau = \mathbb G \bm n=\partial_{\bm n}\bm z,\]
	and therefore the Dirichlet datum of the state equation is also in the space $ \bm H^{t-1/2}(\Gamma)$ for all $t\leq 3/2$ such that $t<\xi$.
	
	Repeating the argument, we obtain in a first step from \Cref{T2.1} that $\bm y\in \bm V^{t}(\Gamma)$ for all $t\leq 3/2$ such that $t<\xi$, which leads, together with the maybe higher regularity of $\bm y_d$ and \Cref{Dauge55b}, to $\bm z \in \bm V^{1/2+t_2}(\Omega)$ for $t_2\leq 5/2$, $t_2<\xi$. The normal trace argument leads to $u\bm\tau\in \Pi_{i=1}^m \bm H^{t_2-1/2}(\Gamma_i)$, but when we paste together the pieces with the help of the zero value at the corners, we cannot go further for the Dirichlet datum of the state equation than
	\begin{equation}\label{DirichletDataRegularity}u\bm\tau\in \bm V^{s}(\Gamma)\mbox{ for }s<3/2,\ s<\xi-1/2.\end{equation}
	The claimed regularity for the optimal control follows from the previous relation.
	
	Taking the same argument for a third time, we obtain the regularity of the other involved variables.
\end{proof}
Notice that a higher regularity of the target state would not lead to a higher regularity of the solution, since it is mainly bounded by the singularities that appear due to the corners. Low regularity of the target would nevertheless lead to a low regularity solution. Suppose for instance that $\xi>2$ ($\omega>0.7\pi$) and $\bm y_d\in\bm H^\alpha(\Omega)$, with $\alpha <2$. If $\alpha <1$, then the gradient of the dual pressure would not be a function, and the argument of the proof would not lead to any conclusion. If $1\leq \alpha\leq 3/2$, then the argument would stop in the first step, obtaining regularity for the control $u\in H^{\alpha-1/2}(\Gamma)$. If $3/2<\alpha<2$, then the argument would finish in the second step obtaining again $u\in H^{\alpha-1/2}(\Gamma)$.

We use the following reformulation of the optimality system in our analysis of the HDG method:
\begin{corollary}\label{T2.9} Let $\bm y_d\in \bm H^{\min\{2,\xi\} }(\Omega)$. Then the solution of the optimality system \eqref{TCOSE11}-\eqref{1TCOOC} also satisfies the following well-posed problem: find
	\begin{gather*}
	u \in H^{1/2}(\Gamma),  \quad  \bm y \in \bm V^1(\Omega),  \quad  \mathbb{L} \in \mathbb L^2(\Omega),  \quad  p \in L^2_0(\Omega),\\
	\bm z \in \bm V^1_0(\Omega),  \quad  \mathbb{G} \in \mathbb L^2(\Omega),  \quad  q \in L^2_0(\Omega),
	\end{gather*}
	such that $ \mathbb L - p\mathbb{I}, \mathbb{G}+q\mathbb{I} \in \mathbb H(\textup{div}, \Omega) $ and
	\begin{subequations}
		\begin{align}
		\innOm{\mathbb{L}}{\mathbb{T}}+\innOm{\bm y}{\nabla \cdot\mathbb{T}} &= \innGa{u\bm\tau}{\mathbb{T}\bm n}, \label{TCOSE1}\\
		-\innOm{\nabla\cdot(\mathbb{L}- p\mathbb I)}{ \bm v} &=0,\label{TCOSE2} \\
		(\nabla\cdot\bm y, w)_{\Omega} &=0,\label{TCOSE3}\\
		\innOm{\mathbb{G}}{\mathbb{T}} +\innOm{\bm z}{\nabla\cdot\mathbb{T}}&=0,\label{TCOASE1}\\
		-\innOm{\nabla\cdot(\mathbb{G} + q\mathbb I)}{\bm v} &= \innOm{\bm y-\bm y_d}{\bm v}, \label{TCOASE2}\\
		\innOm{\nabla \cdot \bm z}{ w} &=0, \label{TCOASE3}\\
		\innGa{\gamma  u\bm\tau-\mathbb{G} \bm n}{\mu\bm{\tau}} &= 0, \label{TCOOC}
		\end{align}
		for all $(\mathbb T, \bm v, w, \mu )\in \mathbb H(\textup{div}, \Omega)\times \bm L^2(\Omega)\times L_0^2(\Omega)\times H^{1/2}(\Gamma)$.
	\end{subequations}
\end{corollary}

\section{HDG Formulation }
\label{sec:HDG_formulation}

Before we introduce the HDG method, we first define some notation.  Let $\mathcal{T}_h$ be a conforming and quasi-uniform collection of disjoint elements that partition $\Omega$.  We denote by $\partial \mathcal{T}_h$ the set $\{\partial K: K\in \mathcal{T}_h\}$. For an element $K$ of the collection  $\mathcal{T}_h$, $e = \partial K \cap \Gamma$ is the boundary face if the Lebesgue measure of $e$ is non-zero. For two elements $K^+$ and $K^-$ of the collection $\mathcal{T}_h$, $e = \partial K^+ \cap \partial K^-$ is the interior face between $K^+$ and $K^-$ if the  Lebesgue measure of $e$ is non-zero. Let $\varepsilon_h^o$ and $\varepsilon_h^{\partial}$ denote the set of interior and boundary faces, respectively. We denote by $\varepsilon_h$ the union of  $\varepsilon_h^o$ and $\varepsilon_h^{\partial}$. We  introduce various inner products for our finite element spaces. We write
\begin{gather*}
(\eta,\zeta)_{\mathcal{T}_h} = \sum_{K\in\mathcal{T}_h} (\eta,\zeta)_K,  \quad  \langle \eta,\zeta\rangle_{\partial\mathcal{T}_h} = \sum_{K\in\mathcal T_h} \langle \eta,\zeta\rangle_{\partial K},\\
(\bm\eta,\bm\zeta)_{\mathcal{T}_h} = \sum_{i=1}^2 (\eta_i,\zeta_i)_{\mathcal T_h},  \quad 	(\mathbb L,\mathbb G)_{\mathcal{T}_h} = \sum_{i,j=1}^2 (L_{ij},G_{ij})_{\mathcal T_h},\\
\langle \bm \eta,\bm\zeta\rangle_{\partial\mathcal{T}_h} = \sum_{i=1}^2 \langle \eta_i,\zeta_i\rangle_{\partial \mathcal T_h},
\end{gather*}
where $ (\cdot,\cdot)_K$ and $ \langle \cdot,\cdot\rangle_{\partial K} $ denote the standard $ L^2 $ inner products on the domains $K\subset\mathbb R^2$ and $\partial K\subset\mathbb R$.
%
%
%

Let $\mathcal{P}^k(D)$ denote the set of polynomials of degree at most $k$ on a domain $D$.  We introduce the following discontinuous finite element spaces
\begin{align}
\mathbb K_h&:=\{\mathbb L \in \mathbb L^2(\Omega):\mathbb L|_K\in[\mathcal P^k(K)]^{2\times 2}, \ \forall K\in\mathcal T_h\},\\
\bm{V}_h  &:= \{\bm{v}\in \bm L^2(\Omega): \bm{v}|_{K}\in [\mathcal{P}^{k+1}(K)]^2, \   \forall K\in \mathcal{T}_h\},\\
{W}_h  &:= \{{w}\in L^2(\Omega): {w}|_{K}\in \mathcal{P}^k(K),  \ \forall K\in \mathcal{T}_h\},\\
\bm {M}_h  &:= \{{\mu}\in \bm L^2(\mathcal{\varepsilon}_h): {\bm\mu}|_{e}\in [\mathcal{P}^k(e)]^2, \  \forall e\in \varepsilon_h\},\\
{M}_h  &:= \{{\mu}\in L^2(\mathcal{\varepsilon}_h^\partial): {\mu}|_{e}\in \mathcal{P}^k(e), \  \forall e\in \varepsilon_h^\partial\},
\end{align}
for the flux variables, velocity, pressure, boundary trace variables, and boundary control, respectively.  Note that the polynomial degree for the scalar variable is one order higher than the polynomial degree for the flux variables and numerical trace.  This combination of spaces has been used for the Navier-Stokes equations in \cite{MR3556409}.  The boundary trace variables will be used to eliminate the state and flux variables from the coupled global equations, thus substantially reducing the number of degrees of freedom.

Let  $\bm M_h(o)$ denote the space defined in the same way as $\bm M_h$, but with $\varepsilon_h$ replaced by $\varepsilon_h^o$. Note that $\bm M_h$ consists of functions which are continuous inside the faces (or edges) $e\in \varepsilon_h$ and discontinuous at their borders. In addition, spatial derivatives of any functions in the finite element spaces are taken piecewise on each element $K\in \mathcal T_h$. Finally, we define
\begin{align*}
{W}_h^0  &= \left\{{w}\in L^2(\Omega): {w}|_{K}\in \mathcal{P}^k(K), \  \forall K\in \mathcal{T}_h \ \textup{and}\  (w,1)_{\Omega}  = 0\right\}.
%
%
%
\end{align*}

To approximate the solution of the mixed weak form \eqref{TCOSE1}-\eqref{TCOOC} of the optimality system, the HDG method seeks approximate fluxes $\mathbb L_h,\mathbb G_h \in \mathbb {K}_h $, states $ \bm y_h, \bm z_h \in \bm V_h $, pressures  $p_h,q_h \in W_h^0$,  interior element boundary traces $ \widehat{\bm y}_h^o,\widehat{\bm z}_h^o \in \bm M_h(o) $, and boundary control $ u_h\in  M_h$ satisfying
\begin{subequations}\label{HDG_discrete2}
	\begin{align}
	(\mathbb L_h,\mathbb T_1)_{\mathcal{T}_h}+(\bm y_h,\nabla\cdot \mathbb T_1)_{\mathcal{T}_h}-\langle \widehat{\bm y}_h^o, \mathbb T_1 \bm{n}\rangle_{\partial\mathcal{T}_h\backslash\varepsilon_h^\partial} &=\langle  u_h\bm \tau, \mathbb T_1 \bm{n}\rangle_{\varepsilon_h^\partial}, \label{HDG_discrete2_a}\\
	(\mathbb L_h,\nabla\bm v_1)_{\mathcal T_h} - (p_h, \nabla\cdot \bm{v}_1)_{\mathcal{T}_h} - \langle(\widehat{\mathbb L}_h-  p_h \mathbb I)\bm n, \bm v_1\rangle_{\partial {\mathcal{T}_h}} &= (\bm f,\bm{v}_1)_{\mathcal{T}_h},  \label{HDG_discrete2_b}\\
	-(\bm y_h,\nabla w_1)_{\mathcal T_h}+\langle \widehat{\bm y}_h^o\cdot\bm n,w_1\rangle_{\partial\mathcal T_h\backslash {\varepsilon_h^\partial}}&= 0,\label{HDG_discrete2_c}
	\end{align}
	for all $(\mathbb T_1,\bm v_1,w_1)\in \mathbb K_h\times\bm{V}_h\times W_h^0$,
	\begin{align}
	(\mathbb G_h,\mathbb T_2)_{\mathcal{T}_h}+(\bm z_h,\nabla\cdot\mathbb T_2)_{\mathcal{T}_h}-\langle \widehat{\bm z}_h^o, \mathbb T_2 \bm{n}\rangle_{\partial\mathcal{T}_h\backslash \varepsilon_h^\partial} &=0, \label{HDG_discrete2_d}\\
	(\mathbb G_h,\nabla\bm v_2)_{\mathcal T_h}+(q_h, \nabla\cdot \bm{v}_2)_{\mathcal{T}_h}- \langle(\widehat{\mathbb G}_h+ q_h\mathbb I)\bm n, \bm v_2\rangle_{\partial {\mathcal{T}_h}} &= (\bm y_h-\bm y_d,\bm{v}_2)_{\mathcal{T}_h},  \label{HDG_discrete2_e}\\
	-(\bm z_h,\nabla w_2)_{\mathcal T_h}+\langle \widehat{\bm z}_h^o\cdot\bm n,w_2\rangle_{\partial\mathcal T_h\backslash\varepsilon_h^\partial}&=0,\label{HDG_discrete2_f}
	\end{align}
	for all $(\mathbb T_2,\bm v_2,w_2)\in \mathbb K_h\times\bm{V}_h\times W_h^0$,
	\begin{align}
	\langle(\widehat{\mathbb L}_h- p_h\mathbb I)\bm n,\bm\mu_1\rangle_{\partial\mathcal T_h\backslash\varepsilon^{\partial}_h}&=0,\label{HDG_discrete2_g}
	\end{align}
	for all $\bm\mu_1\in \bm M_h(o)$,
	\begin{align}
	\langle(\widehat{\mathbb G}_h+ q_h\mathbb I)\bm n,\bm\mu_2\rangle_{\partial\mathcal T_h\backslash\varepsilon^{\partial}_h}&=0,\label{HDG_discrete2_h}
	\end{align}
	for all $\bm\mu_2\in \bm M_h(o)$,
	\begin{align}
	\langle \widehat{\mathbb G}_h  \bm n - \gamma  u_h \bm{\tau},  \mu_3 \bm \tau \rangle_{\varepsilon_h^\partial} &= 0,\label{HDG_discrete2_m}
	\end{align}
	for all $\mu_3 \in  M_h $.  In contrast to \Cref{Analysis_D_S}, here we assume the forcing $ \bm f $ may be nonzero.

	The numerical traces on $\partial\mathcal{T}_h$ are defined as
	\begin{align}
	\widehat{\mathbb L}_h\bm n &= \mathbb L_h\bm n-h^{-1}(\bm P_M \bm y_h-\widehat{\bm y}_h^o)  \quad \ \  \text{on}\ \partial \mathcal{T}_h\backslash \varepsilon_h^\partial, \label{HDG_discrete2_n}\\
	\widehat{\mathbb L}_h\bm n&= \mathbb L_h\bm n-h^{-1}(\bm P_M\bm y_h- u_h\bm \tau)  \quad  \text{on}  \ \varepsilon_h^\partial, \label{HDG_discrete2_o}\\
	\widehat{\mathbb G}_h\bm n&= \mathbb G_h\bm n-h^{-1}(\bm P_M\bm z_h-\widehat{\bm z}_h^o)  \quad \ \  \text{on}\ \partial \mathcal{T}_h\backslash \varepsilon_h^\partial,\label{HDG_discrete2_p}\\
	\widehat{\mathbb G}_h\bm n&= \mathbb G_h\bm n-h^{-1} \bm P_M\bm z_h  \qquad\qquad \ \   \text{on} \   \varepsilon_h^\partial,\label{HDG_discrete2_q}
	\end{align}
\end{subequations}
where $\bm P_M$ denotes the standard edge-wise $L^2$-orthogonal projection from $\bm L^2(e)$ onto $\mathcal{P}^k(e)$. This completes the formulation of the HDG method. It is straightforward to see that the system \eqref{HDG_discrete2} is consistent, i.e., the exact solution satisfies the system.
Furthermore, its implementation can be found in the arXiv preprint of this paper \cite{GongHuMateosSinglerZhang1}.

\section{Error analysis}
\label{sec:analysis}
In this section, we perform a convergence analysis of the HDG method for the tangential Dirichlet boundary control for Stokes equations.


	\subsection{Main result}
	We assume throughout that there exists a unique solution of the optimality system \eqref{TCOSE1}-\eqref{TCOOC} satisfying the following  regularity condition:
	\begin{align*}
	\mathbb L &\in \mathbb H^{r_{\mathbb L}}(\Omega), &  \bm y &\in \bm H^{r_{\bm y}}(\Omega), &  p &\in H^{r_p}(\Omega),\\
	\mathbb G &\in \mathbb H^{r_{\mathbb G}}(\Omega), &  \bm z &\in \bm H^{r_{\bm z}}(\Omega), &  q &\in H^{r_q}(\Omega),
	\end{align*}
	where
	
	\begin{align}\label{reg_assumption}
	r_{\bm y}>1, \  \ r_{\bm z}>2,  \ \ r_{\mathbb L}= r_p>0,  \ \ r_{\mathbb G}>1,  \ \ r_q>1, \ \   \mathbb L - p \mathbb I \in \mathbb H(\textup{div},\Omega).
	\end{align}

	This regularity condition is guaranteed to hold when the polygonal domain is convex; see \Cref{cor:main_result} below.
	
	In \eqref{reg_assumption}, the regularity for $\mathbb L$ and $p$ can be very low; in particular, $\mathbb L$ and $p$ are not guaranteed to have an $L^2$ boundary trace.  This causes difficulty for the numerical analysis of the HDG method.  For the convection diffusion equation, we used a special interpolation operator to deal with this difficulty in \cite{MR3831243,MR3992054}.  Later, in \cite{ChenFuSinglerZhang}, we used a special trace inequality in the numerical analysis of related embedded DG methods; we also use an improved trace inequality in this work, but the analysis is different since the spaces are not the same as in \cite{ChenFuSinglerZhang}.

	Our main result is below:
	\begin{theorem}\label{main_res}
		Let
		\begin{align*}
		s_{\mathbb{L}} &= \min\{r_{\mathbb L}, k+1 \},   \qquad 	s_{\bm y} = \min\{r_{\bm y}, k+2 \}, \qquad s_{p} = \min\{r_p, k+1 \},\\
		s_{\mathbb{G}} &= \min\{r_{\mathbb G}, k+1 \},   \qquad	s_{\bm z} = \min\{r_{\bm z}, k+2 \},  \qquad s_{q} = \min\{r_q, k+1 \},\\
		\mathcal M &= h^{s_{\mathbb L}}\norm{\mathbb L}_{s^{\mathbb L},\Omega}+  h^{s_{p}}\norm{p}_{s^{p},\Omega} + h^{s_{\bm y}-1}\norm{\bm y}_{s^{\bm y},\Omega}, \\
		\mathcal N &=  h^{s_{\mathbb G}}\norm{\mathbb G}_{s^{\mathbb G},\Omega}+  h^{s_{q}}\norm{q}_{s^{q},\Omega}+ h^{s_{\bm z}-1}\norm{\bm z}_{s^{\bm z},\Omega}.
		\end{align*}
		Then for $k\geq 0$, we have
		\begin{align*}
		\norm {u - u_h}_{\varepsilon^\partial_h} + \norm {\bm y -\bm y_h}_{\mathcal T_h} + \norm {\mathbb G - \mathbb G_h}_{\mathcal T_h} + \norm {\bm z-\bm z_h}_{\mathcal T_h} +\norm {q - q_h}_{\mathcal T_h} \lesssim h^{-\frac 1 2 }(h\mathcal M + \mathcal N).
		\end{align*}
		Moreover, if  $k\ge 1$, then
		\begin{align*}
		%
		%
		%
		\hspace{1em}&\hspace{-1em}\norm {\mathbb L - \mathbb L_h}_{\mathcal T_h} +\norm {p - p_h}_{\mathcal T_h} \lesssim 	h^{-1}(h\mathcal M + \mathcal N).
		\end{align*}
	\end{theorem}
	\begin{corollary}\label{cor:main_result}
		Let $\omega\in[\pi/3, \pi)$ be the largest interior angle of $\Gamma$ and let $\xi$
		be  the smallest real part of all of the roots $ \lambda $ of the equation
		\begin{align*}
		\frac{ \sin^2(\lambda\omega)-\lambda^2\sin^2\omega }{ \lambda^2 (\lambda - 1) }=0.
		\end{align*}
		Suppose $ \bm f = \bm 0 $, $ \bm y_d \in \bm H^{\min\{2,\xi\}}(\Omega) $.  Define $ r_{\Omega}$ by
		\[
		r_{\Omega} = \min\left\{ \frac{3}{2}, \xi- \frac 1 2 \right\} \in (1/2, 3/2].
		\]
		Then the regularity condition \eqref{reg_assumption}is satisfied.

		If $ k = 1 $, then for any $ r < r_\Omega $ we have
		\begin{align*}
		\|\mathbb L - \mathbb L_h\|_{\mathcal T_h} &\lesssim h^{r-1/2}, & 	\|\bm y  - \bm y_h\|_{\mathcal T_h} &\lesssim h^{r}, & 	\|p  - p_h\|_{\mathcal T_h} &\lesssim h^{r-1/2},\\
		\|\mathbb G - \mathbb G_h\|_{\mathcal T_h} &\lesssim h^{r}, & 	\|\bm z - \bm z_h\|_{\mathcal T_h} &\lesssim h^{r}, & 	\|q  - q_h\|_{\mathcal T_h} &\lesssim h^{r},\\
		& & \norm{u-u_h}_{\varepsilon_h^\partial} &\lesssim h^{r}. & &
		\end{align*}
		
		If $ k = 0$, then for any $ r < r_\Omega $ we have
		\begin{gather*}
		\|\bm y  - \bm y_h\|_{\mathcal T_h} \lesssim h^{1/2}, \qquad 	\|\mathbb G - \mathbb G_h\|_{\mathcal T_h} \lesssim h^{1/2},\\
		\|\bm z - \bm z_h\|_{\mathcal T_h}\lesssim h^{1/2},  \qquad 	\|q  - q_h\|_{\mathcal T_h} \lesssim h^{1/2},\\
		\norm{u-u_h}_{\varepsilon_h^\partial} \lesssim h^{1/2}.
		\end{gather*}
	\end{corollary}

\Cref{T2.8} gives $ u \in H^r(\Gamma) $, and so the convergence rate for the control is optimal for $ k = 1 $ with respect to this global regularity result.  However, \Cref{T2.8} also gives the higher \textit{local regularity result} \eqref{eqn:opt_control_local_reg} for the control: $ u \in H^\kappa(\Gamma_i) $ for each boundary segment $ \Gamma_i $, where $ \kappa < \min\{3,\xi\}-1/2 $.  Our numerical results in \Cref{sec:numerics} indicate that the actual convergence rate for $ k = 1 $ may indeed be restricted by the local regularity result instead of the global regularity result.  A completely different method of proof is likely required to establish a sharper convergence rate for the control with respect to the local regularity result.

Also, \Cref{T2.8} only yields global regularity results for the other variables.  Our convergence rates for the flux $\mathbb L $ and pressure $p$ are optimal for $ k = 1 $, but suboptimal for the other variables.

\subsection{Preliminary material}
We begin by defining the standard $L^2$ projections $\bm \Pi _{\mathbb K} : \mathbb L^2(\Omega)  \to \mathbb {K}_h$, $\bm \Pi_V :  \bm L^2(\Omega) \to \bm{V} _h$, and $\Pi_W :  L^2(\Omega) \to W_h$ satisfying
\begin{subequations}\label{projection_operator}
	\begin{align}
	(\bm{\Pi}_{\mathbb K}\mathbb L,\mathbb T)_K&=(\mathbb L,\mathbb T)_K  & \forall \ \mathbb T &\in[\mathcal P^{k}(K)]^{2\times 2},\label{projection_operator_a}\\
	(\bm\Pi_V\bm y,\bm v)_K&=(\bm y,\bm v)_K & \forall  \ \bm v &\in[\mathcal P^{k+1}(K)]^2,\label{projection_operator_b}\\
	(\Pi_Wp, w)_K&=(p, w)_K & \forall  \ w&\in \mathcal P^{k}(K).\label{projection_operator_c}
	\end{align}
\end{subequations}
For all faces $e$ of the simplex $K$, we also need the edge-wise $L^2$-orthogonal projections that map into $ \mathcal P^k(e)$ and $ [\mathcal P^k(e)]^2$, respectively:
\begin{subequations}\label{def_P_M}
	\begin{align}
	\langle  P_M  u- u,  \mu\rangle_e &= 0, \quad \forall \mu\in  \mathcal P^k(e),\\
	\langle \bm P_M \bm y-\bm y, \bm \mu\rangle_e &= 0, \quad \forall \bm\mu\in  [\mathcal P^k(e)]^2.
	\end{align}
\end{subequations}
In the analysis, we use the following classical results:
\begin{subequations}\label{standard_L2_projection}
	\begin{align}
	&\|\bm\Pi_{\mathbb K}\mathbb L-\mathbb L\|_{\mathcal T_h}\lesssim h^{s_{\mathbb L}} \|\mathbb L\|_{{s_{\mathbb L}},\Omega}, & 		&\|\bm\Pi_{\bm V}\bm y -\bm y\|_{\mathcal T_h} \lesssim h^{s_{\bm y}} \|\bm y\|_{{s_{\bm y}},\Omega},\\
	&\| P_M  u - u\|_{\partial \mathcal T_h} \lesssim h^{s_{\bm y}-\frac 1 2} \|\bm  y\|_{{s_{\bm y}},\Omega}, & 		&\|\bm\Pi_{\bm V}\bm y -\bm y\|_{\partial \mathcal T_h} \lesssim h^{s_{\bm y}-\frac 1 2} \|\bm y\|_{{s_{\bm y}},\Omega},\\
	&\|\Pi_W p -p\|_{\mathcal T_h}\lesssim h^{s_{p}} \|p\|_{{s_p},\Omega}, & 		&\|\bm P_M \bm y -\bm y\|_{\partial \mathcal T_h} \lesssim h^{s_{\bm y}-\frac 1 2} \|\bm y\|_{{s_{\bm y}},\Omega}.
	\end{align}
\end{subequations}
Similar projection error bounds hold for $\mathbb G$, $\bm z$ and $q$.

Define the HDG operator $\mathscr B: \mathbb K_h\times \bm  V_h \times W_h^0\times \bm M_h\times \mathbb K_h\times \bm  V_h \times W_h^0\times \bm M_h\to \mathbb R $ by
\begin{align}\label{def_B}
\begin{split}
\hspace{1em}&\hspace{-1em}  \mathscr B(\mathbb L_h,\bm y_h,p_h,\widehat{\bm y}_h^o; \mathbb T_1,\bm v_1, w_1,\bm \mu_1)\\
&=(\mathbb L_h,\mathbb T_1)_{\mathcal{T}_h}+(\bm y_h,\nabla\cdot\mathbb T_1)_{\mathcal{T}_h}-\langle \widehat{\bm y}_h^o, \mathbb T_1 \bm{n}\rangle_{\partial\mathcal{T}_h\backslash\varepsilon_h^\partial}\\
&\quad +(\mathbb L_h,\nabla\bm v_1)_{\mathcal T_h}- (p_h, \nabla\cdot \bm{v}_1)_{\mathcal{T}_h}-\langle \mathbb L_h\bm n-p_h\bm n-h^{-1}\bm P_M \bm y_h, \bm v_1\rangle_{\partial {\mathcal{T}_h}}\\
&\quad - \langle h^{-1} \widehat {\bm y}_h^o, \bm v_1\rangle_{\partial\mathcal T_h\backslash\varepsilon_h^\partial} -(\bm y_h,\nabla w_1)_{\mathcal T_h}+\langle \widehat{\bm y}_h^o\cdot\bm n,w_1\rangle_{\partial\mathcal T_h\backslash\varepsilon_h^\partial}\\
&\quad+\langle \mathbb L_h\bm n-p_h\bm n-h^{-1}(\bm P_M \bm y_h-\widehat{\bm y}_h^o), \bm \mu_1\rangle_{\partial {\mathcal{T}_h}\backslash\varepsilon_h^\partial}.
\end{split}
\end{align}
This definition allows us to rewrite the HDG formulation of the optimality system \eqref{HDG_discrete2}: find $(\mathbb L_h,\mathbb G_h,\bm y_h,\bm z_h, \\ p_h,q_h,\widehat {\bm y}_h^o,\widehat {\bm z}_h^o,\bm u_h)\in \mathbb K_h\times\mathbb K_h\times\bm{V}_h\times \bm{V}_h \times W_h^0 \times W_h^0\times \bm M_h(o)\times \bm M_h(o)\times M_h$ satisfying
\begin{subequations}\label{HDG_full_discrete}
	\begin{align}
	&\mathscr B(\mathbb L_h,\bm y_h,p_h,\widehat{\bm y}_h^o;\mathbb T_1,\bm v_1, w_1,\bm \mu_1)=\langle  u_h\bm \tau, \mathbb T_1 \bm{n} +h^{-1}\bm v_1\rangle_{\varepsilon_h^\partial}+(\bm f,\bm v_1)_{\mathcal T_h},\label{HDG_full_discrete_a}\\
	&\mathscr B(\mathbb G_h,\bm z_h,-q_h,\widehat{\bm z}_h^o;\mathbb T_2,\bm v_2, w_2,\bm \mu_2)=(\bm y_h - \bm y_d,\bm v_2)_{\mathcal T_h},\label{HDG_full_discrete_b}\\
	&\langle \mathbb G_h\bm n-h^{-1} \bm P_M\bm z_h,  \mu_3\bm \tau \rangle_{\varepsilon_h^\partial} =\gamma \langle   u_h,  \mu_3 \rangle_{\varepsilon_h^\partial},\label{HDG_full_discrete_e}
	\end{align}
\end{subequations}
for all $\left(\mathbb T_1,\mathbb T_2, \bm v_1,\bm v_2, w_1,w_2,\bm\mu_1,\bm{\mu}_2, \mu_3\right)\in \mathbb K_h\times\mathbb K_h\times\bm{V}_h\times \bm{V}_h \times W_h \times W_h\times \bm M_h(o)\times \bm M_h(o)\times  M_h$.  Also, in the error analysis we frequently use the following identity, which is established by applying integration by parts to \eqref{def_B}:
 \begin{align}\label{def_B2}
 \begin{split}
 \hspace{1em}&\hspace{-1em}  \mathscr B(\mathbb L_h,\bm y_h,p_h,\widehat{\bm y}_h^o; \mathbb T_1,\bm v_1, w_1,\bm \mu_1)\\
 &=(\mathbb L_h,\mathbb T_1)_{\mathcal{T}_h}+(\bm y_h, \nabla\cdot\mathbb T_1)_{\mathcal{T}_h}-\langle \widehat{\bm y}_h^o, \mathbb T_1 \bm{n}\rangle_{\partial\mathcal{T}_h\backslash\varepsilon_h^\partial}\\
 &\quad -(\nabla\cdot\mathbb L_h,\bm v_1)_{\mathcal T_h}+ (\nabla p_h, \bm{v}_1)_{\mathcal{T}_h}+\langle  h^{-1}\bm P_M \bm y_h, \bm v_1\rangle_{\partial {\mathcal{T}_h}}\\
 &\quad - \langle h^{-1} \widehat {\bm y}_h^o, \bm v_1\rangle_{\partial\mathcal T_h\backslash\varepsilon_h^\partial} -(\bm y_h, \nabla w_1)_{\mathcal T_h}+\langle  \widehat{\bm y}_h^o\cdot\bm n,w_1\rangle_{\partial\mathcal T_h\backslash\varepsilon_h^\partial}\\
 &\quad+\langle \mathbb L_h\bm n-p_h\bm n-h^{-1}(\bm P_M \bm y_h-\widehat{\bm y}_h^o), \bm \mu_1\rangle_{\partial {\mathcal{T}_h}\backslash\varepsilon_h^\partial}.
 \end{split}
 \end{align}

The detailed proofs of the following three lemmas can be found in the arXiv preprint of this paper; see \cite{GongHuMateosSinglerZhang1}.
\begin{lemma}\label{energy_norm}
	For any $ ( \mathbb T_h, \bm v_h, w_h, \bm \mu_h ) \in \mathbb K_h\times\bm V_h \times W_h^0 \times \bm M_h $, we have
	\begin{align*}
	\hspace{1em}&\hspace{-1em}  \mathscr B ( \mathbb T_h, \bm v_h, w_h, \bm \mu_h; \mathbb T_h, \bm v_h, w_h, \bm \mu_h) \\
	&= (\mathbb T_h,\mathbb T_h)_{\mathcal{T}_h} + \langle h^{-1}(\bm P_M \bm v_h-\bm \mu_h), \bm P_M \bm v_h - \bm \mu_h\rangle_{\partial {\mathcal{T}_h}\backslash\varepsilon_h^\partial} + \langle h^{-1}\bm P_M \bm v_h, \bm P_M\bm v_h \rangle_{\varepsilon_h^\partial}.
	\end{align*}
\end{lemma}
%

Next, we give another property of the HDG operator $\mathscr B$ that is fundamental to our analysis.
\begin{lemma}\label{identical_equa}
	For all $ (\mathbb L_h,\mathbb G_h,\bm y_h,\bm z_h, p_h,q_h,\widehat {\bm y}_h^o,\widehat {\bm z}_h^o)\in \mathbb K_h\times\mathbb K_h\times\bm{V}_h\times \bm{V}_h \times W_h^0 \times W_h^0\times \bm M_h(o)\times \bm M_h(o) $, we have
	\begin{align*}
	\mathscr B ( \mathbb L_h, \bm y_h, p_h, \widehat {\bm y}_h^o; -\mathbb G_h, \bm z_h, q_h, \widehat {\bm z}_h^o)+ \mathscr B (\mathbb G_h, \bm z_h, -q_h, \widehat {\bm z}_h^o;\mathbb L_h, -\bm y_h, p_h, -\widehat {\bm y}_h^o) = 0.
	\end{align*}
\end{lemma}

\begin{lemma}\label{ex_uni}
	There exists a unique solution of the HDG discretized optimality system \eqref{HDG_full_discrete}.
\end{lemma}

\subsection{Proof of the main result}
In our proof of the main result, we use the following auxiliary HDG problem: for the optimal control $ u $ fixed, find
\begin{align*}
(\mathbb L_h( u),\mathbb G_h( u),\bm y_h( u),\bm z_h( u), p_h( u),q_h( u),\widehat {\bm y}_h^o( u),\widehat {\bm z}_h^o( u))\\
\in \mathbb K_h\times\mathbb K_h\times \bm{V}_h\times  \bm V_h \times W_h^0 \times W_h^0\times \bm M_h(o)\times \bm M_h(o)
\end{align*}
such that
\begin{subequations}\label{HDG_inter_u}
	\begin{align}
	\mathscr B(\mathbb L_h( u),\bm y_h( u),p_h( u),\widehat{ y}_h^o( u);\mathbb T_1,\bm v_1, w_1,\bm \mu_1)&=(\bm f,\bm v_1)_{\mathcal T_h}+\langle  P_M u \bm \tau , h^{-1} \bm v_1+\mathbb T_1\bm n \rangle_{\varepsilon_h^\partial}, \label{HDG_u_a}\\
	\mathscr B(\mathbb G_h( u),\bm z_h( u),-q_h( u),\widehat{\bm z}_h^o( u);\mathbb T_2,\bm v_2, w_2,\bm \mu_2)&=(\bm y_h( u) - \bm y_d,\bm v_2)_{\mathcal T_h},\label{HDG_u_b}
	\end{align}
\end{subequations}
for all $\left(\mathbb T_1,\mathbb T_2, \bm v_1,\bm v_2, w_1,w_2,\bm\mu_1,\bm{\mu}_2\right)\in \mathbb K_h\times\mathbb K_h\times\bm{V}_h\times \bm{V}_h\times  W_h^0 \times W_h^0\times \bm M_h(o)\times \bm M_h(o)$.

We split the proof of the main result, \Cref{main_res}, into eleven steps.  We first consider the solution of the mixed form \eqref{TCOSE1}-\eqref{TCOASE3} of the optimality system, and the solution of the auxiliary problem.  We estimate the errors using $ L^2 $ projections.  Define
\begin{equation}\label{notation_1}
\begin{aligned}
\delta^{\mathbb{L}} &=\mathbb{L}-{\bm\Pi}_{\mathbb{K}}\mathbb{L},  &  \varepsilon^{\mathbb{L}}_h &= {\bm\Pi}_{\mathbb{K}}\mathbb{L}-\mathbb{L}_h( u),\\
\delta^{\bm y}&=\bm y- {\bm \Pi _V} \bm y, &  \varepsilon^{\bm y}_h &={\bm \Pi_{\bm V}}\bm y-\bm y_h( u),\\
\delta^{p} &= p- \Pi_W p,  & \varepsilon^p_h &= \Pi_W p-p_h( u),\\
\delta^{\widehat {\bm y}} &= \bm y-\bm P_{ M} \bm y,  &  \varepsilon^{\widehat {\bm y}}_h &=\bm P_{ M} \bm y-\widehat{\bm y}_h( u),\\
\widehat {\bm\delta}_1 &= \delta^{\mathbb{L}}\bm n-\delta^p \bm n-h^{-1} \bm P_{M} \delta^{\bm y},  & &
\end{aligned}
\end{equation}
where $\widehat {\bm y}_h( u) = \widehat {\bm y}_h^o(u)$ on $\varepsilon_h^o$ and $\widehat {\bm y}_h( u) =  P_M  u \bm \tau$ on $\varepsilon_h^{\partial}$, which gives $\varepsilon_h^{\widehat {\bm y}} = \bm 0$ on $\varepsilon_h^{\partial}$.
\subsubsection{Step 1: The error equation for part 1 of the auxiliary problem \eqref{HDG_u_a}.}
\label{subsec:proof_step1}	
\begin{lemma}\label{error_Lpz}
	We have
	\begin{equation} \label{error_eq_lyp}
	\begin{split}
	\mathscr B(\varepsilon_h^{\mathbb{L}},\varepsilon^{\bm y}_h,\varepsilon^p_h,\varepsilon^{\widehat{\bm y}}_h;\mathbb T_1,\bm v_1,w_1,\bm\mu_1)=\langle \widehat{\bm \delta}_1, \bm v_1\rangle_{\partial \mathcal{T}_h}-\langle \widehat{\bm \delta}_1, \bm \mu_1\rangle_{\partial \mathcal{T}_h\backslash \varepsilon_h^\partial}.
	\end{split}
	\end{equation}
\end{lemma}
\begin{proof}
	Using the definition of $\mathscr B$ \eqref{def_B} gives
	\begin{align*}
	\hspace{1em}&\hspace{-1em} \mathscr B(\bm\Pi_{\mathbb{K}} \mathbb L,\bm\Pi_{V} \bm y,\Pi_W p, \bm P_{ M} \bm y;\mathbb T_1,\bm v_1,w_1,\bm \mu_1)\\
	&=(\bm\Pi_{\mathbb{K}} \mathbb L,\mathbb T_1)_{\mathcal{T}_h}+(\bm\Pi_V \bm y,\nabla\cdot \mathbb T_1)_{\mathcal{T}_h}-\langle \bm P_{ M} \bm y, \mathbb T_1 \bm{n}\rangle_{\partial\mathcal{T}_h\backslash \varepsilon_h^\partial}+(\bm\Pi_{\mathbb{K}} \mathbb L,\nabla \bm v_1)_{\mathcal T_h}\\
	&\quad-(\Pi_W p, \nabla\cdot \bm v_1)_{\mathcal{T}_h} - \langle \bm \Pi_{\mathbb{K}} \mathbb{L} \bm n - \Pi_W p \bm n - h^{-1}\bm P_{M} \bm \Pi_{\bm V} \bm y, \bm v_1\rangle_{\partial {\mathcal{T}_h}} \\
	&\quad-\langle  h^{-1} \bm P_M \bm y, \bm v_1\rangle_{\partial\mathcal T_h\backslash\varepsilon_h^\partial} -(\bm\Pi_V \bm y,\nabla w_1)_{\mathcal T_h}+\langle \bm P_{ M} \bm y\cdot\bm n,w_1\rangle_{\partial\mathcal T_h\backslash \varepsilon_h^\partial}\\
	&\quad+\left\langle\bm \Pi_{\mathbb{K}} \mathbb{L} \bm n- \Pi_W p \bm n-h^{-1} \bm P_{ M} (\bm \Pi_{\bm V} \bm y -\bm y),\bm \mu_1\right\rangle_{\partial\mathcal T_h\backslash\varepsilon^{\partial}_h}.
	\end{align*}
	By properties of $L^2$ projections, we have
	\begin{align*}
	\hspace{1em}&\hspace{-1em}\mathscr B(\bm \Pi_{\mathbb{K}} \mathbb L,\Pi_V \bm y,\Pi_W p,P_{\bm M} \bm y;\mathbb T_1,\bm v_1,w_1,\bm \mu_1)\\
	&=(\mathbb L,\mathbb T_1)_{\mathcal{T}_h}+(\bm y,\nabla\cdot \mathbb T_1)_{\mathcal{T}_h}-\left\langle \bm y, \mathbb T_1 \bm{n}\right\rangle_{\partial\mathcal{T}_h\backslash \varepsilon_h^\partial}+(\mathbb L,\nabla \bm v_1)_{\mathcal T_h}\\
	&\quad -(p, \nabla\cdot \bm v_1)_{\mathcal{T}_h} - \langle \mathbb{L} \bm n- p \bm n -h^{-1} \bm P_{ M} \bm y, \bm v_1\rangle_{\partial {\mathcal{T}_h}} \\
	&\quad+\langle\delta^{ \mathbb L} \bm n-\delta^p \bm n-h^{-1} \bm P_{ M} \delta^{\bm y} ,\bm v_1\rangle_{\partial {\mathcal{T}_h}}-\langle  h^{-1}\bm P_M \bm y,\bm v_1\rangle_{\partial \mathcal T_h\backslash \varepsilon_h^\partial} \\
	&\quad-(\bm y,\nabla w_1)_{\mathcal T_h}+\langle \bm y\cdot\bm n,w_1\rangle_{\partial\mathcal T_h\backslash \varepsilon_h^\partial}+\langle  \mathbb{L} \bm n- p \bm n, \bm \mu_1\rangle_{\partial {\mathcal{T}_h}\backslash \varepsilon_h^\partial}\\
	&\quad+ \langle h^{-1} \bm P_M \delta^{\bm y}, \bm \mu_1\rangle_{\partial {\mathcal{T}_h}\backslash \varepsilon_h^\partial} - \langle \delta^{\mathbb L} \bm n -\delta^p \bm n,\bm \mu_1\rangle_{\partial {\mathcal{T}_h}\backslash \varepsilon_h^\partial}.
	\end{align*}
	The exact solution $ (\mathbb{L},\bm y,p) $ satisfies
	\begin{align*}
	(\mathbb L,\mathbb T_1)_{\mathcal{T}_h}+(\bm y,\nabla\cdot\mathbb T_1)_{\mathcal{T}_h}-\langle\bm y,\mathbb T_1\bm n\rangle_{\partial \mathcal{T}_h\backslash\varepsilon_h^\partial}&=\langle  u\bm \tau,\mathbb{T}_1 \bm n \rangle_{\varepsilon_h^\partial},\\
	(\mathbb{L}, \nabla \bm v_1)_{\mathcal{T}_h}-(p,\nabla \cdot \bm v_1)_{\mathcal{T}_h}- \langle\mathbb{L}\bm n-p\bm n, \bm v_1\rangle_{\partial \mathcal{T}_h}&= (\bm f,\bm v_1),  \\
	-(\bm y, \nabla w_1)_{\mathcal{T}_h}+\langle\bm y\cdot \bm n,w_1\rangle_{\partial \mathcal{T}_h\backslash\varepsilon_h^\partial}&=0,\\
	\langle \mathbb{L}\bm n-p \bm n,\bm \mu_1\rangle_{\partial \mathcal{T}_h\backslash \varepsilon_h^\partial}&=0,
	\end{align*}
	for all $(\mathbb T_1,\bm v_1,w_1,\bm \mu_1)\in \mathbb{K}_h \times \bm V_h\times W_h\times \bm M_h(o)$. Therefore,
	\begin{align*}
	\hspace{1em}&\hspace{-1em}\mathscr B(\bm \Pi_{\mathbb{K}} \mathbb L,\bm \Pi_V \bm y,\Pi_W p, \bm P_{ M} \bm y;\mathbb T_1,\bm v_1,w_1,\bm \mu_1)\\
	&=\langle  u\bm \tau, \mathbb T_1 \bm{n}\rangle_{ \varepsilon_h^\partial}+(\bm f,\bm v_1)_{\mathcal{T}_h}+\langle\delta^{ \mathbb L} \bm n-\delta^p \bm n-h^{-1} \bm P_{ M} \delta^{\bm y} ,\bm v_1\rangle_{\partial {\mathcal{T}_h}} \\
	&\quad +\langle h^{-1} \bm P_M \bm y, \bm v_1\rangle_{ \varepsilon_h^\partial}+ \langle h^{-1} \bm P_M \delta^{\bm y}, \bm \mu_1\rangle_{\partial {\mathcal{T}_h}\backslash \varepsilon_h^\partial} - \langle \delta^{\mathbb L} \bm n -\delta^p \bm n,\bm \mu_1\rangle_{\partial {\mathcal{T}_h}\backslash \varepsilon_h^\partial}\\
	&=(\bm f,\bm v_1)_{\mathcal{T}_h}+ \langle  (P_M u)\bm \tau, h^{-1} \bm v_1+\mathbb T_1\bm n\rangle_{\varepsilon_h^\partial}+\langle \widehat{\bm \delta}_1, \bm v_1\rangle_{\partial \mathcal{T}_h}-\langle \widehat{\bm \delta}_1, \bm \mu_1\rangle_{\partial \mathcal{T}_h\backslash \varepsilon_h^\partial}.
	\end{align*}
	Subtracting part 1 of the auxiliary problem \eqref{HDG_u_a} gives the result.
\end{proof}

\subsubsection{Step 2: Estimate for $\varepsilon_h^{\mathbb{L}}$.}
The proof of the following lemma is also given in the arXiv preprint of this paper, see \cite{GongHuMateosSinglerZhang1}.
\begin{lemma} \label{grad_y_h}
	We have
	\begin{align}
	\|\nabla \varepsilon_h^{\bm y}\|_{\mathcal T_h} + 	h^{-\frac 12}\|\varepsilon_h^{\bm y} - \varepsilon_h^{\widehat{\bm y}} \|_{\partial\mathcal T_h} \lesssim  \|\varepsilon_h^{\mathbb{L}}\|_{\mathcal{T}_h}+ h^{-\frac{1}{2}}\|\bm P_{ M} \varepsilon_h^{\bm y} -\varepsilon_h^{\widehat{\bm y}}\|_{\partial \mathcal{T}_h}.
	\end{align}
\end{lemma}
Next, we introduce the improved trace inequality.
\begin{lemma}\label{improved_trace_inequality}
	Let $e$ be a face of $K\in\mathcal T_h$. If $ \mathbb L-p\mathbb I \in  \mathbb H^{s}(\Omega)\cap \mathbb H(\textup{div}, \Omega)$ with $s>0$, then for all $\bm \mu\in [\mathcal {\bm P}^{k}(e)]^2$, we have
	\begin{align}
	\langle( \mathbb L-p\mathbb I)\bm n, \bm \mu \rangle_e \lesssim  h^{-1/2} \|\bm \mu\|_{e}(\|\mathbb L-p\mathbb I\|_{K}+h\|\nabla\cdot(\mathbb L-p\mathbb I)\|_K).
	\end{align}
\end{lemma}
The proof of this lemma can be found in \cite[Lemma 2.4]{MR3612900} for the vector case; the proof of the tensor case is trival.

\begin{lemma} \label{error_energy_norm}
		Let $\mathcal M =  h^{s_{\mathbb L}}\norm{\mathbb L}_{s^{\mathbb L},\Omega}+  h^{s_{p}}\norm{p}_{s^{p},\Omega} + h^{s_{\bm y}-1}\norm{\bm y}_{s^{\bm y},\Omega}$, we have
	\begin{align*}
	\|\varepsilon_h^{\mathbb L}\|_{\mathcal{T}_h}+h^{-\frac 1  2}\|{\bm P_M\varepsilon_h^{\bm y}-\varepsilon_h^{\widehat{\bm y}}}\|_{\partial \mathcal T_h} \lesssim \mathcal M.
	\end{align*}

\end{lemma}

\begin{proof}
	First, since $\varepsilon_h^{\widehat{\bm y}}=\bm 0$ on $\varepsilon_h^\partial$, the energy identity for $\mathscr B$ in  \Cref{energy_norm} gives
	\begin{align*}
	\mathscr B(\varepsilon_h^{\mathbb{L}},\varepsilon^{\bm y}_h,\varepsilon^p_h,\varepsilon^{\widehat{\bm y}}_h; \varepsilon_h^{\mathbb{L}},\varepsilon^{\bm y}_h,\varepsilon^p_h,\varepsilon^{\widehat{\bm y}}_h)=\|\varepsilon_h^{\mathbb{L}}\|^2_{\mathcal{T}_h} +h^{-1} \|\bm P_{ M} \varepsilon_h^{\bm y} -\varepsilon_h^{\widehat{\bm y}}\|^2_{\partial \mathcal{T}_h}.
	\end{align*}
	Then taking $(\mathbb T_1,\bm v_1,w_1,\bm \mu_1)=(\varepsilon_h^{\mathbb{L}},\varepsilon^{\bm y}_h,\varepsilon^p_h,\varepsilon^{\widehat{\bm y}}_h)$ in the error equation \eqref{error_eq_lyp} gives
	\begin{align*}
	\hspace{2em}&\hspace{-2em} \|\varepsilon_h^{\mathbb{L}}\|^2_{\mathcal{T}_h} +h^{-1} \| \bm P_{M} \varepsilon_h^{\bm y} -\varepsilon_h^{\widehat{\bm y}}\|^2_{\partial \mathcal{T}_h} \\
	&=\langle \widehat{\bm \delta}_1,  \varepsilon_h^{\bm y}-\varepsilon_h^{\widehat{\bm y}} \rangle_{\partial \mathcal{T}_h}\\
	&=\langle \delta^{\mathbb{L}}\bm n-\delta^p \bm n-h^{-1} \bm P_{M} \delta^{\bm y},  \varepsilon_h^{\bm y}-\varepsilon_h^{\widehat{\bm y}} \rangle_{\partial \mathcal{T}_h}\\
	&=\langle \delta^{\mathbb{L}}\bm n-\delta^p \bm n,  \varepsilon_h^{\bm y}-\varepsilon_h^{\widehat{\bm y}} \rangle_{\partial \mathcal{T}_h}  - \langle h^{-1}  \delta^{\bm y}, \bm P_{M}  \varepsilon_h^{\bm y}-\varepsilon_h^{\widehat{\bm y}} \rangle_{\partial \mathcal{T}_h} \\
	&\le Ch^{-\frac 1 2}\|\varepsilon_h^{\bm y}-\varepsilon_h^{\widehat{\bm y}}\|_{\partial \mathcal{T}_h}(\|\delta^{\mathbb L}-\delta^{p}\mathbb I\|_{\mathcal{T}_h} + h\|\nabla\cdot(\delta^{\mathbb L}-\delta^{p}\mathbb I)\|_{\mathcal{T}_h})\\
	&\quad + Ch^{-\frac 1 2}\|\delta^{\bm y}\|_{\partial \mathcal T_h} h^{-\frac 1 2}\|\bm P_{M}  \varepsilon_h^{\bm y}-\varepsilon_h^{\widehat{\bm y}} \|_{\partial \mathcal T_h}\\
	&\le Ch^{-\frac 1 2}\|\varepsilon_h^{\bm y}-\varepsilon_h^{\widehat{\bm y}}\|_{\partial \mathcal{T}_h}(\|\delta^{\mathbb L}\|_{\mathcal{T}_h}+\|\delta^{p}\|_{\mathcal{T}_h})\\
	&\quad  +Ch^{-\frac 1 2}\|\varepsilon_h^{\bm y}-\varepsilon_h^{\widehat{\bm y}}\|_{\partial \mathcal{T}_h}( h\|\nabla\cdot({\mathbb L}-p\mathbb I)\|_{\mathcal{T}_h} + h\|\nabla\cdot(\bm \Pi_{\mathbb K} (\mathbb L-p\mathbb I))\|_{\mathcal{T}_h})\\
	&\quad + Ch^{-\frac 1 2}\|\delta^{\bm y}\|_{\partial \mathcal T_h} h^{-\frac 1 2}\|\bm P_{M}  \varepsilon_h^{\bm y}-\varepsilon_h^{\widehat{\bm y}} \|_{\partial \mathcal T_h}\\
	&\le Ch^{-\frac 1 2}\|\varepsilon_h^{\bm y}-\varepsilon_h^{\widehat{\bm y}}\|_{\partial \mathcal{T}_h}(\|\delta^{\mathbb L}\|_{\mathcal{T}_h}+\|\delta^{p}\|_{\mathcal{T}_h})\\
	&\quad +Ch^{-\frac 1 2}\|\varepsilon_h^{\bm y}-\varepsilon_h^{\widehat{\bm y}}\|_{\partial \mathcal{T}_h} h\|\nabla\cdot({\mathbb L}-p\mathbb I)\|_{\mathcal{T}_h}\\
	&\quad   + Ch^{-\frac 1 2}\|\varepsilon_h^{\bm y}-\varepsilon_h^{\widehat{\bm y}}\|_{\partial \mathcal{T}_h}h\|\nabla\cdot(\bm \Pi_{\mathbb K} (\mathbb L-p\mathbb I)) -\nabla\cdot(\bm \Pi_{\mathbb K}^0 (\mathbb L-p\mathbb I)) \|_{\mathcal{T}_h}\\
	&\quad + Ch^{-\frac 1 2}\|\delta^{\bm y}\|_{\partial \mathcal T_h} h^{-\frac 1 2}\|\bm P_{M}  \varepsilon_h^{\bm y}-\varepsilon_h^{\widehat{\bm y}} \|_{\partial \mathcal T_h},
	\end{align*}
	where $\bm \Pi_{\mathbb K}^0$ is the $L^2$ projection into the space of piecewise constant functions and we used  \Cref{improved_trace_inequality}. Finally, we use Young's inequality, $0<s_{\mathbb L}, s_p<1$, the fact that $\mathbb L - p\mathbb I\in \mathbb{H}(\textup{div},\Omega)$ implies that $\|\nabla\cdot(\mathbb L - p\mathbb I)\|_{\mathcal{T}_h}$ is bounded independent of $h$,  and  \Cref{grad_y_h} to obtain
	\begin{align*}
	\hspace{2em}&\hspace{-2em} \|\varepsilon_h^{\mathbb L}\|_{\mathcal{T}_h}^2+h^{-1}\|{\bm P_M\varepsilon_h^{\bm y}-\varepsilon_h^{\widehat{\bm y}}}\|_{\partial \mathcal T_h}^2\\
	&\lesssim h^{-\frac 1 2}\|\varepsilon_h^{\bm y}-\varepsilon_h^{\widehat{\bm y}}\|_{\partial \mathcal{T}_h}(\|\delta^{\mathbb L}\|_{\mathcal{T}_h}+\|\delta^{p}\|_{\mathcal{T}_h}) +h^{-\frac 1 2}\|\varepsilon_h^{\bm y}-\varepsilon_h^{\widehat{\bm y}}\|_{\partial \mathcal{T}_h} h\|\nabla\cdot({\mathbb L}-p\mathbb I)\|_{\mathcal{T}_h}\\
	&\quad   + h^{-\frac 1 2}\|\varepsilon_h^{\bm y}-\varepsilon_h^{\widehat{\bm y}}\|_{\partial \mathcal{T}_h}\|\bm \Pi_{\mathbb K} (\mathbb L-p\mathbb I) -\bm \Pi_{\mathbb K}^0 (\mathbb L-p\mathbb I) \|_{\mathcal{T}_h} + h^{-\frac 1 2}\|\delta^{\bm y}\|_{\partial \mathcal T_h} h^{-\frac 1 2}\|\bm P_{M}  \varepsilon_h^{\bm y}-\varepsilon_h^{\widehat{\bm y}} \|_{\partial \mathcal T_h}\\
	& \lesssim\norm{\delta^{\mathbb L}}_{\mathcal T_h}^2 + \norm{\delta^{p}}_{\mathcal T_h}^2 +h^2\|\nabla\cdot({\mathbb L}-p\mathbb I)\|_{\mathcal{T}_h}^2 +\|\bm \Pi_{\mathbb K} (\mathbb L-p\mathbb I) -\bm \Pi_{\mathbb K}^0 (\mathbb L-p\mathbb I) \|_{\mathcal{T}_h}^2+ h^{-1}\norm{\delta^{\bm y}}_{\partial\mathcal T_h}^2\\
	&\lesssim h^{2s_{\mathbb L}}\norm{\mathbb L}_{s^{\mathbb L},\Omega}^2+  h^{2s_{p}}\norm{p}_{s^{p},\Omega}^2 + h^{2s_{\bm y}-2}\norm{\bm y}_{s^{\bm y},\Omega}^2.
	\end{align*}
\end{proof}

\subsubsection{Step 3: Estimate for $\varepsilon_h^p$.}
\begin{lemma}\label{pressure_estimate_p_1}
	We have $ \|{\varepsilon_h^p}\|_{\mathcal T_h} \lesssim \mathcal M $, where $\mathcal M$ is defined in \Cref{error_energy_norm}.
\end{lemma}
\begin{proof}
	We utilize an inf-sup proof strategy for the pressure; cf. \cite[Proposition 3.4]{MR2772094}, \cite[Lemma 5.3]{MR3556409}.  We know from \cite{FBrezzi_MFortin_book_1} that for any function $ \vartheta \in L^2 (\Omega)$ such that $(\vartheta,1)_\Omega=0$, we have
	\begin{align}\label{infsup}
	\| \vartheta\|_\Omega \lesssim \sup_{\bm v\in {\bm H}_0^1 (\Omega)\backslash\{0\}} \frac{(\vartheta,\nabla\cdot \bm v)_{\Omega}}{\norm{\bm v} _{\bm H ^1 (\Omega)}}.
	\end{align}
	Since
	\begin{align*}
	(\varepsilon_h^p,1)_{\mathcal T_h} = (\Pi_W p - p_h(\bm u),1)_{\mathcal T_h} = (\Pi_W p,1)_{\mathcal T_h} -(p_h( u),1)_{\mathcal T_h} =0,
	\end{align*}
	we can take $\vartheta:=\varepsilon_h^p$ in \eqref{infsup}. Then we have
	\begin{align*}
	\|\varepsilon_h^p \|_\Omega \lesssim \sup_{\bm v\in {\bm H}_0^1 (\Omega)\backslash\{0\}} \frac{(\varepsilon_h^p,\nabla\cdot \bm v)_{\Omega}}{\norm{\bm v} _{\bm H ^1 (\Omega)}},
	\end{align*}
	and
	\begin{align*}
	(\varepsilon_h^p ,\nabla\cdot \bm v)_\Omega =-(\nabla \varepsilon_h^p ,\bm \Pi_{ V} \bm v)_{\mathcal{T}_h}+\langle\varepsilon_h^p, \bm v\cdot \bm n\rangle_{\partial \mathcal{T}_h}.
	\end{align*}
	Next, taking  $(\mathbb T_1,\bm v_1, w_1,\bm \mu_1) = (0,\bm \Pi_{V} \bm v,0,0)$ in \Cref{error_Lpz} and using \eqref{def_B2} give
	\begin{align*}
	(\nabla \varepsilon_h^p ,\bm \Pi_{ V} \bm v)_{\mathcal{T}_h} =  (\nabla\cdot\varepsilon_h^{\mathbb L}, \bm \Pi_{ V} \bm v)_{\mathcal{T}_h}-\langle h^{-1}(\bm P_M\varepsilon_h^{\bm y}-\varepsilon_h^{\widehat{\bm y}}),\bm\Pi_{V}\bm v\rangle_{\partial \mathcal{T}_h} + \langle \widehat {\bm \delta}_1, \bm \Pi_{ V} \bm v \rangle_{\partial \mathcal T_h},
	\end{align*}
	where we used $\varepsilon_h^{\widehat {\bm y}} = 0$ on $\varepsilon_h^\partial$.
	The above two equalities give
	\begin{align*}
	(\varepsilon_h^p ,\nabla\cdot \bm v)_\Omega  &= - (\nabla\cdot\varepsilon_h^{\mathbb L}, \bm \Pi_{ V} \bm v)_{\mathcal{T}_h}+\langle h^{-1}(\bm P_M\varepsilon_h^{\bm y}-\varepsilon_h^{\widehat{\bm y}}),\bm\Pi_{V}\bm v\rangle_{\partial \mathcal{T}_h} +\langle\varepsilon_h^p, \bm v\cdot \bm n\rangle_{\partial \mathcal{T}_h} - \langle \widehat {\bm \delta}_1, \bm \Pi_{ V} \bm v \rangle_{\partial \mathcal T_h}\\
	&= - (\nabla\cdot\varepsilon_h^{\mathbb L}, \bm v)_{\mathcal{T}_h}+\langle h^{-1}(\bm P_M\varepsilon_h^{\bm y}-\varepsilon_h^{\widehat{\bm y}}),\bm\Pi_{V}\bm v\rangle_{\partial \mathcal{T}_h} +\langle\varepsilon_h^p, \bm v\cdot \bm n\rangle_{\partial \mathcal{T}_h}- \langle \widehat {\bm \delta}_1, \bm \Pi_{ V} \bm v \rangle_{\partial \mathcal T_h}\\
	&=  (\varepsilon_h^{\mathbb L}, \nabla \bm v)_{\mathcal{T}_h}+\langle h^{-1}(\bm P_M\varepsilon_h^{\bm y}-\varepsilon_h^{\widehat{\bm y}}),\bm\Pi_{V}\bm v\rangle_{\partial \mathcal{T}_h}+\langle  -\varepsilon_h^{\mathbb L} \bm n + \varepsilon_h^p\bm n, \bm P_M \bm v\rangle_{\partial \mathcal{T}_h}- \langle \widehat {\bm \delta}_1, \bm \Pi_{ V} \bm v \rangle_{\partial \mathcal T_h}.
	\end{align*}
	Next, we take $(\mathbb T_1,\bm v_1, w_1,\bm \mu_1) = (0, 0,0,\bm P_M\bm v )$ in \Cref{error_Lpz}.  Since $\bm v\in \bm H_0^1(\Omega)$ we have
	\begin{align*}
	\langle  \varepsilon_h^{\mathbb L} \bm n - \varepsilon_h^p\bm n, \bm P_M \bm v\rangle_{\partial \mathcal{T}_h} = \langle h^{-1}(\bm P_M\varepsilon_h^{\bm y}-\varepsilon_h^{\widehat{\bm y}}),\bm P_M\bm v\rangle_{\partial \mathcal{T}_h}-\langle \widehat{\bm \delta}_1, \bm P_M\bm v\rangle_{\partial \mathcal{T}_h}.
	\end{align*}
	This implies
	\begin{align*}
	(\varepsilon_h^p ,\nabla\cdot \bm v)_\Omega  &=  (\varepsilon_h^{\mathbb L}, \nabla \bm v)_{\mathcal{T}_h}+\langle h^{-1}(\bm P_M\varepsilon_h^{\bm y}-\varepsilon_h^{\widehat{\bm y}}),\bm\Pi_{V}\bm v - \bm P_M \bm v \rangle_{\partial \mathcal{T}_h} - \langle \widehat{\bm \delta}_1,\bm \Pi_V \bm v-  \bm P_M\bm v\rangle_{\partial \mathcal{T}_h}.
	\end{align*}
	
		For the above equality, the first two terms can be easily handled by Cauchy-Schwarz inequality. For the last term $- \langle \widehat{\bm \delta}_1,\bm \Pi_V \bm v-  \bm P_M\bm v\rangle_{\partial \mathcal{T}_h}$, we can use the same technique as in the proof of \Cref{error_energy_norm} to get
		\begin{align*}
		- \langle \widehat{\bm \delta}_1,\bm \Pi_V \bm v-  \bm P_M\bm v\rangle_{\partial \mathcal{T}_h}& \le Ch^{-\frac 1 2}\|\bm \Pi_V \bm v-  \bm P_M\bm v\|_{\partial \mathcal{T}_h}(\|\delta^{\mathbb L}\|_{\mathcal{T}_h}+\|\delta^{p}\|_{\mathcal{T}_h})\\
		&\quad+Ch^{-\frac 1 2}\|\bm \Pi_V \bm v-  \bm P_M\bm v\|_{\partial \mathcal{T}_h} h\|\nabla\cdot({\mathbb L}-p\mathbb I)\|_{\mathcal{T}_h}\\
		&\quad + Ch^{-\frac 1 2}\|\bm \Pi_V \bm v-  \bm P_M\bm v\|_{\partial \mathcal{T}_h} \|\bm \Pi_{\mathbb K} (\mathbb L-p\mathbb I) -\bm \Pi_{\mathbb K}^0 (\mathbb L-p\mathbb I) \|_{\mathcal{T}_h}\\
		&\quad + Ch^{-\frac 1 2}\|\delta^{\bm y}\|_{\partial \mathcal T_h} h^{-\frac 1 2}\|\bm \Pi_V \bm v-  \bm P_M\bm v \|_{\partial \mathcal T_h}.
		\end{align*}

		Applying the Cauchy-Schwarz inequality and using \Cref{error_energy_norm} give the desired result.
	
\end{proof}

\subsubsection{Step 4: Estimate for $\varepsilon_h^{ y}$ by a duality argument. }
For any $\bm \Theta \in \bm L^2(\Omega)$, the dual problem is given by
\begin{equation} \label{dual_pde}
\begin{split}
\mathbb{A}-\nabla\bm \Phi&=0\qquad\qquad~\text{in}\ \Omega,\\
-\nabla\cdot\mathbb{A}-\nabla\Psi&=\bm \Theta \qquad\text{in}\ \Omega,\\
\nabla\cdot\bm \Phi&=0\qquad\qquad~\text{in}\ \Omega,\\
\bm \Phi&=0\qquad\qquad~\text{on}\ \partial\Omega.
\end{split}
\end{equation}
Since the domain $\Omega$ is convex, we have the following regularity estimate
\begin{align}\label{regularity_dual}
\|\mathbb{A}\|_1+\|\bm \Phi\|_2+\|\Psi\|_1\le C\|\Theta\|_{\Omega}.
\end{align}

In the proof of the next lemma for estimating $\varepsilon_h^{\bm y}$, we use the following notation:
\begin{align}\label{notation_2}
\delta^{\mathbb{A}} &=\mathbb{A}-{\bm\Pi _{\mathbb K}} \mathbb{A}, \quad \delta^{\bm \Phi}=\bm \Phi- {\bm \Pi_{ V}} \bm \Phi, \quad
\delta^{\Psi}=\Psi - \Pi_W \Psi,  \quad \delta^{\widehat {\bm \Phi}} = \bm \Phi- \bm P_{M}\bm \Phi.
\end{align}

	\begin{lemma}\label{error_y_lyp}
		Let $\mathcal M$  be defined as in \Cref{error_energy_norm}. Then we have
		\begin{align} \label{error_yu}
		\|\varepsilon_h^{\bm y}\|_{\mathcal T_h} \lesssim h\mathcal M.
		\end{align}
	\end{lemma}

\begin{proof}
	We consider the dual problem \eqref{dual_pde} with $\bm\Theta = \varepsilon_h^{\bm y}$.  In the definition of $ \mathscr{B} $ in \eqref{def_B}  and \eqref{def_B2}, we take  $(\mathbb T_1,\bm v_1,w_1,\bm \mu_1) = ( -\bm \Pi_{\mathbb{K}} \mathbb{A},  \bm \Pi_{ V} \bm \Phi,  \Pi_W \Psi, \bm P_{M} \bm \Phi)$.  Since $\bm \Phi=0$ on $\varepsilon_h^\partial$, $\varepsilon_h^{\widehat {\bm y}} = 0$ on $\varepsilon_h^\partial$,   $\nabla\cdot \bm \Phi =0$,  by using integration by parts we have
	\begin{align*}
	\hspace{1em}&\hspace{-1em}  \mathscr B(\varepsilon_h^{\mathbb{L}},\varepsilon^y_h,\varepsilon^p_h,\varepsilon^{\widehat{y}}_h;-\bm \Pi_{\mathbb{K}} \mathbb{A},\bm \Pi_{ V} \bm \Phi,\bm \Pi_W \Psi, \bm P_{\bm M} \bm \Phi)\\
	&= -(\varepsilon_h^{\mathbb{L}},\bm \Pi_{\mathbb{K}} \mathbb{A})_{\mathcal{T}_h}-(\varepsilon^{\bm y}_h,\nabla\cdot \bm \Pi_{\mathbb{K}} \mathbb{A})_{\mathcal{T}_h}+\langle \varepsilon^{\widehat{\bm y}}_h, \bm \Pi_{\mathbb{K}} \mathbb{A} \bm{n}\rangle_{\partial\mathcal{T}_h}-(\nabla\cdot\varepsilon_h^{\mathbb{L}}, \bm \Pi_{V} \bm \Phi)_{\mathcal T_h}-(\varepsilon^p_h, \nabla\cdot \bm \Pi_{ V} \bm \Phi)_{\mathcal{T}_h}\\
	&\quad +\langle \varepsilon^p_h\bm n+h^{-1}(\bm P_{ M} \varepsilon^{\bm y}_h -\varepsilon^{\widehat{\bm y}}_h), \bm \Pi_{ V} \bm \Phi\rangle_{\partial {\mathcal{T}_h}} - (\varepsilon^{\bm y}_h,\nabla  \Pi_W \Psi)_{\mathcal T_h} +\langle \varepsilon^{\widehat{\bm y}}_h\cdot\bm n,\Pi_W \Psi\rangle_{\partial\mathcal T_h}\\
	&\quad +\langle\varepsilon_h^{\mathbb{L}} \bm n- \varepsilon^p_h \bm n-h^{-1} (\bm P_{ M} \varepsilon^{\bm y}_h-\varepsilon^{\widehat{\bm y}}_h),\bm P_{ M} \bm \Phi \rangle_{\partial\mathcal T_h}\\
	&=-(\varepsilon_h^{\mathbb{L}},\mathbb{A})_{\mathcal{T}_h}+(\varepsilon^{\bm y}_h,\nabla\cdot  \delta^{\mathbb{A}})_{\mathcal{T}_h}-(\varepsilon^{\bm y}_h,\nabla\cdot  \mathbb{A})_{\mathcal{T}_h}-\langle \varepsilon^{\widehat{\bm y}}_h, \delta^{\mathbb{A}} \bm{n}\rangle_{\partial\mathcal{T}_h} - \langle \varepsilon_h^{\mathbb L} \bm n, \bm \Phi\rangle_{\partial \mathcal T_h}\\
	&\quad + (\varepsilon_h^{\mathbb L},\nabla\bm{\Phi})_{\mathcal T_h}+(\varepsilon^p_h, \nabla\cdot \delta^{\bm \Phi})_{\mathcal{T}_h} +  \langle \varepsilon^p_h\bm n+h^{-1}(\bm P_{ M} \varepsilon^{\bm y}_h -\varepsilon^{\widehat{\bm y}}_h), \bm \Phi\rangle_{\partial {\mathcal{T}_h}}\\
	&\quad-\langle  \varepsilon^p_h\bm n+h^{-1}(\bm P_{ M} \varepsilon^{\bm y}_h -\varepsilon^{\widehat{\bm y}}_h), \delta^{\bm \Phi}\rangle_{\partial {\mathcal{T}_h}}-(\varepsilon^{\bm y}_h,\nabla \Psi)_{\mathcal T_h} + (\varepsilon^{\bm y}_h,\nabla \delta^\Psi)_{\mathcal T_h}\\
	&\quad-\langle \varepsilon^{\widehat{\bm y}}_h\cdot\bm n,\delta^\Psi\rangle_{\partial\mathcal T_h}+\langle\varepsilon_h^{\mathbb{L}} \bm n- \varepsilon^p_h \bm n-h^{-1} (\bm P_{ M} \varepsilon^{\bm y}_h-\varepsilon^{\widehat{\bm y}}_h),\bm \Phi \rangle_{\partial\mathcal T_h}\\
	&=-(\varepsilon_h^{\mathbb{L}},\mathbb{A}-\nabla \bm\Phi)_{\mathcal{T}_h}+(\varepsilon^{\bm y}_h,\nabla\cdot  \delta^{\mathbb{A}})_{\mathcal{T}_h}-(\varepsilon^{\bm y}_h,\nabla\cdot  \mathbb{A}+\nabla\Psi)_{\mathcal{T}_h}-\langle \varepsilon^{\widehat{\bm y}}_h, \delta^{\mathbb{A}} \bm{n}\rangle_{\partial\mathcal{T}_h}+(\varepsilon^p_h, \nabla\cdot \delta^{\bm \Phi})_{\mathcal{T}_h} \\
	&\quad -\langle  \varepsilon^p_h\bm n+h^{-1}(\bm P_{ M} \varepsilon^{\bm y}_h -\varepsilon^{\widehat{\bm y}}_h), \delta^{\bm \Phi}\rangle_{\partial {\mathcal{T}_h}}+ (\varepsilon^{\bm y}_h,\nabla \delta^\Psi)_{\mathcal T_h}-\langle \varepsilon^{\widehat{\bm y}}_h\cdot\bm n,\delta^\Psi\rangle_{\partial\mathcal T_h}\\
	&=(\varepsilon^{\bm y}_h,\nabla\cdot  \delta^{\mathbb{A}})_{\mathcal{T}_h} + \|\varepsilon_h^{\bm y}\|_{\mathcal T_h}^2-\langle \varepsilon^{\widehat{\bm y}}_h, \delta^{\mathbb{A}} \bm{n}\rangle_{\partial\mathcal{T}_h} +(\varepsilon^p_h, \nabla\cdot \delta^{\bm \Phi})_{\mathcal{T}_h} \\
	&\quad -\langle  \varepsilon^p_h\bm n+h^{-1}(\bm P_{ M} \varepsilon^{\bm y}_h -\varepsilon^{\widehat{\bm y}}_h), \delta^{\bm \Phi}\rangle_{\partial {\mathcal{T}_h}}  + (\varepsilon^{\bm y}_h,\nabla \delta^\Psi)_{\mathcal T_h} -\langle \varepsilon^{\widehat{\bm y}}_h\cdot\bm n,\delta^\Psi\rangle_{\partial\mathcal T_h}.
	\end{align*}
	Here we used $\varepsilon^{\widehat {\bm y}}_h=0$ on $\varepsilon^{\partial}_h$, $\mathbb A + \Psi\mathbb I \in \mathbb H(\textup{div}, \Omega)$, and $\langle\varepsilon^{\widehat {\bm y}}_h, (\mathbb{A} + \Psi\mathbb I)\bm n\rangle_{\partial\mathcal T_h}=0$, which holds since $\varepsilon^{\widehat {\bm y}}_h$ is a single-valued function on interior edges and $\varepsilon^{\widehat {\bm y}}_h=0$ on $\varepsilon^{\partial}_h$.
	
	Next, integrate by parts to obtain
	\begin{align*}
	(\varepsilon^{\bm y}_h,\nabla\cdot  \delta^{\mathbb{A}})_{\mathcal{T}_h}&=-(\nabla \varepsilon^{\bm y}_h,\delta^{\mathbb{A}})_{\mathcal{T}_h}+\langle \varepsilon^{\bm y}_h,\delta^{\mathbb{A}}\bm n \rangle_{\partial \mathcal{T}_h}=\langle \varepsilon^{\bm y}_h,\delta^{\mathbb{A}}\bm n \rangle_{\partial \mathcal{T}_h},\\
	(\varepsilon^p_h, \nabla\cdot \delta^{\bm \Phi})_{\mathcal{T}_h}&=-(\nabla \varepsilon^p_h, \delta^{\bm \Phi})_{\mathcal{T}_h}+\langle \varepsilon^p_h, \delta^{\bm \Phi}\cdot\bm n\rangle_{\partial \mathcal{T}_h}=\langle \varepsilon^p_h, \delta^{\bm \Phi}\cdot\bm n\rangle_{\partial \mathcal{T}_h},\\
	(\varepsilon_h^{\bm y},\nabla \delta^{\Psi})_{\mathcal{T}_h}&=\langle \varepsilon_h^{\bm y}\cdot \bm n,\delta^{\Psi} \rangle_{\partial \mathcal{T}_h} -(\nabla\cdot \varepsilon_h^{\bm y},\delta^{\Psi})_{ \mathcal{T}_h}=\langle \varepsilon_h^{\bm y}\cdot \bm n,\delta^{\Psi} \rangle_{\partial \mathcal{T}_h}.
	\end{align*}
	Then
	\begin{align}\label{first_eq}
	\begin{split}
	\hspace{1em}&\hspace{-1em}  \mathscr B(\varepsilon_h^{\mathbb{L}},\varepsilon^{\bm y}_h,\varepsilon^p_h,\varepsilon^{\widehat{\bm y}}_h;-\bm \Pi_{\mathbb{K}} \mathbb{A},\bm\Pi_{ V} \bm \Phi, \Pi_W \Psi,\bm P_{ M} \bm \Phi)\\
	&=\|\varepsilon_h^{\bm y}\|_{\mathcal{T}_h}^2 -\langle h^{-1}(\bm P_{ M} \varepsilon_h^{\bm y} -\varepsilon_h^{\widehat{\bm y}}) , \delta^{\bm \Phi} \rangle_{\partial \mathcal {T}_h} +\langle \varepsilon_h^{\bm y} -\varepsilon_h^{\widehat{\bm y}} ,\delta^{\mathbb{A}} \bm n+\delta^{\Psi}\bm n\rangle_{\partial \mathcal{T}_h}.
	\end{split}
	\end{align}
	On the other hand, using $\bm \Phi =0 $ on $\varepsilon_h^\partial$ and the error equation \eqref{error_eq_lyp} gives
	\begin{align}\label{second_eq}
	\mathscr B(\varepsilon_h^{\mathbb{L}},\varepsilon^{\bm y}_h,\varepsilon^p_h,\varepsilon^{\widehat{\bm y}}_h;-\bm \Pi_{\mathbb{K}} \mathbb{A},\bm\Pi_{ V} \bm \Phi, \Pi_W \Psi,\bm P_{ M} \bm \Phi)=\langle \widehat{\bm \delta}_1, \bm \Pi_{ V} \bm \Phi - \bm P_{ M} \bm \Phi \rangle_{\partial \mathcal{T}_h}.
	\end{align}

		Next, we split the proof into three cases.
		
		(1) When $s_{\mathbb L} = s_p>1/2$,  we note that
		\begin{align*}
		  \langle \widehat{\bm \delta}_1, \bm P_{ M} \bm \Phi \rangle_{\partial \mathcal{T}_h}
		&= \langle -\bm \Pi_{\mathbb K} \mathbb{L} \bm n +\Pi_W p \bm n+h^{-1} \bm P_{ M} (\bm \Pi_{ V} \bm y -\bm y), \bm P_{ M} \bm \Phi \rangle_{\partial \mathcal{T}_h} +\langle \mathbb{L} \bm n-p\bm n, \bm P_{ M} \bm \Phi \rangle_{\partial \mathcal{T}_h}\\
		&=\langle -\bm \Pi_{\mathbb K} \mathbb{L} \bm n +\Pi_W p \bm n +h^{-1} \bm P_{ M} (\bm \Pi_{ V} \bm y -\bm y), \bm \Phi \rangle_{\partial \mathcal{T}_h} +\langle \mathbb{L} \bm n-p\bm n, \bm \Phi \rangle_{\partial \mathcal{T}_h}\\
		&=\langle \widehat{\bm \delta}_1, \bm \Phi \rangle_{\partial \mathcal{T}_h}.
		\end{align*}
		Here we used the fact $\langle \mathbb{L} \bm n-p\bm n, \bm \Phi \rangle_{\partial \mathcal{T}_h} =0 $ and  $\langle \mathbb{L} \bm n-p\bm n, \bm P_{M} \bm \Phi \rangle_{\partial \mathcal{T}_h}=0$ since $\mathbb L - p\mathbb I \in \mathbb H(\textup{div},\Omega)$ and $\bm \Phi = \bm 0$ on $\varepsilon_h^\partial$.
		
		Compare the equations \eqref{first_eq} and \eqref{second_eq} to obtain
		\begin{align*}
		\|\varepsilon_h^{\bm y}\|_{\mathcal{T}_h}^2&=-\langle \varepsilon_h^{\bm y} -\varepsilon_h^{\widehat{\bm y}} ,\delta^{\mathbb{A}} \bm n+\delta^{\Psi}\bm n\rangle_{\partial \mathcal{T}_h} - \langle \widehat{\bm \delta}_1,\delta^{\bm \Phi} \rangle_{\partial \mathcal{T}_h} +\langle h^{-1}(\bm P_{ M} \varepsilon_h^{\bm y} -\varepsilon_h^{\widehat{\bm y}}) , \delta^{\bm \Phi} \rangle_{\partial \mathcal {T}_h} \\
		&\lesssim h^{-\frac 1 2}\|\varepsilon_h^{\bm y} -\varepsilon_h^{\widehat{\bm y}}\|_{\partial \mathcal {T}_h} h^{\frac 1 2}(\|\delta^{\mathbb A}\|_{\partial \mathcal {T}_h} +\|\delta^{\Psi}\|_{\partial \mathcal {T}_h} ) + \|\widehat{\bm \delta}_1\|_{\partial \mathcal {T}_h} \|\delta^{\bm \Phi}\|_{\partial \mathcal {T}_h} + h^{-\frac 1 2}\|\bm P_M \varepsilon_h^{\bm y} -\varepsilon_h^{\widehat{\bm y}}\|_{\partial \mathcal {T}_h}  h^{-\frac 1 2 }\|\delta^{\bm \Phi}\|_{\partial \mathcal {T}_h}\\
		&\lesssim (h^{s_{\mathbb L}+1}\norm{\mathbb L}_{s^{\mathbb L},\Omega}+  h^{s_{p}+1}\norm{p}_{s^{p},\Omega} + h^{s_{\bm y}}\norm{\bm y}_{s^{\bm y},\Omega})\|\varepsilon_h^{\bm y}\|_{\mathcal{T}_h},
		\end{align*}
		where in the last step we used  \Cref{grad_y_h,error_energy_norm}, the standard projection error estimates (\ref{standard_L2_projection}) and the a priori estimate (\ref{regularity_dual}) for the dual problem.

	(2) When $0<s_{\mathbb L} = s_p \le 1/2$, hence the quantity $\|(\mathbb L-p\mathbb I)n \|_{\partial \mathcal T_h}$ is not well defined and the analysis in (1) is not applicable. We need to refine the analysis.
	
	\hspace{1em} (i) If $k\ge 1$, compare the equations \eqref{first_eq} and \eqref{second_eq}, and use the same technique as in the proof of \Cref{error_energy_norm} to obtain
	\begin{align*}
	\|\varepsilon_h^{\bm y}\|_{\mathcal{T}_h}^2&=-\langle \varepsilon_h^{\bm y} -\varepsilon_h^{\widehat{\bm y}} ,\delta^{\mathbb{A}} \bm n+\delta^{\Psi}\bm n\rangle_{\partial \mathcal{T}_h} {+ \langle \widehat{\bm \delta}_1,\bm \Pi_{ V} \bm \Phi - \bm P_{ M} \bm \Phi\rangle_{\partial \mathcal{T}_h}} +\langle h^{-1}(\bm P_{ M} \varepsilon_h^{\bm y} -\varepsilon_h^{\widehat{\bm y}}) , \delta^{\bm \Phi} \rangle_{\partial \mathcal {T}_h} \\
	&=-\langle \varepsilon_h^{\bm y} -\varepsilon_h^{\widehat{\bm y}} ,\delta^{\mathbb{A}} \bm n+\delta^{\Psi}\bm n\rangle_{\partial \mathcal{T}_h} {+ \langle \delta^{\mathbb{L}}\bm n-\delta^p \bm n,\bm \Pi_{ V} \bm \Phi - \bm P_{ M} \bm \Phi\rangle_{\partial \mathcal{T}_h}}\\
	&\quad  {-\langle h^{-1} \bm P_{M} \delta^{\bm y},\bm \Pi_{ V} \bm \Phi - \bm P_{ M} \bm \Phi\rangle_{\partial \mathcal{T}_h}}+\langle h^{-1}(\bm P_{ M} \varepsilon_h^{\bm y} -\varepsilon_h^{\widehat{\bm y}}) , \delta^{\bm \Phi} \rangle_{\partial \mathcal {T}_h} \\
	&\lesssim h^{-\frac 1 2}\|\varepsilon_h^{\bm y} -\varepsilon_h^{\widehat{\bm y}}\|_{\partial \mathcal {T}_h} h^{\frac 1 2}(\|\delta^{\mathbb A}\|_{\partial \mathcal {T}_h} +\|\delta^{\Psi}\|_{\partial \mathcal {T}_h} ) \\
	&\quad +\|\bm \Pi_{ V} \bm \Phi - \bm P_{ M} \bm \Phi\|_{\partial \mathcal {T}_h} (\|\delta^{\mathbb L}\|_{\mathcal{T}_h}+\|\delta^{p}\|_{\mathcal{T}_h} + h\|\nabla\cdot({\mathbb L}-p\mathbb I)\|_{\mathcal{T}_h})\\
	&\quad +\|\bm \Pi_{ V} \bm \Phi - \bm P_{ M} \bm \Phi\|_{\partial \mathcal {T}_h} \|\bm \Pi_{\mathbb K} (\mathbb L-p\mathbb I) -\bm \Pi_{\mathbb K}^0 (\mathbb L-p\mathbb I) \|_{\mathcal{T}_h}\\
	&\quad + h^{-1} \|\bm P_{M} \delta^{\bm y}\|_{\partial \mathcal T_h}\|\bm \Pi_{ V} \bm \Phi - \bm P_{ M} \bm \Phi\|_{\partial \mathcal {T}_h}  + h^{-\frac 1 2}\|\bm P_M \varepsilon_h^{\bm y} -\varepsilon_h^{\widehat{\bm y}}\|_{\partial \mathcal {T}_h}  h^{-\frac 1 2 }\|\delta^{\bm \Phi}\|_{\partial \mathcal {T}_h}\\
	&\lesssim (h^{s_{\mathbb L}+1}\norm{\mathbb L}_{s^{\mathbb L},\Omega}+  h^{s_{p}+1}\norm{p}_{s^{p},\Omega} + h^{s_{\bm y}}\norm{\bm y}_{s^{\bm y},\Omega})\|\varepsilon_h^{\bm y}\|_{\mathcal{T}_h}.
	\end{align*}
	
	\hspace{1em} (ii) If $k= 0$, we note that
		\begin{align*}
		  \langle \widehat{\bm \delta}_1, \bm P_{ M} \bm \Phi \rangle_{\partial \mathcal{T}_h}&= \langle -\bm \Pi_{\mathbb K} \mathbb{L} \bm n +\Pi_W p \bm n+h^{-1} \bm P_{ M} (\bm \Pi_{ V} \bm y -\bm y), \bm P_{ M} \bm \Phi \rangle_{\partial \mathcal{T}_h} +\langle \mathbb{L} \bm n-p\bm n, \bm P_{ M} \bm \Phi \rangle_{\partial \mathcal{T}_h}\\
		&=\langle -\bm \Pi_{\mathbb K} \mathbb{L} \bm n +\Pi_W p \bm n +h^{-1} \bm P_{ M} (\bm \Pi_{ V} \bm y -\bm y), \bm \Phi \rangle_{\partial \mathcal{T}_h} +\langle \mathbb{L} \bm n-p\bm n, \bm \Phi \rangle_{\partial \mathcal{T}_h}\\
		&=\langle \widehat{\bm \delta}_1, \bm \Phi \rangle_{\partial \mathcal{T}_h}.
		\end{align*}
		Here we used the fact $\langle \mathbb{L} \bm n-p\bm n, \bm \Phi \rangle_{\partial \mathcal{T}_h} =0 $ and  $\langle \mathbb{L} \bm n-p\bm n, \bm P_{M} \bm \Phi \rangle_{\partial \mathcal{T}_h}=0$ since $\mathbb L - p\mathbb I \in \mathbb H(\textup{div},\Omega)$ and $\bm \Phi = \bm 0$ on $\varepsilon_h^\partial$.
		
		Compare the equations \eqref{first_eq} and \eqref{second_eq} to obtain
		\begin{align*}
		\|\varepsilon_h^{\bm y}\|_{\mathcal{T}_h}^2&=-\langle \varepsilon_h^{\bm y} -\varepsilon_h^{\widehat{\bm y}} ,\delta^{\mathbb{A}} \bm n+\delta^{\Psi}\bm n\rangle_{\partial \mathcal{T}_h} - \langle \widehat{\bm \delta}_1,\delta^{\bm \Phi} \rangle_{\partial \mathcal{T}_h} +\langle h^{-1}(\bm P_{ M} \varepsilon_h^{\bm y} -\varepsilon_h^{\widehat{\bm y}}) , \delta^{\bm \Phi} \rangle_{\partial \mathcal {T}_h} \\
		& = -\langle \varepsilon_h^{\bm y} -\varepsilon_h^{\widehat{\bm y}} ,\delta^{\mathbb{A}} \bm n+\delta^{\Psi}\bm n\rangle_{\partial \mathcal{T}_h} - \langle  \delta^{\mathbb{L}}\bm n-\delta^p \bm n, \delta^{\bm \Phi} \rangle_{\partial \mathcal{T}_h} \\
		&\quad + \langle  h^{-1} \bm P_{M} \delta^{\bm y}, \delta^{\bm \Phi} \rangle_{\partial \mathcal{T}_h}+\langle h^{-1}(\bm P_{ M} \varepsilon_h^{\bm y} -\varepsilon_h^{\widehat{\bm y}}) , \delta^{\bm \Phi} \rangle_{\partial \mathcal {T}_h} \\
		& = -\langle \varepsilon_h^{\bm y} -\varepsilon_h^{\widehat{\bm y}} ,\delta^{\mathbb{A}} \bm n+\delta^{\Psi}\bm n\rangle_{\partial \mathcal{T}_h} - (\delta^{\mathbb{L}}-\delta^p \mathbb I, \nabla \delta^{\bm \Phi})_{\mathcal T_h} - (\nabla\cdot(\delta^{\mathbb{L}}-\delta^p \mathbb I),  \delta^{\bm \Phi})_{\mathcal T_h}\\
		&\quad + \langle  h^{-1} \bm P_{M} \delta^{\bm y}, \delta^{\bm \Phi} \rangle_{\partial \mathcal{T}_h}+\langle h^{-1}(\bm P_{ M} \varepsilon_h^{\bm y} -\varepsilon_h^{\widehat{\bm y}}) , \delta^{\bm \Phi} \rangle_{\partial \mathcal {T}_h} \\
		& = -\langle \varepsilon_h^{\bm y} -\varepsilon_h^{\widehat{\bm y}} ,\delta^{\mathbb{A}} \bm n+\delta^{\Psi}\bm n\rangle_{\partial \mathcal{T}_h} - (\delta^{\mathbb{L}}-\delta^p \mathbb I, \nabla \delta^{\bm \Phi})_{\mathcal T_h} - (\nabla\cdot(\mathbb{L}- p \mathbb I),  \delta^{\bm \Phi})_{\mathcal T_h}\\
		&\quad + \langle  h^{-1} \bm P_{M} \delta^{\bm y}, \delta^{\bm \Phi} \rangle_{\partial \mathcal{T}_h}+\langle h^{-1}(\bm P_{ M} \varepsilon_h^{\bm y} -\varepsilon_h^{\widehat{\bm y}}) , \delta^{\bm \Phi} \rangle_{\partial \mathcal {T}_h} \\
		&\lesssim (h^{s_{\mathbb L}+1}\norm{\mathbb L}_{s^{\mathbb L},\Omega}+  h^{s_{p}+1}\norm{p}_{s^{p},\Omega} + h^{s_{\bm y}}\norm{\bm y}_{s^{\bm y},\Omega})\|\varepsilon_h^{\bm y}\|_{\mathcal{T}_h}.
		\end{align*}
\end{proof}

\Cref{error_energy_norm}, \Cref{pressure_estimate_p_1}, \Cref{error_y_lyp}, and the triangle inequality yield optimal convergence rates for $\|\mathbb L- \mathbb L_h( u)\|_{\mathcal T_h}$, $\| p - p_h( u)\|_{\mathcal T_h}$ and $\| \bm y - \bm y_h( u)\|_{\mathcal T_h}$:

	\begin{lemma}\label{lemma:step4_conv_rates}
		Let $\mathcal M$ be defined as in \Cref{error_energy_norm}. Then we have
		\begin{align*}
		h\|\mathbb{L} -\mathbb{L}_h( u)\|_{\mathcal{T}_h}+h\|p -p_h( u)\|_{\mathcal{T}_h} +	\|\bm y-\bm y_h( u)\|_{\mathcal{T}_h}\lesssim h\mathcal M.
		\end{align*}
	\end{lemma}

\subsubsection{Step 5: The error equation for part 2 of the auxiliary problem \eqref{HDG_u_b}.}
We continue to focus on the solution of the auxiliary problem and the solution of the mixed formultion \eqref{TCOSE1}-\eqref{TCOASE3}  of the optimality system.  Next, we consider the dual variables, i.e., $\mathbb{G}$, $\bm z$ and $q$.  We estimate the errors using $ L^2 $ projections, and we use the following notation.
\begin{equation}\label{notation_3}
\begin{aligned}
\delta^{\mathbb{G}} &=\mathbb{G}-{\bm \Pi_{\mathbb{K}}} \mathbb{G},  &  \varepsilon^{\mathbb{G}}_h &={\bm \Pi_{\mathbb{K}}} \mathbb{G}-\mathbb{G}_h( u),\\
\delta^{\bm z}&=\bm z- \bm{\Pi_V} \bm z, &  \varepsilon^{\bm z}_h &= \bm{\Pi_V} \bm z-\bm z_h( u),\\
\delta^q&=q- {\Pi_W} q, &  \varepsilon^{q}_h &= {\Pi_W} q - q_h( u),\\
\delta^{\widehat {\bm z}} &= \bm z- \bm P_{ M} \bm z,  &  \varepsilon^{\widehat {\bm z}}_h &= \bm P_{ M} \bm z-\widehat{\bm z}_h( u),\\
\widehat {\bm \delta}_2 &= \delta^{\mathbb{G}}\bm n+\delta^q \bm n- h^{-1} \bm P_{ M} \delta^{\bm z}, & &
\end{aligned}
\end{equation}
where $\widehat {\bm z}_h( u) = \widehat {\bm z}_h^o(u)$ on $\varepsilon_h^o$ and $\widehat {\bm z}_h( u) =  0$ on $\varepsilon_h^{\partial}$, which gives $\varepsilon_h^{\widehat {\bm z}} = \bm 0$ on $\varepsilon_h^{\partial}$.

The proof of the following result is similar to the proof of \Cref{error_Lpz}, and is omitted.
\begin{lemma}\label{error_eq_Gqz}
	We have
	\begin{align} \label{error_B_2}
	 \mathscr B(\varepsilon_h^{\mathbb{G}},\varepsilon_h^{\bm z},-\varepsilon^q_h,\varepsilon_h^{\widehat{\bm z}};\mathbb T_2,\bm v_2,w_2,\bm\mu_2)=\langle\widehat{\bm \delta}_2,\bm v_2\rangle_{\partial \mathcal{T}_h}-\langle \widehat{\bm \delta}_2,\bm \mu_2\rangle_{\partial \mathcal{T}_h\backslash \varepsilon_h^\partial}+(\bm y-\bm y_h( u),\bm v_2)_{\mathcal T_h}.
	\end{align}
\end{lemma}

\subsubsection{Step 6: Estimate for $\varepsilon_h^{\mathbb{G}}$ and  $\varepsilon_h^z$.}
To estimate $\varepsilon_h^{\mathbb{G}}$, we use the following discrete Poincar\'{e} inequality from \cite[Proposition A.2]{MR3626531}.
\begin{lemma}\label{lemma:discr_Poincare_ineq}
	We have
	\begin{align}\label{poin_in}
	\|\varepsilon_h^{\bm z}\|_{\mathcal T_h} \le C(\|\nabla \varepsilon_h^{\bm z}\|_{\mathcal T_h} + h^{-\frac 1 2} \|\varepsilon_h^{\bm z} - \varepsilon_h^{\widehat {\bm z}}\|_{\partial\mathcal T_h}).
	\end{align}
\end{lemma}

	\begin{lemma}\label{lemma:step6_main_lemma}
		Let $\mathcal M$ be defined as in \Cref{error_energy_norm}. Then we have
		\begin{align*}
		\|\varepsilon_h^{\mathbb G}\|_{\mathcal{T}_h}+h^{-\frac 1 2}\|{\bm P_M\varepsilon_h^{\bm z}-\varepsilon_h^{\widehat{\bm z}}}\|_{\partial \mathcal T_h}+\|\varepsilon_h^{\bm z}\|_{\mathcal{T}_h} \lesssim h\mathcal M + \mathcal N,
		\end{align*}
		where  $\mathcal N =  h^{s_{\mathbb G}}\norm{\mathbb G}_{s^{\mathbb G},\Omega}+  h^{s_{q}}\norm{q}_{s^{q},\Omega}+ h^{s_{\bm z}-1}\norm{\bm z}_{s^{\bm z},\Omega}$.

	\end{lemma}

\begin{proof}
	The inequality in  \Cref{grad_y_h} holds with $ (\mathbb{L},\bm y, \widehat {\bm y}) $ replaced by $ (\mathbb{G},\bm z,\widehat{\bm z} ) $, which gives
	\begin{align}\label{grad_z_h}
	\|\nabla \varepsilon_h^{\bm z}\|_{\mathcal{T}_h} +h^{-\frac{1}{2}}\|\varepsilon_h^{\bm z} -\varepsilon_h^{\widehat{\bm z}}\|_{\partial \mathcal{T}_h} \lesssim \|\varepsilon_h^{\mathbb{G}}\|_{\mathcal{T}_h} +h^{-\frac{1}{2}} \|\bm P_{ M} \varepsilon_h^{\bm z} -\varepsilon_h^{\widehat{\bm z}}\|_{\partial \mathcal{T}_h}.
	\end{align}
	Next, since $\varepsilon_h^{\widehat{\bm z}}=0$ on $\varepsilon_h^\partial$, the energy identity for $\mathscr B$ in  \Cref{energy_norm} gives
	\begin{align}
	\mathscr B(\varepsilon_h^{\mathbb{G}},\varepsilon_h^{\bm z},-\varepsilon_h^q,\varepsilon_h^{\widehat{\bm z}};\varepsilon_h^{\mathbb{G}},\varepsilon_h^{\bm z},-\varepsilon_h^q,\varepsilon_h^{\widehat{\bm z}})=\|\varepsilon_h^{\mathbb{G}}\|^2_{\mathcal T_h}+h^{-1}\|\bm P_{ M} \varepsilon_h^{\bm z}-\varepsilon_h^{\widehat{\bm z}}\|^2_{\partial\mathcal T_h}.
	\end{align}
	Using $(\mathbb T_2,\bm v_2,w_2,\bm \mu_2)=(\varepsilon_h^{\mathbb{G}},\varepsilon_h^{\bm z},\varepsilon_h^q,\varepsilon_h^{\widehat{\bm z}})$ in the error equation \eqref{error_B_2} gives
	\begin{align*}
	\|\varepsilon_h^{\mathbb{G}}\|^2_{\mathcal T_h}+h^{-1}\|\bm P_{ M} \varepsilon_h^{\bm z}-\varepsilon_h^{\widehat{\bm z}}\|^2_{\partial\mathcal T_h}&=\langle\widehat{\bm\delta}_2,\varepsilon_h^{\bm z} -\varepsilon_h^{\widehat{\bm z}}\rangle_{\partial \mathcal{T}_h}+(\bm y-\bm y_h( u), \varepsilon_h^{\bm z})_{\mathcal T_h}\\
	&=:T_1+T_2.
	\end{align*}
	For $ T_1 $, use \eqref{grad_z_h} and Young's inequality:
	\begin{align*}
	T_1&=\langle\widehat{\bm \delta}_2,\varepsilon_h^{\bm z} -\varepsilon_h^{\widehat{\bm z}}\rangle_{\partial \mathcal{T}_h}\le   Ch^{\frac{1}{2}} (\|\delta^{\mathbb G}\|_{\partial \mathcal{T}_h} + \|\delta^{ q}\|_{\partial \mathcal{T}_h}) h^{-\frac{1}{2}}\|\varepsilon_h^{\bm z}-\varepsilon_h^{\widehat{\bm z}}\|_{\partial \mathcal{T}_h}  +C h^{-\frac 1 2}\|\delta^{\bm z}\|_{\partial \mathcal T_h} h^{-\frac 1 2}\|\bm P_{M}  \varepsilon_h^{\bm z}-\varepsilon_h^{\widehat{\bm z}} \|_{\partial \mathcal T_h}\\
	&\le Ch(\|\delta^{\mathbb G}\|_{\partial \mathcal{T}_h}^2 + \|\delta^{q}\|_{\partial \mathcal{T}_h}^2 ) +Ch^{-1}\|\delta^{\bm z}\|_{\partial \mathcal{T}_h}^2 + \frac 1 4 \|\varepsilon_h^{\mathbb G}\|_{\mathcal T_h}^2 + \frac{1}{4h}\|\bm P_{M}  \varepsilon_h^{\bm z}-\varepsilon_h^{\widehat{\bm z}} \|_{\partial \mathcal T_h}^2.
	\end{align*}
	For the term $T_2$, apply \Cref{lemma:discr_Poincare_ineq} and \eqref{grad_z_h} to obtain
	\begin{align*}
	T_2&= (\bm y-\bm y_h( u),\varepsilon_h^{\bm z})_{\mathcal T_h} \le  \|\bm y-\bm y_h( u)\|_{\mathcal T_h} \|\varepsilon_h^{\bm z}\|_{\mathcal T_h}\\
	&\le C\|\bm y-\bm y_h( u)\|_{\mathcal T_h} (\|\nabla \varepsilon_h^{\bm z}\|_{\mathcal T_h} + h^{-\frac 1 2} \|\varepsilon_h^{\bm z} - \varepsilon_h^{\widehat {\bm z}}\|_{\partial\mathcal T_h})\\
	&\le  C\|\bm y-\bm y_h( u)\|_{\mathcal T_h} (\|\varepsilon^{\mathbb G}_h\|_{\mathcal T_h}+h^{-\frac1 2}\|\bm P_{ M}\varepsilon_h^{\bm z}-\varepsilon^{\widehat {\bm z}}_h\|_{\partial\mathcal T_h})\\
	&\le  C\|\bm y-\bm y_h( u)\|_{\mathcal T_h}^2 + \frac 1 4
	\|\varepsilon_h^{\mathbb G}\|_{\mathcal T_h}^2 + \frac 1 {4h} \|{ \bm P_{ M}\varepsilon_h^{\bm z}-\varepsilon_h^{\widehat{\bm z}}}\|_{\partial \mathcal T_h}^2.
	\end{align*}
	These estimates and \Cref{lemma:discr_Poincare_ineq} give our desired result.
\end{proof}

\subsubsection{Step 7: Estimate for $\varepsilon_h^q$. }

	\begin{lemma}\label{pressure_estimate_q_1}
		Let  $\mathcal M$ and $\mathcal N$  be defined as in \Cref{error_y_lyp,lemma:step6_main_lemma}, respectively. Then we have
		\begin{align}\label{error_qo}
		\|{\varepsilon_h^q}\|_\Omega\lesssim  h\mathcal M + \mathcal N.
		\end{align}
	\end{lemma}

\begin{proof}
	By the same argument as in \Cref{pressure_estimate_p_1}, we have
	\begin{align*}
	\|\varepsilon_h^q \|_\Omega \lesssim \sup_{\bm v\in {\bm H}_0^1 (\Omega)\backslash\{0\}} \frac{(\varepsilon_h^q,\nabla\cdot \bm v)_{\Omega}}{\norm{\bm v} _{\bm H ^1 (\Omega)}},
	\end{align*}
	and
	\begin{align*}
	(\varepsilon_h^q ,\nabla\cdot \bm v)_\Omega =-(\nabla \varepsilon_h^q ,\bm \Pi_{ V} \bm v)_{\mathcal{T}_h}+\langle\varepsilon_h^q, \bm v\cdot \bm n\rangle_{\partial \mathcal{T}_h}.
	\end{align*}
	Next, taking  $(\mathbb T_2,\bm v_2, w_2,\bm \mu_2) = (0,\bm \Pi_{V} \bm v,0,0)$ in \Cref{error_eq_Gqz} and using \eqref{def_B2} give
	\begin{align*}
	(\nabla \varepsilon_h^q ,\bm \Pi_{ V} \bm v)_{\mathcal{T}_h}&= - (\nabla\cdot\varepsilon_h^{\mathbb G}, \bm \Pi_{ V} \bm v)_{\mathcal{T}_h}+\langle h^{-1}(\bm P_M\varepsilon_h^{\bm z}-\varepsilon_h^{\widehat{\bm z}}),\bm\Pi_{V}\bm v\rangle_{\partial \mathcal{T}_h} \\
	&\quad- \langle \widehat {\bm \delta}_2, \bm \Pi_{ V} \bm v \rangle_{\partial \mathcal T_h}-(\bm y-\bm y_h( u), \bm \Pi_{ V}  \bm v)_{\mathcal T_h}.
	\end{align*}
	These equalities give
	\begin{align*}
(\varepsilon_h^q ,\nabla\cdot \bm v)_\Omega&=  (\nabla\cdot\varepsilon_h^{\mathbb G}, \bm \Pi_{ V} \bm v)_{\mathcal{T}_h}-\langle h^{-1}(\bm P_M\varepsilon_h^{\bm z}-\varepsilon_h^{\widehat{\bm z}}),\bm\Pi_{V}\bm v\rangle_{\partial \mathcal{T}_h}\\
	&\quad+\langle\varepsilon_h^q, \bm v\cdot \bm n\rangle_{\partial \mathcal{T}_h} + \langle \widehat {\bm \delta}_2, \bm \Pi_{ V} \bm v \rangle_{\partial \mathcal T_h}+(\bm y-\bm y_h( u), \bm \Pi_{ V}  \bm v)_{\mathcal T_h}\\
	&=  (\nabla\cdot\varepsilon_h^{\mathbb G}, \bm v)_{\mathcal{T}_h}- \langle h^{-1}(\bm P_M\varepsilon_h^{\bm z}-\varepsilon_h^{\widehat{\bm z}}),\bm\Pi_{V}\bm v\rangle_{\partial \mathcal{T}_h}\\
	&\quad+\langle\varepsilon_h^q, \bm v\cdot \bm n\rangle_{\partial \mathcal{T}_h} + \langle \widehat {\bm \delta}_2, \bm \Pi_{ V} \bm v \rangle_{\partial \mathcal T_h}+(\bm y-\bm y_h( u), \bm \Pi_{ V}  \bm v)_{\mathcal T_h}\\
	&= - (\varepsilon_h^{\mathbb G}, \nabla \bm v)_{\mathcal{T}_h}-\langle h^{-1}(\bm P_M\varepsilon_h^{\bm z}-\varepsilon_h^{\widehat{\bm z}}),\bm\Pi_{V}\bm v\rangle_{\partial \mathcal{T}_h}\\
	&\quad +\langle  \varepsilon_h^{\mathbb G} \bm n + \varepsilon_h^q\bm n, \bm P_M \bm v\rangle_{\partial \mathcal{T}_h} + \langle \widehat {\bm \delta}_2, \bm \Pi_{ V} \bm v \rangle_{\partial \mathcal T_h}+(\bm y-\bm y_h( u), \bm \Pi_{ V}  \bm v)_{\mathcal T_h}.
	\end{align*}
	Next, using $(\mathbb T_2,\bm v_2, w_2,\bm \mu_2) = (0, 0,0,\bm P_M\bm v )$ in \Cref{error_eq_Gqz} and $\bm v\in \bm H_0^1(\Omega)$ gives
	\begin{align*}
	\langle  \varepsilon_h^{\mathbb G} \bm n + \varepsilon_h^q\bm n, \bm P_M \bm v\rangle_{\partial \mathcal{T}_h} = \langle h^{-1}(\bm P_M\varepsilon_h^{\bm z}-\varepsilon_h^{\widehat{\bm z}}),\bm P_M\bm v\rangle_{\partial \mathcal{T}_h}-\langle \widehat{\bm \delta}_2, \bm P_M\bm v\rangle_{\partial \mathcal{T}_h}.
	\end{align*}
	This implies
	\begin{align*}
	(\varepsilon_h^q ,\nabla\cdot \bm v)_\Omega  & =  -(\varepsilon_h^{\mathbb G}, \nabla \bm v)_{\mathcal{T}_h}-\langle h^{-1}(\bm P_M\varepsilon_h^{\bm z}-\varepsilon_h^{\widehat{\bm z}}),\bm\Pi_{V}\bm v - \bm P_M \bm v \rangle_{\partial \mathcal{T}_h} \\
	&\quad+ \langle \widehat{\bm \delta}_2,\bm \Pi_V \bm v-  \bm P_M\bm v\rangle_{\partial \mathcal{T}_h}+(\bm y-\bm y_h( u), \bm \Pi_{ V}  \bm v)_{\mathcal T_h}.
	\end{align*}
	Applying the Cauchy-Schwarz inequality, we obtain
	\begin{align*}
	|	(\varepsilon_h^q ,\nabla\cdot \bm v)_\Omega| &\lesssim \|\varepsilon_h^{\mathbb G}\|_{\mathcal T_h}\|\nabla\bm v\|_{\mathcal T_h} + h^{-\frac 1 2} \|\bm P_M \varepsilon_h^{\bm z} - \varepsilon_h^{\widehat{\bm z}}\|_{\partial\mathcal T_h} \|\nabla \bm v\|_{\mathcal T_h} \\
	&\quad+ h^{\frac 1 2} \|\widehat{\bm \delta}_2\|_{\partial \mathcal T_h}\|\nabla \bm v\|_{\mathcal T_h} + \|\bm y-\bm y_h( u)\|_{\mathcal T_h} \|\bm v\|_{\mathcal T_h}.
	\end{align*}
	Since $\bm v\in \bm H_0^1(\Omega)$, the Poincar\'{e} inequality yields our final result.
\end{proof}

\Cref{lemma:step6_main_lemma}, \Cref{pressure_estimate_q_1}, and the triangle inequality give optimal convergence rates for $\|\mathbb G- \mathbb G_h( u)\|_{\mathcal T_h}$, $\| q - q_h( u)\|_{\mathcal T_h}$, and $\| \bm z - \bm z_h( u)\|_{\mathcal T_h}$:

	\begin{lemma}\label{lemma:step7_conv_rates}
		Let  $\mathcal M$ and $\mathcal N$  be defined as in \Cref{error_y_lyp} and \Cref{lemma:step6_main_lemma}, respectively. Then we have
		\begin{align*}
		\|\mathbb{G} -\mathbb{G}_h( u)\|_{\mathcal{T}_h}+\|q -q_h( u)\|_{\mathcal{T}_h}+	\|\bm z-\bm z_h( u)\|_{\mathcal{T}_h} \lesssim 	h\mathcal M + \mathcal N.
		\end{align*}	
	\end{lemma}

\subsubsection{Step 8: Estimate for $\|u-u_h\|_{\varepsilon_h^\partial}$ and $\norm {y-y_h}_{\mathcal T_h}$.}

Next, we consider the solution of the auxiliary problem and the solution of the HDG discretization of the optimality system \eqref{HDG_full_discrete}.  Our main result follows from bounding the errors between these solutions as well as \Cref{lemma:step4_conv_rates} and \Cref{lemma:step7_conv_rates}.

Define
\begin{align*}
\zeta_{\mathbb L}&=\mathbb L_h( u)-\mathbb L_h, & \zeta_{\bm y}&=\bm y_h( u)-\bm y_h, & \zeta_p&=p_h( u)- p_h,\\
\zeta_{\mathbb G}&=\mathbb G_h( u)-\mathbb G_h, & \zeta_{\bm z}&=\bm z_h( u)-\bm z_h, & \zeta_q&=q_h( u)- q_h,
\end{align*}
and
\begin{align*}
\zeta_{\widehat {\bm y}}&=\widehat{\bm y}^o_h( u)-\widehat{\bm y}_h^o \;\; \textup{on} \;\; \varepsilon_h^{o}, & \zeta_{\widehat {\bm y}} &=  P_M u \bm\tau -  u_h\bm \tau \;\; \textup{on} \;\; \varepsilon_h^{\partial} ,\\
\zeta_{\widehat {\bm z}}&=\widehat{\bm z}^o_h( u)-\widehat{\bm z}_h^o \;\; \textup{on} \;\; \varepsilon_h^{o}, & \zeta_{\widehat {\bm z}} &= 0\;\; \textup{on} \;\; \varepsilon_h^{\partial}.
\end{align*}
Subtracting the auxiliary problem \eqref{HDG_inter_u} and the HDG problem \eqref{HDG_full_discrete} yields the error equations
\begin{subequations}\label{eq_yh}
	\begin{align}
	\mathscr B(\zeta_{\mathbb L},\zeta_{\bm y},\zeta_p,\zeta_{\widehat{\bm y}};\mathbb T_1,\bm v_1, w_1,\bm \mu_1)&=\langle  (P_M  u - u_h)\bm\tau, h^{-1} \bm v_1+\mathbb T_1\bm n \rangle_{\varepsilon_h^\partial},\label{eq_yh1}\\
	\mathscr B(\zeta_{\mathbb G},\zeta_{\bm z},-\zeta_q,\zeta_{\widehat{\bm z}};\mathbb T_2,\bm v_2, w_2,\bm \mu_2)&=(\zeta_{\bm y},\bm v_2)_{\mathcal{T}_h},\label{eq_yh2}
	\end{align}
\end{subequations}
for all $\left(\mathbb T_1,\mathbb T_2, \bm v_1,\bm v_2, w_1,w_2,\bm\mu_1,\bm{\mu}_2\right)\in \mathbb K_h\times\mathbb K_h\times\bm{V}_h\times \bm{V}_h \times W_h^0 \times W_h^0\times \bm M_h(o)\times \bm M_h(o)$.

\begin{lemma}\label{lemma:estimate_for_u_uh}
	We have
	\begin{align}\label{eq_uuh_yhuyh}
	\begin{split}
	\gamma\| u-  u_h\|^2_{\varepsilon_h^\partial}+\|\zeta_{\bm y}\|^2_{\mathcal T_h} &= \langle \gamma  u\bm{\tau}-\mathbb G_h( u) \bm n+h^{-1}  \bm P_M\bm z_h( u), (u-  u_h)\bm{\tau}\rangle_{\varepsilon_h^\partial} \\
	& \quad- \langle \gamma  u_h \bm \tau- \mathbb G_h\bm n +h^{-1} \bm P_M\bm z_h,  (u-  u_h)\bm{\tau} \rangle_{\varepsilon_h^\partial}.
	\end{split}
	\end{align}
\end{lemma}
\begin{proof}
	First,
	\begin{align*}
	\langle &\gamma  u\bm \tau-\mathbb G_h( u) \bm n+h^{-1}  \bm P_M\bm z_h( u), (u- u_h)\bm{\tau}\rangle_{\varepsilon_h^\partial}- \langle \gamma  u_h \bm \tau- \mathbb G_h\bm n +h^{-1} \bm P_M\bm z_h, ( u- u_h)\bm{\tau}\rangle_{\varepsilon_h^\partial}\\
	&= \gamma\norm{ u- u_h}_{\varepsilon_h^\partial}^2 +\langle -\zeta_{\mathbb G} \bm n+ h^{-1}  \bm P_M \zeta_{\bm z}, (u-  u_h)\bm{\tau}\rangle_{\varepsilon_h^\partial}.
	\end{align*}
	Next, \Cref{identical_equa} gives
	\begin{align*}
	\mathscr B(\zeta_{\mathbb L},\zeta_{\bm y},\zeta_{p}, \zeta_{\widehat{\bm y}}; - \zeta_{\mathbb G},\zeta_{\bm z},\zeta_{q}, \zeta_{\widehat{\bm z}}) +\mathscr B ( \zeta_{\mathbb G},\zeta_{\bm z},-\zeta_{q}, \zeta_{\widehat{\bm z}}; \zeta_{\mathbb L},-\zeta_{\bm y},\zeta_{p},-\zeta_{\widehat{\bm y}})  = 0.
	\end{align*}
	However, taking $(\mathbb T_1,\bm v_1, w_1,\bm \mu_1) = (- \zeta_{\mathbb G},\zeta_{\bm z},\zeta_{q}, \zeta_{\widehat{\bm z}})$ and $(\mathbb T_2,\bm v_2, w_2,\bm \mu_2)=(\zeta_{\mathbb L},-\zeta_{\bm y},\zeta_{p},-\zeta_{\widehat{\bm y}})$ in the error equations \eqref{eq_yh}  yield
	\begin{align*}
	\hspace{3em}&\hspace{-3em} \mathscr B(\zeta_{\mathbb L},\zeta_{\bm y},\zeta_{p}, \zeta_{\widehat{\bm y}}; - \zeta_{\mathbb G},\zeta_{\bm z},\zeta_{q}, \zeta_{\widehat{\bm z}}) +\mathscr B ( \zeta_{\mathbb G},\zeta_{\bm z},-\zeta_{q}, \zeta_{\widehat{\bm z}}; \zeta_{\mathbb L},-\zeta_{\bm y},\zeta_{p},-\zeta_{\widehat{\bm y}})\\
	&=  -(\zeta_{ \bm y},\zeta_{ \bm y})_{\mathcal{T}_h} + \langle  P_M ( u- u_h) \bm \tau, -\zeta_{\mathbb G} \bm{n} + h^{-1} \zeta_{ \bm z}  \rangle_{{\varepsilon_h^{\partial}}}\\
	&= -(\zeta_{\bm y},\zeta_{ \bm y})_{\mathcal{T}_h} + \langle   (u- u_h) \bm \tau, -\zeta_{\mathbb G} \bm{n} + h^{-1} \bm P_M \zeta_{ \bm z}  \rangle_{{\varepsilon_h^{\partial}}}.
	\end{align*}
	
	Comparing these equalities gives
	\begin{align*}
	(\zeta_{\bm y},\zeta_{ \bm y})_{\mathcal{T}_h} =  \langle   (u- u_h) \bm \tau, -\zeta_{\mathbb G} \bm{n} + h^{-1} \bm P_M \zeta_{ \bm z}  \rangle_{{\varepsilon_h^{\partial}}}.
	\end{align*}
\end{proof}

	\begin{theorem}\label{error_u}
		Let  $\mathcal M$ and $\mathcal N$  be defined as in \Cref{error_y_lyp} and \Cref{lemma:step6_main_lemma}, respectively. Then we have
		\begin{align*}
		\norm{ u-  u_h}_{\varepsilon_h^\partial}+\norm {\bm y-\bm y_h}_{\mathcal T_h}	\lesssim 		h^{-\frac 1 2 }(h\mathcal M + \mathcal N).
		\end{align*}
	\end{theorem}

\begin{proof}
	The optimality condition \eqref{TCOOC} gives $\langle \gamma u\bm \tau-\mathbb G\bm n, (u-u_h)\bm \tau\rangle_{\varepsilon_h^\partial} = 0$.  Also,
	\begin{align*}
	\langle	\gamma  u_h\bm \tau - \mathbb G_h \bm n  + h^{-1} \bm P_M \bm z_h, (u-u_h)\bm \tau \rangle_{\varepsilon_h^\partial} = 	\langle	\gamma  u_h\bm \tau - \mathbb G_h \bm n + h^{-1} \bm P_M \bm z_h,  (P_M u- u_h) \bm \tau\rangle_{\varepsilon_h^\partial}=0,
	\end{align*}
	where we used the HDG optimality condition \eqref{HDG_discrete2_m} and \eqref{HDG_discrete2_q}.  %
	Using these equalities in \eqref{eq_uuh_yhuyh} from \Cref{lemma:estimate_for_u_uh} gives
	\begin{align*}
	\gamma\norm{ u- u_h}_{\varepsilon_h^\partial}^2  + \norm {\zeta_{\bm y}}_{\mathcal T_h}^2 &=  \langle \gamma  u \bm \tau-\mathbb G_h( u) \bm n+h^{-1}  \bm P_M\bm z_h( u), (u-  u_h)\bm \tau\rangle_{\varepsilon_h^\partial}\\
	&=\langle (\mathbb G -\mathbb G_h( u))\bm n + h^{-1} \bm P_M \bm z_h( u),  (u-  u_h)\bm \tau \rangle_{\varepsilon_h^\partial}.
	\end{align*}
	First, by the estimation of the standard $L^2$ projection in \eqref{standard_L2_projection} and the trace inequality  we have
	\begin{align}\label{reduce_convergence_rate}
		\begin{split}
			\|(\mathbb G -\mathbb G_h( u))\bm n\|_{\varepsilon_h^\partial}&\le \|\mathbb G -\mathbb G_h(u)\|_{\partial \mathcal T_h} \le  \|\mathbb G - \bm \Pi_{\mathbb K} \mathbb G\|_{\partial \mathcal T_h}  +  \|\bm \Pi_{\mathbb K} \mathbb G -\mathbb G_h( u) \|_{\partial \mathcal T_h} \\
			&\lesssim  h^{s_{\mathbb G}-\frac 1 2}\|\mathbb G\|_{s^{\mathbb G},\Omega} +  h^{-\frac 1 2}\|\bm \Pi_{\mathbb K} \mathbb G -\mathbb G_h( u) \|_{ \mathcal T_h}\\
			&\lesssim h^{s_{\mathbb G}-\frac 12 } \norm{\mathbb G}_{s^{\mathbb G},\Omega} + h^{-\frac 1 2}\norm {\varepsilon_h^{\mathbb G}}_{\mathcal T_h}.
		\end{split}
	\end{align}
	Next, since $\widehat {\bm z}_h( u) = \bm z = \bm 0$ on $\varepsilon_h^{\partial}$ we have
	\begin{align*}
 \|\bm P_M \bm z_h(u)\|_{\varepsilon_h^\partial}&=\|\bm P_M \bm z_h(u)  - \bm P_M \bm \Pi_V \bm z + \bm P_M \bm \Pi_V \bm z - \bm P_M  \bm z  + \bm P_M \bm z- \widehat {\bm z}_h( u) \|_{\varepsilon_h^\partial}  \\
	&\leq  \|\bm P_M \varepsilon_h^{\bm z} -\varepsilon_h^{\widehat {\bm z}}\|_{\partial\mathcal T_h} + \|\bm{\Pi}_V \bm z - \bm z\|_{\varepsilon_h^\partial}.
	\end{align*}
	This yields
	\begin{align*}
	\norm{ u- u_h}_{\varepsilon_h^\partial}  + \|\zeta_{\bm y}\|_{\mathcal T_h}\lesssim h^{-\frac 1 2}\norm {\varepsilon_h^{\mathbb G}}_{\mathcal T_h} +h^{s_{\mathbb G}-\frac 12 } \norm{\mathbb G}_{s^{\mathbb G},\Omega}  +h^{-1}\|\bm P_M \varepsilon_h^{\bm z} -\varepsilon_h^{\widehat {\bm z}}\|_{\partial\mathcal T_h} + h^{-1}\|\bm{\Pi}_V \bm z - \bm z\|_{\varepsilon_h^\partial}.
	\end{align*}
	\Cref{lemma:step6_main_lemma} and properties of the $ L^2 $ projection give the desired result.
\end{proof}

\begin{remark}\label{Remark4.1}The application of a trace theorem to estimate the normal derivatives of the adjoint state in \eqref{reduce_convergence_rate} yields suboptimal results only. In the case of the Poisson equation, and using standard finite element methods, optimal error estimates for the so called {\em variational} normal derivative can be found in \cite{MR4000217,MR3319574,Winkler_29}. However, proving sharp convergence rates for normal derivatives requires regularity results in weighted $W^{k,\infty}(\Omega)$ spaces and non-standard duality arguments. Unfortunately, establishing sharp convergence rates for normal derivatives using the HDG methods is unclear to us; we leave this interesting project to be explored in future work.	
\end{remark}

\subsubsection{Step 9: Estimates for $\|\mathbb G-\mathbb G_h\|_{\mathcal T_h}$ and $\|z-z_h\|_{\mathcal T_h}$.}

	\begin{lemma}
		Let  $\mathcal M$ and $\mathcal N$ be defined as in \Cref{error_y_lyp} and \Cref{lemma:step6_main_lemma}, respectively. Then we have
		\begin{align*}
		\norm {\zeta_{\mathbb G}}_{\mathcal T_h} +
		\|\zeta_{\bm z}\|_{\mathcal T_h}\lesssim 	h^{-\frac 1 2 }(h\mathcal M + \mathcal N).
		\end{align*}
	\end{lemma}
	\begin{proof}
		Using the energy identity for $ \mathscr B $ as in \Cref{energy_norm}, the error equation  \eqref{eq_yh2}, $\zeta_{ \widehat {\bm z}} = 0$ on $\varepsilon_h^{\partial}$, the discrete Poincar\'{e} inequality in \Cref{lemma:discr_Poincare_ineq}, and \Cref{grad_y_h} gives
		\begin{align*}
		\hspace{2em}&\hspace{-2em} \mathscr B ( \zeta_{\mathbb G},\zeta_{\bm z},-\zeta_{q}, \zeta_{\widehat{\bm z}}; \zeta_{\mathbb G},\zeta_{\bm z},-\zeta_{q}, \zeta_{\widehat{\bm z}}) \\
		&=(\zeta_{\mathbb G},\zeta_{\mathbb G})_{{\mathcal{T}_h}}+h^{-1}  \|\bm P_M \zeta_{\bm z} -\zeta_{\widehat {\bm z}}\|_{\partial\mathcal T_h}^2 \\
		&=(\zeta_{\bm y},\zeta_{\bm z})_{\mathcal T_h}\\
		&\le \norm{\zeta_{\bm y}}_{\mathcal T_h} \norm{\zeta_{\bm z}}_{\mathcal T_h}\\
		&\lesssim \norm{\zeta_{\bm y}}_{\mathcal T_h} (\|\nabla \zeta_{\bm z}\|_{\mathcal T_h} + h^{-\frac 1 2} \| \zeta_{\bm z} - \zeta_{\widehat {\bm z}}\|_{\partial\mathcal T_h}) \\
		&\lesssim \norm{\zeta_{\bm y}}_{\mathcal T_h} (\|\zeta_{\mathbb G}\|_{\mathcal T_h} + h^{-\frac 1 2} \|\bm P_M \zeta_{\bm z} - \zeta_{\widehat {\bm z}}\|_{\partial\mathcal T_h}).
		\end{align*}
		This implies
		\begin{align*}
		\norm {\zeta_{\mathbb G}}_{\mathcal T_h} +h^{-\frac1 2}\|\bm P_M\zeta_{\bm z}-\zeta_{\widehat {\bm z}}\|_{\partial\mathcal T_h}\lesssim  		h^{-\frac 1 2 }(h\mathcal M + \mathcal N).
		\end{align*}
		Using the discrete Poincar\'{e} inequality again yields
		\begin{align*}
		\|\zeta_{\bm z}\|_{\mathcal T_h} & \lesssim \|\nabla \zeta_{\bm z}\|_{\mathcal T_h} + h^{-\frac 1 2} \|\zeta_{\bm z} - \zeta_{\widehat {\bm{z}}}\|_{\partial\mathcal T_h}\\
		&\lesssim \|\zeta_{\mathbb G}\|_{\mathcal T_h} + h^{-\frac 1 2} \|\bm P_M \zeta_{\bm z} - \zeta_{\widehat {\bm z}}\|_{\partial\mathcal T_h}\\
		&\lesssim  h^{-\frac 1 2 }(h\mathcal M + \mathcal N).
		\end{align*}
		This finishes the proof.
	\end{proof}
	
	The above lemma, the triangle inequality, \Cref{lemma:step4_conv_rates,lemma:step7_conv_rates} give the following result:
	\begin{theorem}
		Let  $\mathcal M$ and $\mathcal N$ be defined in \Cref{error_y_lyp,lemma:step6_main_lemma}, respectively. Then we have
		\begin{align*}
		\norm {\mathbb G - \mathbb G_h}_{\mathcal T_h}  +
		\norm {\bm z - \bm z_h}_{\mathcal T_h}  \lesssim 	h^{-\frac 1 2 }(h\mathcal M + \mathcal N).
		\end{align*}
	\end{theorem}
	
	\subsubsection{Step 10: Estimate for $\|q-q_h\|_{\mathcal T_h}$.}
	
	\begin{lemma}\label{pressure_q1}
		Let  $\mathcal M$ and $\mathcal N$ be defined in \Cref{error_y_lyp,lemma:step6_main_lemma}, respectively. Then we have
		\begin{align*}
		\|{\zeta_q}\|_\Omega  \lesssim 	h^{-\frac 1 2 }(h\mathcal M + \mathcal N).
		\end{align*}
	\end{lemma}

\begin{proof}
	By the same argument as in \Cref{pressure_estimate_p_1}, we have
	\begin{align*}
	\| \zeta_q\|_\Omega \lesssim \sup_{\bm v\in {\bm H}_0^1 (\Omega)\backslash\{0\}} \frac{(\zeta_q,\nabla\cdot \bm v)_{\Omega}}{\norm{\bm v} _{\bm H ^1 (\Omega)}},
	\end{align*}
	and
	\begin{align*}
	(\zeta_q,\nabla\cdot \bm v)_\Omega =-(\nabla \zeta_q ,\bm \Pi_{ V} \bm v)_{\mathcal{T}_h}+\langle \zeta_q, \bm v\cdot \bm n\rangle_{\partial \mathcal{T}_h}.
	\end{align*}
	Next, taking  $(\mathbb T_2,\bm v_2, w_2,\bm \mu_2) = (0,\bm \Pi_{V} \bm v,0,0)$ in the error equation \eqref{eq_yh2}, using \eqref{def_B2} and $\zeta_{\widehat {\bm z}} = 0$ on $\varepsilon_h^\partial$ give
	\begin{align*}
	(\nabla \zeta_q,\bm \Pi_{ V} \bm v)_{\mathcal{T}_h} = - (\nabla\cdot\zeta_{\mathbb G}, \bm \Pi_{ V} \bm v)_{\mathcal{T}_h} +\langle h^{-1}(\bm P_M\zeta_{\bm z}-\zeta_{\widehat{\bm z}}),\bm\Pi_{V}\bm v\rangle_{\partial \mathcal{T}_h} - (\zeta_{\bm y},\bm\Pi_V \bm v)_{\mathcal{T}_h}.
	\end{align*}
	The above two equalities give
	\begin{align*}
	(\zeta_q ,\nabla\cdot \bm v)_\Omega  &=  (\nabla\cdot\zeta_{\mathbb G}, \bm \Pi_{ V} \bm v)_{\mathcal{T}_h}-\langle h^{-1}(\bm P_M\zeta_{\bm z}-\zeta_{\widehat{\bm z}}),\bm\Pi_{V}\bm v\rangle_{\partial \mathcal{T}_h} +\langle\zeta_q, \bm v\cdot \bm n\rangle_{\partial \mathcal{T}_h} + (\zeta_{\bm y},\bm\Pi_V \bm v)_{\mathcal{T}_h}\\
	&=  (\nabla\cdot\zeta_{\mathbb G}, \bm v)_{\mathcal{T}_h}-\langle h^{-1}(\bm P_M\zeta_{\bm z}-\zeta_{\widehat{\bm z}}),\bm\Pi_{V}\bm v\rangle_{\partial \mathcal{T}_h} +\langle\zeta_q, \bm v\cdot \bm n\rangle_{\partial \mathcal{T}_h}+ (\zeta_{\bm y},\bm\Pi_V \bm v)_{\mathcal{T}_h}\\
	&= - (\zeta_{\mathbb G}, \nabla \bm v)_{\mathcal{T}_h}-\langle h^{-1}(\bm P_M\zeta_{\bm z}-\zeta_{\widehat{\bm z}}),\bm\Pi_{V}\bm v\rangle_{\partial \mathcal{T}_h}+\langle  \zeta_{\mathbb G} \bm n + \zeta_q\bm n, \bm P_M \bm v\rangle_{\partial \mathcal{T}_h\backslash\varepsilon_h^\partial} + (\zeta_{\bm y}, \bm v)_{\mathcal{T}_h}.
	\end{align*}
	Next, take $(\mathbb T_2,\bm v_2, w_2,\bm \mu_2) = (0, 0,0,\bm P_M\bm v )$ in \eqref{eq_yh2} and use $\bm v\in \bm H_0^1(\Omega)$ to obtain
	\begin{align*}
	\langle  \zeta_{\mathbb G} \bm n +\zeta_q\bm n, \bm P_M \bm v\rangle_{\partial \mathcal{T}_h\backslash\varepsilon_h^\partial} = \langle h^{-1}(\bm P_M\zeta_{\bm z}-\zeta_{\widehat{\bm z}}),\bm P_M\bm v\rangle_{\partial \mathcal{T}_h}.
	\end{align*}
	This implies
	\begin{align*}
	(\zeta_q ,\nabla\cdot \bm v)_\Omega  &= - (\zeta_{\mathbb G}, \nabla \bm v)_{\mathcal{T}_h}-\langle h^{-1}(\bm P_M\zeta_{\bm z}-\zeta_{\widehat{\bm z}}),\bm\Pi_{V}\bm v - \bm P_M \bm v \rangle_{\partial \mathcal{T}_h}+ (\zeta_{\bm y}, \bm v)_{\mathcal{T}_h},
	\end{align*}
	and therefore
	\begin{align*}
 |	(\zeta_q ,\nabla\cdot \bm v)_\Omega|\lesssim \|\zeta_{\mathbb G}\|_{\mathcal T_h}\|\nabla\bm v\|_{\mathcal T_h} + h^{-\frac 1 2} \|\bm P_M \zeta_{\bm z} - \zeta_{\widehat{\bm z}}\|_{\partial\mathcal T_h} \|\nabla \bm v\|_{\mathcal T_h} + \|\zeta_{\bm y}\|_{\mathcal T_h} \| \bm v\|_{\mathcal T_h}.
	\end{align*}
	Since  $\bm v\in \bm H_0^1(\Omega)$, the Poincar\'{e} inequality gives the desired result.
\end{proof}

The above lemma, the triangle inequality, and \Cref{lemma:step7_conv_rates} give the following error bound:

	\begin{theorem}
		Let  $\mathcal M$ and $\mathcal N$ be defined as in \Cref{error_y_lyp,lemma:step6_main_lemma}, respectively. Then we have
		\begin{align*}
		\|q - q_h\|_{\mathcal T_h}  \lesssim 	h^{-\frac 1 2 }(h\mathcal M + \mathcal N).
		\end{align*}
	\end{theorem}

\subsubsection{Step 11: Estimates for $\|p - p_h\|_{\mathcal T_h}$ and $\|\mathbb L - \mathbb L_h\|_{\mathcal T_h}$.}

\begin{lemma}\label{pressure_p1}
	For $k\ge 1$, we have
	\begin{equation*}
	\|{\zeta_p}\|_\Omega  \lesssim \|\zeta_{\mathbb L}\|_{\mathcal T_h}+ h^{-\frac 1 2} \|\bm P_M \zeta_{\bm y} - \zeta_{\widehat{\bm y}}\|_{\partial\mathcal T_h\backslash\varepsilon_h^\partial}+ h^{-\frac 1 2} \|\bm P_M \zeta_{\bm y}\|_{\varepsilon_h^\partial}  + h^{-\frac 1 2} \| P_M  u -  u_h \|_{\varepsilon_h^\partial}.
	\end{equation*}
\end{lemma}

\begin{proof}
	As in the proof of \Cref{pressure_estimate_p_1}, we have
	\begin{align*}
	\| \zeta_p\|_\Omega \lesssim \sup_{\bm v\in {\bm H}_0^1 (\Omega)\backslash\{0\}} \frac{(\zeta_p,\nabla\cdot \bm v)_{\Omega}}{\norm{\bm v} _{\bm H ^1 (\Omega)}},
	\end{align*}
	and
	\begin{align*}
	(\zeta_p,\nabla\cdot \bm v)_\Omega =-(\nabla \zeta_p ,\bm \Pi_{ V} \bm v)_{\mathcal{T}_h}+\langle \zeta_p, \bm v\cdot \bm n\rangle_{\partial \mathcal{T}_h}.
	\end{align*}
	Use $(\mathbb T_1,\bm v_1, w_1,\bm \mu_1) = (0,\bm \Pi_{V} \bm v,0,0)$ in \eqref{eq_yh1} and \eqref{def_B2}, and $\bm v\in \bm H_0^1(\Omega)$ to obtain
	\begin{align*}
	(\nabla \zeta_p,\bm \Pi_{ V} \bm v)_{\mathcal{T}_h} &=  (\nabla\cdot\zeta_{\mathbb L}, \bm \Pi_{ V} \bm v)_{\mathcal{T}_h} - \langle h^{-1}(\bm P_M\zeta_{\bm y}-\zeta_{\widehat{\bm y}}),\bm\Pi_{V}\bm v\rangle_{\partial \mathcal{T}_h\backslash\varepsilon_h^\partial}\\
	&\quad - \langle h^{-1} \bm P_M\zeta_{\bm y},\bm\Pi_{V}\bm v\rangle_{\varepsilon_h^\partial} +  \langle  (P_M  u -  u_h)\bm \tau, h^{-1} \bm\Pi_{V}\bm v\rangle_{\varepsilon_h^\partial}.
	\end{align*}
	This gives
	\begin{align*}
	(\zeta_p ,\nabla\cdot \bm v)_\Omega&= - (\nabla\cdot\zeta_{\mathbb L}, \bm \Pi_{ V} \bm v)_{\mathcal{T}_h}+\langle h^{-1}(\bm P_M\zeta_{\bm y}-\zeta_{\widehat{\bm y}}),\bm\Pi_{V}\bm v\rangle_{\partial \mathcal{T}_h\backslash\varepsilon_h^\partial}+\langle\zeta_p, \bm v\cdot \bm n\rangle_{\partial \mathcal{T}_h} \\
	&\quad + \langle h^{-1} \bm P_M\zeta_{\bm y},\bm\Pi_{V}\bm v\rangle_{\varepsilon_h^\partial} - \langle (P_M  u -  u_h)\bm \tau, h^{-1} \bm\Pi_{V}\bm v\rangle_{\varepsilon_h^\partial}\\
	&= - (\nabla\cdot\zeta_{\mathbb L}, \bm v)_{\mathcal{T}_h}+\langle h^{-1}(\bm P_M\zeta_{\bm y}-\zeta_{\widehat{\bm y}}),\bm\Pi_{V}\bm v\rangle_{\partial \mathcal{T}_h\backslash\varepsilon_h^\partial}+\langle\zeta_p, \bm v\cdot \bm n\rangle_{\partial \mathcal{T}_h}\\
	&\quad + \langle h^{-1} \bm P_M\zeta_{\bm y},\bm\Pi_{V}\bm v\rangle_{\varepsilon_h^\partial} - \langle (P_M  u -  u_h)\bm \tau, h^{-1} \bm\Pi_{V}\bm v\rangle_{\varepsilon_h^\partial}\\
	&=  (\zeta_{\mathbb L}, \nabla \bm v)_{\mathcal{T}_h}+\langle h^{-1}(\bm P_M\zeta_{\bm y}-\zeta_{\widehat{\bm y}}),\bm\Pi_{V}\bm v\rangle_{\partial \mathcal{T}_h\backslash\varepsilon_h^\partial}+ \langle h^{-1} \bm P_M\zeta_{\bm y},\bm\Pi_{V}\bm v\rangle_{\varepsilon_h^\partial} \\
	&\quad +\langle  -\zeta_{\mathbb L} \bm n + \zeta_p\bm n, \bm P_M \bm v\rangle_{\partial \mathcal{T}_h\backslash\varepsilon_h^\partial} - \langle (P_M  u -  u_h)\bm \tau, h^{-1} \bm\Pi_{V}\bm v\rangle_{\varepsilon_h^\partial}.
	\end{align*}
	Next, take $(\mathbb T_1,\bm v_1, w_1,\bm \mu_1) = (0, 0,0,\bm P_M\bm v )$ in \eqref{eq_yh1}  and use $\bm v\in \bm H_0^1(\Omega)$ to give
	\begin{align*}
	\langle  \zeta_{\mathbb L} \bm n - \zeta_p\bm n, \bm P_M \bm v\rangle_{\partial \mathcal{T}_h\backslash\varepsilon_h^\partial} = \langle h^{-1}(\bm P_M\zeta_{\bm y}-\zeta_{\widehat{\bm y}}),\bm P_M\bm v\rangle_{\partial \mathcal{T}_h\backslash\varepsilon_h^\partial}.
	\end{align*}
	This implies
	\begin{align*}
	(\zeta_p ,\nabla\cdot \bm v)_\Omega  &=  (\zeta_{\mathbb L}, \nabla \bm v)_{\mathcal{T}_h}+\langle h^{-1}(\bm P_M\zeta_{\bm y}-\zeta_{\widehat{\bm y}}),\bm\Pi_{V}\bm v - \bm P_M \bm v \rangle_{\partial \mathcal{T}_h\backslash\varepsilon_h^\partial}\\
	&\quad + \langle h^{-1} \bm P_M\zeta_{\bm y},\bm\Pi_{V}\bm v\rangle_{\varepsilon_h^\partial} - \langle (P_M  u -  u_h)\bm \tau, h^{-1} \bm\Pi_{V}\bm v\rangle_{\varepsilon_h^\partial}\\
	&=  (\zeta_{\mathbb L}, \nabla \bm v)_{\mathcal{T}_h}+\langle h^{-1}(\bm P_M\zeta_{\bm y}-\zeta_{\widehat{\bm y}}),\bm\Pi_{V}\bm v - \bm P_M \bm v \rangle_{\partial \mathcal{T}_h\backslash\varepsilon_h^\partial}\\
	&\quad + h^{-1}\langle  \bm P_M\zeta_{\bm y},\bm\Pi_{V}\bm v - \bm v\rangle_{\varepsilon_h^\partial} - h^{-1}\langle (P_M  u -  u_h)\bm \tau, \bm\Pi_{V}\bm v - \bm v \rangle_{\varepsilon_h^\partial}.
	\end{align*}
	We obtain
	\begin{align*}
	|	(\zeta_p ,\nabla\cdot \bm v)_\Omega| &\lesssim \|\zeta_{\mathbb L}\|_{\mathcal T_h}\|\nabla\bm v\|_{\mathcal T_h} + h^{-\frac 1 2} \|\bm P_M \zeta_{\bm y} - \zeta_{\widehat{\bm y}}\|_{\partial\mathcal T_h\backslash\varepsilon_h^\partial} \|\nabla \bm v\|_{\mathcal T_h}\\
	&\quad + h^{-\frac 1 2} \|\bm P_M \zeta_{\bm y}\|_{\varepsilon_h^\partial}\|\nabla\bm v\|_{\mathcal T_h}  + h^{-\frac 1 2} \|P_M  u -  u_h \|_{\varepsilon_h^\partial}\|\nabla\bm v\|_{\mathcal T_h},
	\end{align*}
	and the result follows.
\end{proof}

	\begin{lemma}
		Let  $\mathcal M$ and $\mathcal N$ be defined as  in \Cref{error_y_lyp,lemma:step6_main_lemma}, respectively. If  $k\geq 1$ holds, then
		\begin{align*}
		\norm {\zeta_{\mathbb L}}_{\mathcal T_h} + \norm {\zeta_{p}}_{\mathcal T_h} &\lesssim \mathcal M + h^{-1}\mathcal N.
		\end{align*}
	\end{lemma}
	
	\begin{proof}
		By \Cref{energy_norm} and the error equation \eqref{eq_yh1}, we have
		\begin{align*}
		\hspace{2em}&\hspace{-2em}\mathscr B ( \zeta_{\mathbb L},\zeta_{\bm y},\zeta_{p}, \zeta_{\widehat{\bm y}}; \zeta_{\mathbb L},\zeta_{\bm y},\zeta_{p}, \zeta_{\widehat{\bm y}}) \\
		&=(\zeta_{\mathbb L}, \zeta_{\mathbb L})_{{\mathcal{T}_h}}+\langle (h^{-1} (\bm P_M \zeta_{\bm y}-\zeta_{\widehat {\bm y}}) , \zeta_{\bm y}-\zeta_{\widehat {\bm y}} \rangle_{\partial{{\mathcal{T}_h}}\backslash\varepsilon_h^\partial}+ \langle h^{-1}\bm P_M \zeta_{\bm y}, \bm P_M\zeta_{\bm y} \rangle_{\varepsilon_h^\partial}\\
		&= \langle (P_M  u -  u_h)\bm{\tau}, \zeta_{\mathbb L}\cdot \bm{n} + h^{-1} \zeta_{\bm y} \rangle_{{\varepsilon_h^{\partial}}}\\
		& =\langle   (u-  u_h)\bm \tau, \zeta_{\mathbb L}\cdot \bm{n} + h^{-1} \bm P_M \zeta_{\bm y} \rangle_{{\varepsilon_h^{\partial}}}\\
		&	\lesssim \norm { u- u_h}_{\varepsilon_h^{\partial}} (\norm {\zeta_{\mathbb L}}_{\varepsilon_h^{\partial}} + h^{-1} \norm {\bm P_M \zeta_{\bm y}}_{\varepsilon_h^{\partial}})\\
		&	\lesssim h^{-\frac 1 2}\norm { u- u_h}_{\varepsilon_h^{\partial}} (\norm {\zeta_{\mathbb L}}_{\mathcal T_h} + h^{-\frac 1 2} \norm {\bm P_M \zeta_{\bm y}}_{\varepsilon_h^{\partial}}),
		\end{align*}
		which gives
		\begin{align*}
		\hspace{4em}&\hspace{-4em}\norm {\zeta_{\mathbb L}}_{\mathcal T_h} + h^{-\frac 1 2} \|\bm P_M \zeta_{\bm y} - \zeta_{\widehat{\bm y}}\|_{\partial\mathcal T_h\backslash\varepsilon_h^\partial} + h^{-\frac 1 2} \|\bm P_M \zeta_{\bm y}\|_{\varepsilon_h^\partial} \lesssim \mathcal M + h^{-1}\mathcal N.
		\end{align*}
		This bound together with \Cref{pressure_p1} gives  the final result.
	\end{proof}

	The above lemma, the triangle inequality, and \Cref{lemma:step4_conv_rates} complete the proof of the following main result:
	\begin{theorem}
		Let  $\mathcal M$ and $\mathcal N$ be defined in \Cref{error_y_lyp,lemma:step6_main_lemma}, respectively. If  $k\geq 1$ holds, then
		\begin{align*}
		\|p-p_h\|_{\mathcal{T}_h} +\norm {\mathbb L - \mathbb L_h}_{\mathcal T_h} \lesssim  \mathcal M + h^{-1}\mathcal N.
		\end{align*}
	\end{theorem}

\section{Numerical Experiments}
\label{sec:numerics}



In order to emphasize the dependance of the order of convergence with respect to the regularity of the optimal control, we consider one example in a square domain and another one in a non-square domain.

\begin{example}\label{example2}
	In this example, we consider the domain $\Omega = [0,1/8]\times [0,1/8]$ and set $\bm f=\bm 0$,  $\gamma = 1$.  
	We take as target state the large vortex described in \cite{MR3683678}
	\begin{align*}
	\bm y_d = 200\times 8^3[x_1^2(1-8x_1)^2x_2(1-8x_2)(1-16x_2); - x_1(1-8x_1)(1-16x_1)x_2^2(1-8x_2)^2].
	\end{align*}
For illustration, we show in \Cref{fig:direc_field_1} a plot of $\bm y_d$ rescaled to $[0,1]\times [0,1]$  (left), and the computed optimal control on a relatively coarse mesh, i.e., $h = \sqrt{2}/64$ (right).
\begin{figure}
	\centerline{
		\hbox{\includegraphics[width=0.56\textwidth]{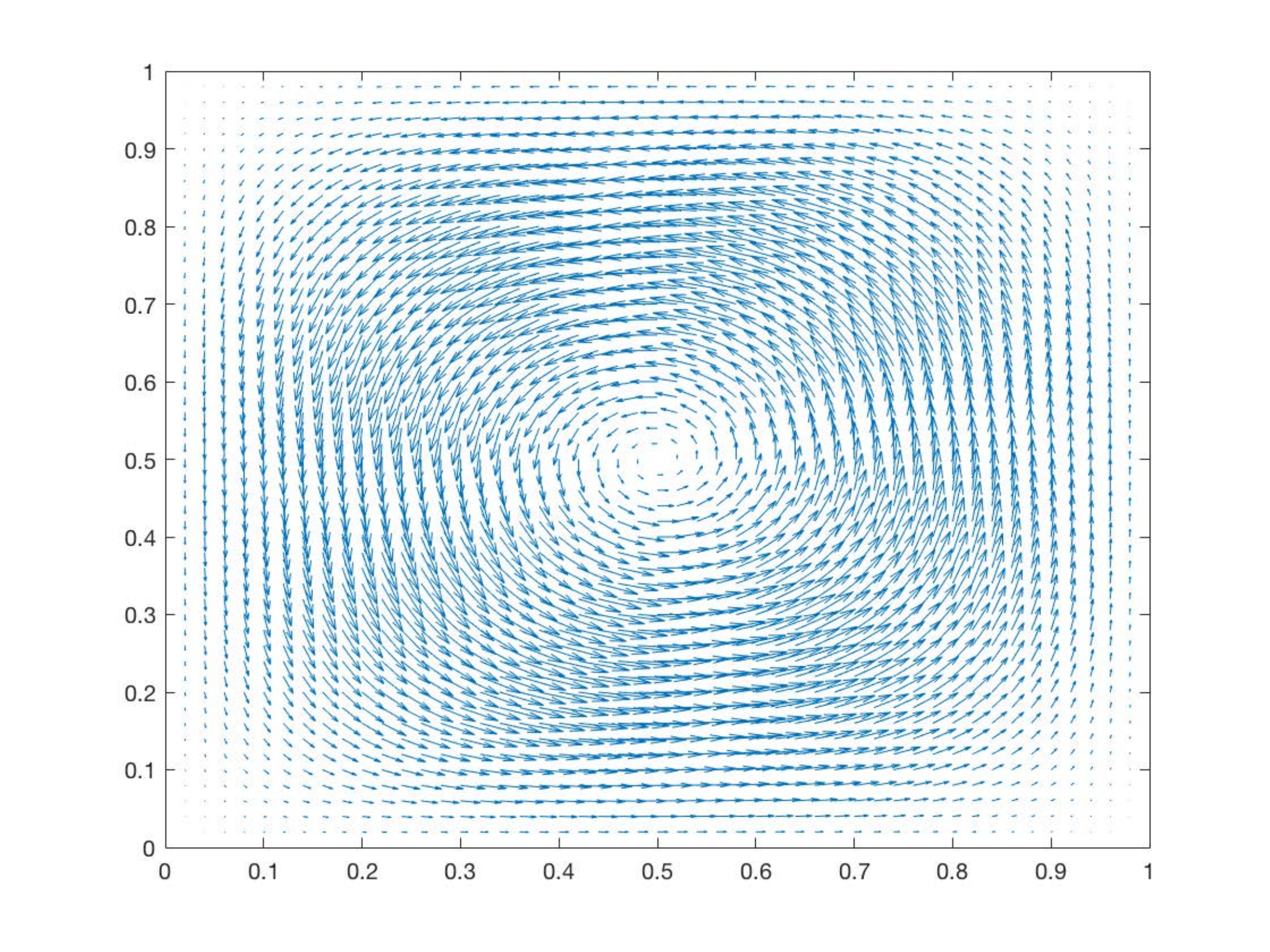}}
		\hbox{\includegraphics[width=0.44\textwidth]{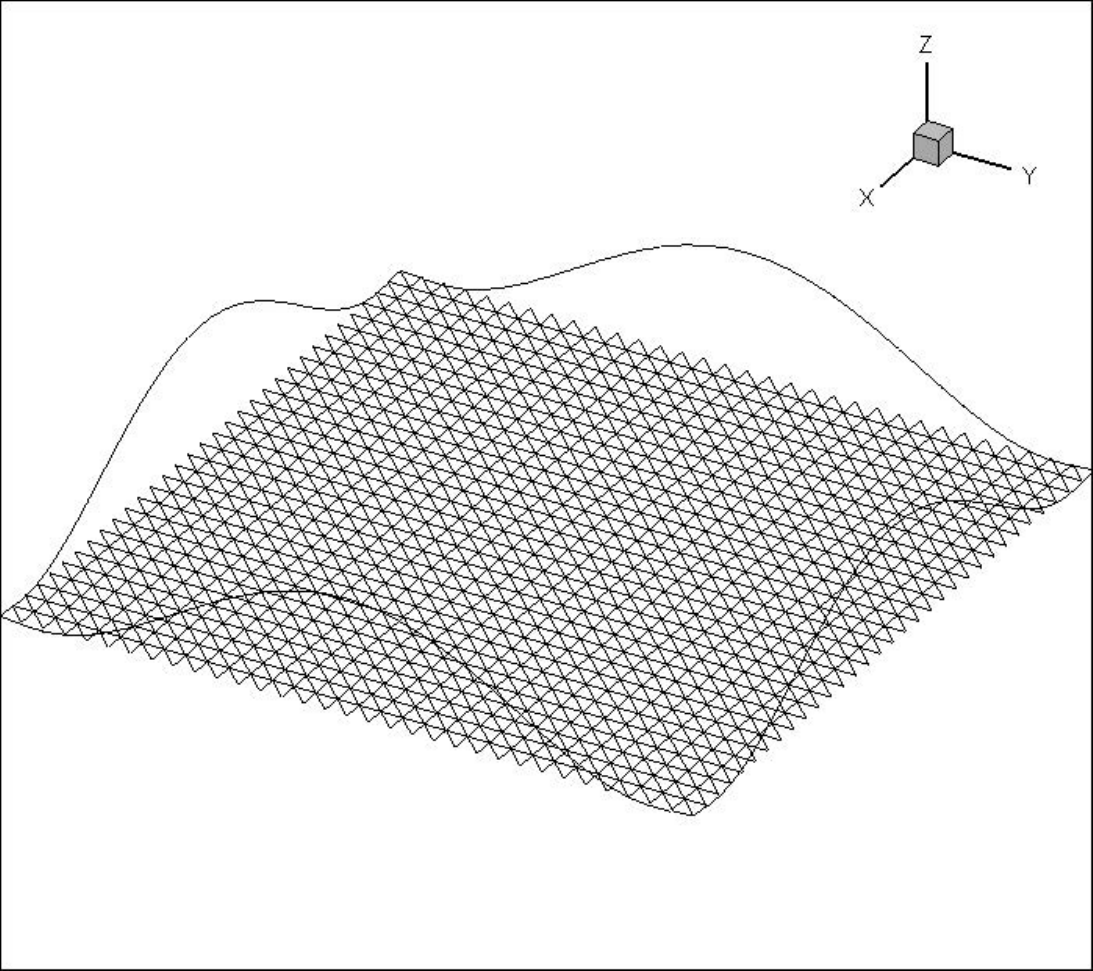}}
	}
	\caption{\Cref{example2}: Left is the  desired state $ \bm y_d $ and right is the Dirichlet boundary control $u$.}
	\label{fig:direc_field_1}
	\centering
\end{figure}

	
For a square domain, the singular exponent is $ \xi \approx 2.74 $ (cf. \cite{MR977489}). \Cref{cor:main_result} yields directly an order of convergence for the control variable of $0.5$ for $k=0$ and (almost) $1.5$ for $k=1$.

Nevertheless, for this example, we notice in \Cref{fig:direc_field_1} that not only is the optimal control equal to zero on the corners, but also the derivative of the control is zero.  If we assume the derivative of the optimal control is exactly zero at the corners, we can improve
  the global regularity $u\in H^{s}(\Gamma)$ for all $s<3/2$ given in \eqref{opt_control_global_reg} using the local regularity $ u \in H^t(\Gamma_i) $ for all $t<2.24$ given in \eqref{eqn:opt_control_local_reg}.
  Since the derivative of $u$ is zero at the corners, we have that $ u \in H^t(\Gamma) $ for all $t<2.24$.
  The same argument can be used to improve the global regularity of the Dirichlet data $u\bm \tau\in\bm V^s(\Gamma)$ given in \eqref{DirichletDataRegularity} to obtain $u\bm \tau\in\bm V^t(\Gamma)$ for all $t<2.24$.
   This leads to a higher regularity of the optimal state:
   although  \Cref{T2.1} is not directly applicable, since each component of $u\bm \tau$ has zero {\em derivative} at the corners, there is an extra compatibility condition satisfied and we can apply the trace theorem \cite[Theorem 1.5.2.8]{PGrisvard_book_1} to obtain $r_{\bm y}\approx 2.74$.
   For the other variables now we have $r_{\bm z}\approx 3.74$, $r_{\mathbb L} = r_p = r_{\mathbb G}-1=r_q-1 \approx 1.74$. With this regularity, a direct application of  \Cref{main_res} yields also an order of convergence for the control variable of $2.24$ for $k=2$.

The numerical results are shown in \Cref{table_3} for $ k = 0 $, \Cref{table_4} for $ k = 1 $ and  \Cref{table_5} for $ k = 2 $.
	Since we do not have an explicit expression for the exact solution, we solved the problem numerically for a triangulation with $524288$ elements, i.e., $h = \sqrt{2}/2^{12}$,  and compared this reference solution against other solutions computed on meshes with larger $h$.  	

	\begin{table}
		\begin{center}
			\begin{tabular}{cccccc|c}
				\hline
				${h}/{\sqrt 2}$ &1/16& 1/32&1/64 &1/128 & 1/256 &\textup{EO} \\
				\hline
				$\|\mathbb G - \mathbb G_h\|_{\mathcal T_h}$&1.66E-03   &8.76E-04   &4.62E-04   &2.38E-04   &1.20E-04
				\\
				order&-&  0.92   &0.92   &0.96   &0.99 & 0.5\\
				\hline
				$\|\bm y - \bm y_h\|_{\mathcal T_h}$ &6.20E-04   &3.27E-04   &1.74E-04   &6.83E-05   &2.04E-05\\
				order&-& 0.93   &0.91   &1.35   &1.75 & 0.5\\
				\hline
				$\|\bm z- \bm z_h\|_{\mathcal T_h}$&6.89E-05   &1.83E-05   &4.85E-06   &1.26E-06   &3.21E-07 \\
				order&-& 1.90& 1.92   &1.94  &1.98& 0.5\\
				\hline
				$\norm{{q}-{q}_h}_{\mathcal T_h}$&5.61E-04   &3.72E-04   &1.92E-04   &9.31E-05   &4.52E-05 \\
				order&-& 0.59   &0.96   &1.04   &1.04& 0.5\\
				\hline
				$\norm{{ u}-{ u}_h}_{\varepsilon_h^\partial}$&6.34E-03   &3.13E-03   &1.67E-03   &7.88E-04   &3.73E-04 \\
				order&-& 1.02  &0.910  &1.08   &1.08 &0.5\\
				\hline
			\end{tabular}
		\end{center}
		\caption{\Cref{example2}, $k=0$: Errors, observed convergence orders, and expected order (EO) for the control $u$, pressure $p$,  dual pressure $q$, state $\bm y$, adjoint state $\bm z$,  and their fluxes $\mathbb L$ and $\mathbb G$.}\label{table_3}
	\end{table}
	\begin{table}
		\begin{center}
			\begin{tabular}{cccccc|c}
				\hline
				${h}/{\sqrt 2}$ &1/16& 1/32&1/64 &1/128 & 1/256    &\textup{EO} \\
				\hline
				$\|\mathbb L - \mathbb L_h\|_{\mathcal T_h}$&3.38E-02   &1.61E-02   &6.51E-03   &2.23E-03   &7.35E-04\\
				order&-&1.07   &1.31   &1.55   &1.60&1\\
				\hline
				$\|\mathbb G - \mathbb G_h\|_{\mathcal T_h}$&6.65E-04   &2.02E-04   &5.90E-05   &1.58E-05   &4.06E-06
				\\
				order&-&  1.72   &1.77   &1.90   &1.96 &1.5 \\
				\hline
				$\|\bm y - \bm y_h\|_{\mathcal T_h}$ &2.74E-04   &8.39E-05   &1.72E-05   &2.67E-06   &4.16E-07\\
				order&-& 1.71   &2.28   &2.69   &2.68   &1.5\\
				\hline
				$\|\bm z- \bm z_h\|_{\mathcal T_h}$&1.82E-05   &2.86E-06   &3.78E-07   &4.80E-08   &6.03E-09 \\
				order&-& 2.67   &2.92   &2.98   &2.99 &1.5\\
				\hline
				$\norm{{p}-{p}_h}_{\mathcal T_h}$&2.86E-02   &1.20E-02   &4.40E-03   &1.28E-03   &3.69E-04 \\
				order&-& 1.25   &1.45   &1.78   &1.79  &1\\
				\hline
				$\norm{{q}-{q}_h}_{\mathcal T_h}$&2.74E-04   &8.70E-05   &2.54E-05   &6.67E-06   &1.67E-06 \\
				order&-&  1.66   &1.77   &1.93   &2.00  &1.5\\
				\hline
				$\norm{{ u}-{ u}_h}_{\varepsilon_h^\partial}$& 2.13E-03  & 8.60E-04   &2.54E-04   &7.02E-05   &1.90E-05\\
				order&-& 1.31   &1.76   &1.86   &1.89  &1.5\\
				\hline
			\end{tabular}
		\end{center}
		\caption{\Cref{example2}, $k=1$: Errors, observed convergence orders, and expected order (EO) for the control $u$, pressure $p$,  dual pressure $q$, state $\bm y$, adjoint state $\bm z$,  and their fluxes $\mathbb L$ and $\mathbb G$.}\label{table_4}
	\end{table}
	\begin{table}
		\begin{center}
			\begin{tabular}{cccccc|c}
				\hline
				${h}/{\sqrt 2}$ &1/16& 1/32&1/64 &1/128 & 1/256    &\textup{EO}\\
				\hline
				$\|\mathbb L - \mathbb L_h\|_{\mathcal T_h}$& 1.87E-02   &6.42E-03   &1.82E-03   &5.08E-04   &1.53E-04\\
				order&-&  1.54  &1.81   &1.84   &1.72& 1.74\\
				\hline
				$\|\mathbb G - \mathbb G_h\|_{\mathcal T_h}$&1.96E-04   &4.76E-05   &8.02E-06   &1.18E-06   &1.69E-07
				\\
				order&-& 2.04   &2.57   &2.76  &2.80 &2.24\\
				\hline
				$\|\bm y - \bm y_h\|_{\mathcal T_h}$ &1.07E-04   &1.79E-05   &2.56E-06   &3.38E-07   &5.26E-08\\
				order&-&  2.58   &2.80   &2.92  &2.68   &2.24\\
				\hline
				$\|\bm z- \bm z_h\|_{\mathcal T_h}$&6.93E-06   &5.63E-07   &3.83E-08   &2.47E-09   &1.57E-10 \\
				order&-&  3.62&  3.87&  3.95&   3.97&2.24\\
				\hline
				$\norm{{p}-{p}_h}_{\mathcal T_h}$&1.49E-02   &4.80E-03   &1.29E-03   &3.58E-04   &1.14E-04\\
				order&-& 1.64&  1.88&   1.85&  1.64  &1.74\\
				\hline
				$\norm{{q}-{q}_h}_{\mathcal T_h}$&9.49E-05   &2.32E-05   &4.36E-06   &6.39E-07   &9.30E-08 \\
				order&-&    2.02&   2.41&  2.77&  2.78&  2.24\\
				\hline
				$\norm{{ u}-{ u}_h}_{\varepsilon_h^\partial}$& 1.05E-03   &2.63E-04   &5.57E-05   &1.19E-05   &2.55E-06\\
				order&-&   2.00&   2.23&  2.23  &2.22  &2.24\\
				\hline
			\end{tabular}
		\end{center}
		\caption{\Cref{example2}, $k=2$: Errors, observed convergence orders, and expected order (EO) for the control $u$, pressure $p$,  dual pressure $q$, state $\bm y$, adjoint state $\bm z$,  and their fluxes $\mathbb L$ and $\mathbb G$.}\label{table_5}
	\end{table}
\end{example}

\begin{example}\label{example3}
	In this example, we choose the same data as in \Cref{example2}, but the domain is the convex hull of the points $\{(0,0), (\sqrt{3}/8,0), (\sqrt{3}/4,1/8), (0,1/8)\}$; see \Cref{fig:direc_field_2}. Since the largest angle is $5\pi/6$, then \eqref{singular_ex} yields $\xi\approx 1.53$. \Cref{cor:main_result} yields an order of convergence $0.5$ for $k=0$ and approximately $1.03$ for $k=1$. In this case, we cannot improve the global regularity  $u\in H^s(\Gamma)$ for all $s<1.03$. For illustration, we plot the computed optimal control in \Cref{fig:direc_field_2}.
\begin{figure}
	\centerline{
		\hbox{\includegraphics[width=5in]{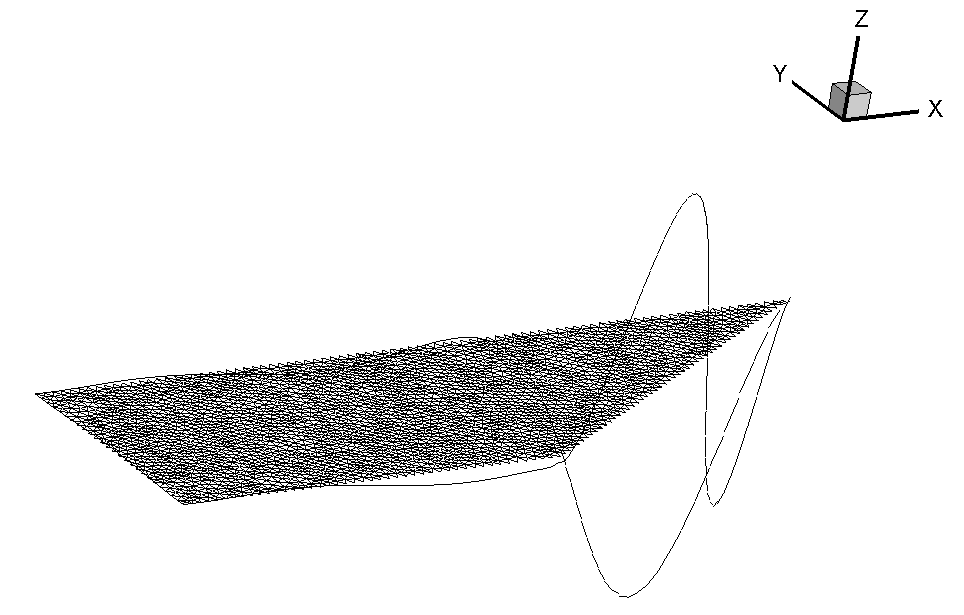}}
	}
	\caption{\Cref{example3}: The computed Dirichlet boundary control $u$.}
	\label{fig:direc_field_2}
	\centering
\end{figure}

The numerical results are shown in \Cref{table_6} for $ k = 0 $ and  \Cref{table_7} for $ k = 1$.
	Since we do not have an explicit expression for the exact solution, we solved the problem numerically for a triangulation with $327680$ elements, i.e., $h =  5.9146\times 10^{-4}$,  and compared this reference solution against other solutions computed on meshes with larger $h$.  As in the previous example, the experimental orders of convergence are higher than the expected orders.

	\begin{table}
		\begin{center}
			\begin{tabular}{cccccc|c}
				\hline
				${h}/{\sqrt 2}$ &1/16& 1/32&1/64 &1/128 & 1/256 &\textup{EO} \\
				\hline
				$\|\mathbb G - \mathbb G_h\|_{\mathcal T_h}$&3.76E-01   &2.97E-01   &2.31E-01   &1.51E-01   &8.53E-02
				\\
				order&-& 0.34   &0.35   &0.61   &0.82 & 0.5\\
				\hline
				$\|\bm y - \bm y_h\|_{\mathcal T_h}$ & 3.71E-01   &2.49E-01   &1.39E-01   &6.01E-02   &2.16E-02\\
				order&-& 0.57   &0.83   &1.21 &1.47& 0.5\\
				\hline
				$\|\bm z- \bm z_h\|_{\mathcal T_h}$& 2.94E-02   &1.00E-02   &3.85E-03   &1.35E-03   &4.01E-04 \\
				order&-&  1.55  &1.37   &1.51   &1.75& 0.5\\
				\hline
				$\norm{{q}-{q}_h}_{\mathcal T_h}$&
				1.44   &1.07   &7.48E-01   &3.86E-01   &1.52E-01 \\
				order&-& 0.41&   0.52   &0.95   &1.34& 0.5\\
				\hline
				$\norm{{ u}-{ u}_h}_{\varepsilon_h^\partial}$&3.12 &2.57   &1.84   &1.01   &4.18E-01 \\
				order&-& 0.27   &0.48  &0.86   &1.27 &0.5\\
				\hline
			\end{tabular}
		\end{center}
		\caption{\Cref{example3}, $k=0$: Errors, observed convergence orders, and expected order (EO) for the control $u$, pressure $p$,  dual pressure $q$, state $\bm y$, adjoint state $\bm z$,  and their fluxes $\mathbb L$ and $\mathbb G$.}\label{table_6}
	\end{table}
	\begin{table}
		\begin{center}
			\begin{tabular}{cccccc|c}
				\hline
				${h}/{\sqrt 2}$ &1/16& 1/32&1/64 &1/128 & 1/256    &\textup{EO} \\
				\hline
				$\|\mathbb L - \mathbb L_h\|_{\mathcal T_h}$&5.67 &  2.99  &1.67   &9.14E-01   &4.12E-01\\
				order&-&0.92   &0.84  &0.86&  1.14&0.53\\
				\hline
				$\|\mathbb G - \mathbb G_h\|_{\mathcal T_h}$&1.02E-01   &5.26E-02   &2.40E-02   &8.46E-03   &2.37E-03
				\\
				order&-&  0.95 &1.13   &1.50   &1.83 &1.03 \\
				\hline
				$\|\bm y - \bm y_h\|_{\mathcal T_h}$ &6.60E-02   &1.71E-02   &4.48E-03   &1.31E-03   &3.09E-04\\
				order&-& 1.94  &1.93   &1.76   &2.09   &1.03\\
				\hline
				$\|\bm z- \bm z_h\|_{\mathcal T_h}$&7.56E-03   &1.28E-03   &2.15E-04   &3.42E-05   &4.67E-06 \\
				order&-& 2.55&   2.57&   2.65&   2.87 &1.03\\
				\hline
				$\norm{{p}-{p}_h}_{\mathcal T_h}$& 1.82E+01   &3.97E+00   &1.75E+0   &9.77E-01   &3.99E-01 \\
				order&-& 2.20   &1.18  &0.84   &1.28  &0.53\\
				\hline
				$\norm{{q}-{q}_h}_{\mathcal T_h}$& 2.58E-01   &5.62E-02   &1.60E-02   &4.83E-03   &1.29E-03 \\
				order&-& 2.19   &1.80   &1.73   &1.89 &1.03\\
				\hline
				$\norm{{ u}-{ u}_h}_{\varepsilon_h^\partial}$& 7.12E-01   &2.55E-01   &1.12E-01   &4.86E-02   &1.58E-02\\
				order&-& 1.48  &1.17   &1.21   &1.61&1.03\\
				\hline
			\end{tabular}
		\end{center}
		\caption{\Cref{example3}, $k=1$: Errors, observed convergence orders, and expected order (EO) for the control $u$, pressure $p$,  dual pressure $q$, state $\bm y$, adjoint state $\bm z$,  and their fluxes $\mathbb L$ and $\mathbb G$.}\label{table_7}
	\end{table}
\end{example}

As noticed in \Cref{Remark4.1}, experimental orders of convergence for the control variable are about 0.5 higher than predicted by  \Cref{main_res}. The obtention of higher experimental orders of convergence than the ones predicted by the theory in the control variable is common in numerical experiments for Dirichlet control problems; see e.g. \cite{MR2272157,MR2837508,MR3070527,MR3465458,MR3641789,MR3456959}. In the case of Poisson equation and using standard Lagrange $\mathcal P_1$ elements, it has taken several years to get optimal error estimates; see
\cite{MR2558321} for smooth domains or \cite{ApelMateosPfeffererRosch17,Apel2019} for polygonal domains. One of the key points for this improvement is the use of weighted $W^{k,\infty}(\Omega)$ norms for the adjoint state and not-standard duality arguments to obtain estimates for the variational normal derivative; see also \cite{MR4000217,MR3319574,Winkler_29}.
We are not aware of any similar study for HDG discretization.

For the state and the adjoint state, the experimental orders are much higher, as often happens in numerical experiments for control problems; see \cite{MR2114385} for an approach to this problem using superconvergence properties of the optimal controls at the barycenters of the elements or \cite{MR3070527} for a duality approach, using improved error convergence rates in Sobolev norms of negative exponent.

\section{Conclusion}

In this work, we considered a tangential Dirichlet boundary control problem for the Stokes equations. First, we established well-posedness and regularity results for the optimal control problem based on a weak mixed formulation of the PDE on polygonal domains. Next, we used an existing superconvergent HDG method to approximate the solution of the optimality system and established optimal convergence rates for the control under certain assumptions on the domain $\Omega$ and the desired state $\bm y_d$. However, the numerical experiments show higher convergence rates than our theoretical results; this may be due to the higher local regularity of the control on individual edges of the domain.  This phenomenon is not present in Dirichlet boundary control problems for the Poisson equation.

As far as we are aware, this is the first work to explore the analysis of this tangential Dirichlet control problem of Stokes equations and the numerical analysis of a computational method for this problem. There are a number of topics that can be explored in the future, including using standard conforming finite elements for this problem, using an energy space for the control (see \cite{MR3317816,MR3614013,MR3824871} for the Poisson equation), devising divergence free and pressure robust HDG schemes, and considering more complicated PDEs, such as the Oseen and Navier-Stokes equations.




\end{document}